\documentclass[12pt]{article}
\usepackage{amssymb,amsmath,amsthm,amsxtra,amsfonts}
\usepackage{geometry}
\usepackage{aliascnt}
\usepackage{color}
\usepackage[usenames,dvipsnames,table]{xcolor}
\usepackage{subcaption}
\usepackage{bbm}
\usepackage{dsfont}
\usepackage{graphicx}
\usepackage{mathrsfs}
\usepackage{tikz}
\usetikzlibrary{plotmarks,arrows,calc,patterns}
\usepackage[longnamesfirst,round]{natbib}
\usepackage[hypertexnames=false,colorlinks=true,citecolor=Blue,urlcolor=Blue,linkcolor=Blue]{hyperref}
\usepackage{algorithm,algpseudocode}
\usepackage{threeparttable}
\usepackage{booktabs}
\usepackage{verbdef}
\verbdef{\vtext}{rdrobust}
\usepackage{rotating}
\usepackage{multirow,tabularx}
\usepackage{doi}

\newtheorem{theorem}{Theorem}[section]
\newtheorem{proposition}{Proposition}[section]
\newtheorem{lemma}{Lemma}[section]

\newtheorem{remark}{Remark}[section]
\newtheorem{assumption}{Assumption}[section]
\newtheorem*{assumption*}{Assumption}

\newcommand{\R}{\ensuremath{\mathbf{R}}}

\newcommand{\E}{\ensuremath{\mathbb{E}}}

\newcommand{\dx}{\ensuremath{\mathrm{d}}}
\newcommand{\Var}{\ensuremath{\mathrm{Var}}}
\newcommand{\Cov}{\ensuremath{\mathrm{Cov}}}
\newcommand{\cw}{\ensuremath{\xrightarrow{\mathrm{d}}}}

\newcommand{\si}{\perp \! \! \! \perp}

\DeclareMathOperator*{\argmax}{arg\,max}

\newcommand{\one}{\ensuremath{\mathds{1}}}

\begin{document}

\title{\bf Isotonic Regression Discontinuity Designs}
\author{
	Andrii Babii\footnote{Department of Economics, University of North Carolina at Chapel Hill - Gardner Hall, CB 3305 Chapel Hill, NC
		27599-3305. Corresponding author. Email: \href{mailto:babii.andrii@gmail.com}{babii.andrii@gmail.com}.} \\
	\textit{\normalsize UNC Chapel Hill}  \and Rohit Kumar\footnote{Indian Statistical Institute, Delhi, 7 S.J.S. Sansanwal Marg, ISI Delhi Centre, New Delhi, Delhi 110016.} \\ \textit{\normalsize ISI, Delhi}
}
\maketitle

\begin{abstract}
	This paper studies the estimation and inference for the isotonic regression at the boundary point, an object that is particularly interesting and required in the analysis of monotone regression discontinuity designs. We show that the isotonic regression is inconsistent in this setting and derive the asymptotic distributions of boundary corrected estimators. Interestingly, the boundary corrected estimators can be bootstrapped without subsampling or additional nonparametric smoothing which is not the case for the interior point. The Monte Carlo experiments indicate that shape restrictions can improve dramatically the finite-sample performance of unrestricted estimators. Lastly, we apply the isotonic regression discontinuity designs to estimate the causal effect of incumbency in the U.S. House elections.
\end{abstract}


\begin{jel}
	\; C14, C31.
\end{jel}

\section{Introduction}
Regression discontinuity designs, see \cite{thistlethwaite1960regression}, are widely recognized as one of the most credible quasi-experimental strategies for identifying and estimating causal effects. In a nutshell, such designs exploit the discontinuity in the treatment assignment probability around the cutoff value of some running variable. The discontinuous treatment assignment probability frequently occurs due to laws and regulations governing economic and political life. A comprehensive list of empirical applications using regression discontinuity designs can be found in \cite{lee2010regression}; see also \cite{imbens2008regression}, and \cite{cattaneo2019practical} for the methodological review and \cite{imbens2009recent}, \cite{abadie2018econometric} for their place in the program evaluation literature. On the methodological side, in the seminal paper \cite{hahn2001identification} translate regression discontinuity designs into the potential outcomes framework and introduce the local polynomial nonparametric estimator to estimate causal effects in sharp and fuzzy designs.

Regression discontinuity designs encountered in the empirical practice are frequently monotone. Indeed, development and educational programs are often prescribed based on poverty or achievement scores that are monotonically related to average outcomes. For instance, when evaluating the effect of subsidies for fertilizers on yields, the yield per acre is expected to be non-decreasing in the size of the farm due to the increasing returns to scale. Alternatively, when evaluating the effectiveness of the cash transfers program on the households food expenditures, we expect that more affluent households spend, on average, more on food, since food is a normal good; see Section \ref{sec:appendix_literature} in the Supplementary Material for a sample of other examples drawn from the empirical research.

Despite this prevalence in the empirical practice, little is known about how monotonicity can be incorporated in the estimation procedure.\footnote{Shape restrictions in regression discontinuity designs is a relatively unexplored area. To the best of our knowledge, \cite{armstrong2015adaptive} is the only existing work focusing on the k-nearest neighbors estimator.}  Estimation and inference in regression discontinuity designs require estimating the conditional mean functions at the boundary of the support. In this paper, we focus on the isotonic regression estimator and establish its formal statistical properties for the boundary point.\footnote{This estimator is relatively unknown in econometrics and to the best of our knowledge has not been previously considered in the RDD setting; see an excellent review paper \cite{chetverikov2018econometrics} for a comprehensive discussion of this estimator and further references on the econometrics of shape restrictions.} The isotonic regression estimator originates from the work of \cite{ayer1955empirical}, and \cite{brunk1956inequality}. \cite{brunk1970estimation} derives its asymptotic distribution at the \textit{interior point} under restrictive assumptions that the regressor is deterministic and regression errors are homoskedastic. His treatment builds upon the ideas of \cite{rao1969estimation}, who derived the asymptotic distribution of the monotone density estimator, also known as the Grenander estimator; see \cite{grenander1956theory}. \cite{wright1981asymptotic} provides the final characterization of the large sample distribution for the interior point when the regressor is random and regression errors are heteroskedastic.

However, to the best of our knowledge, little is known about the behavior of the isotonic regression at the \textit{boundary point}, which is a building block of our isotonic regression discontinuity (iRDD) estimators. This situation contrasts strikingly with the local polynomial estimator, whose boundary behavior is well-understood; see \cite{fan1992variable}. Most of the existing results for isotonic estimators at the boundary are available only for the Grenander estimator; see \cite{woodroofe1993penalized}, \cite{kulikov2006behavior}, and \cite{balabdaoui2011grenander}. More precisely, we know that the Grenander estimator is inconsistent at the boundary of the support and that the consistent estimator can be obtained with additional boundary correction or penalization. At the same time, some isotonic estimators, e.g., in the current status model, are consistent at the boundary without corrections; see \cite{durot2018limit}. \cite{anevski2002monotone} discuss the inconsistency of the isotonic regression at the discontinuity point with deterministic equally spaced covariate and homoskedasticity. However, \cite{anevski2002monotone} do not discuss whether the isotonic regression with random covariate is consistent at the boundary of its support and do not provide a consistent estimator even in the restrictive equally spaced fixed design case.

In this paper, we aim to understand the behavior of the isotonic regression with a random regressor at the boundary of its support. We show formally that the isotonic regression estimator is inconsistent and tends to underestimate the value of the regression at the left boundary. The inconsistency is related to the extreme-value behavior of the closest to the boundary observation. We introduce boundary-corrected estimators and derive large sample approximations to corresponding distributions. The major technical difficulty when deriving asymptotic distributions in this setting is to establish the tightness of the maximizer of a certain empirical process. This condition is typically needed to apply the argmax continuous mapping theorem of \cite{kim1990cube}. The difficulty stems from the fact that conventional tightness results of \cite{kim1990cube} and \cite{van2000weak} do not always apply. For the Grenander estimator, \cite{kulikov2006behavior} suggest a solution to this problem based on the Koml\'{o}s-Major-Tusn\'{a}dy strong approximation. In our setting, this approach entails the strong approximation to the general empirical process, see \cite{koltchinskii1994komlos} and \cite{chernozhukov2015constrained}, which is more problematic to apply due to slower convergence rates. Consequently, we provide the alternative generic proof which does not rely on the strong approximation and which might be applied to other boundary-corrected shape-constrained estimators.

Since the asymptotic distribution is not pivotal, we introduce a novel trimmed wild bootstrap procedure and establish its consistency. The procedure consists of trimming values of the estimated regression function that are very close to the boundary when simulating wild bootstrap samples. Somewhat unexpectedly, we discover that the  trimming and the appropriate boundary correction restores the consistency of the wild bootstrap without additional nonparametric smoothing or subsampling, which is typically needed in such settings.\footnote{In contrast, the bootstrap typically fails at the interior point; see \cite{kosorok2008bootstrapping}, \cite{sen2010inconsistency}, \cite{guntuboyina2018nonparametric}, and \cite{patra2018consistent} for the discussion of this problem and various case-specific remedies, and \cite{cattaneo2017bootstrap} for generic solutions that apply to a class of cube-root consistent estimators.}

The paper is organized as follows. In Section~\ref{sec:isotonic_regression}, we develop the asymptotic theory for the isotonic regression estimator at the boundary point. The reader interested more in the regression discontinuity designs can skip this section and go directly to Section~\ref{sec:irdd}, where we discuss isotonic sharp and fuzzy designs. In Section~\ref{sec:mc_experiments}, we study the finite sample performance of the iRDD estimator with Monte Carlo experiments. Section~\ref{sec:empirical_application} estimates the effect of incumbency using the sharp iRDD on the data of \cite{lee2008randomized}. Section~\ref{sec:conclusion} concludes. Additional results, proofs, and Monte Carlo experiments appear in the Appendix and the Supplementary Material.

\section{Isotonic regression at the boundary}\label{sec:isotonic_regression}
In this section, we study the isotonic regression estimator at the boundary point. Readers interested in regression discontinuity designs only may skip this section and go directly to Section~\ref{sec:irdd}.

\subsection{Estimator}
We focus on the generic nonparametric regression
\begin{equation*}
Y = m(X)+\varepsilon,\qquad \E[\varepsilon|X]=0,
\end{equation*}
where the conditional mean function $m(x)=\E[Y|X=x]$ is assumed to be montone. For simplicity of presentation, we normalize the support to $[0,1]$ and assume that $m$ belongs to the set of non-decreasing functions on $[0,1]$
\begin{equation*}
	\mathcal{M}[0,1] = \left\{\phi:[0,1]\to\R:\; \phi(x_1)\leq \phi(x_2),\; \forall x_1\leq x_2 \right\}.
\end{equation*}

For a sample $(Y_i,X_i)_{i=1}^n$, the isotonic regression estimator, denoted $\hat m$, solves the constrained least-squares problem
\begin{equation*}
	\min_{\phi\in\mathcal{M}[0,1]}\sum_{i=1}^n(Y_{i}-\phi(X_{i}))^2;
\end{equation*}
see \cite{brunk1970estimation}. Note that the estimator is uniquely determined at data points and is conventionally interpolated as a piecewise constant function elsewhere. Let $X_{(1)}<X_{(2)}<\dots <X_{(n)}$ be the ordered values of the covariate and let $(Y_{(1)},Y_{(2)}\dots,Y_{(n)})$ be the corresponding observations of the outcome variable. It is easy to see that the isotonic regression is alternatively characterized as a solution to
\begin{equation*}
	\min_{\phi_1\leq \phi_2\leq\dots\leq\phi_n}\sum_{i=1}^n(Y_{(i)} - \phi_i)^2.
\end{equation*}
While, the isotonic regression is a solution to the convex optimization problem and can be obtained with standard constrained convex optimization packages, an efficient way to compute the estimator is via the pool adjacent violators algorithm; see \cite{ayer1955empirical}. Although the isotonic regression features the number of estimated parameters of the same magnitude as the sample size, the pool adjacent violators algorithm is typically computationally cheaper than nonparametric kernel estimators, and its computational complexity is closer to that of the OLS estimator.

\subsection{Large sample distribution}
We are interested in estimating the value of the regression function at the boundaries of its support $[0,1]$. More precisely, we focus on the regression at the left boundary, denoted $m(0)=\lim_{x\downarrow0}m(x)$. The natural estimator of $m(0)$ is $\hat m(X_{(1)})$. Unfortunately, it follows from Theorem~\ref{thm:inconsistency} in the appendix that this estimator is inconsistent. The inconsistency occurs because $X_{(1)}$ converges to zero too fast according to laws of the extreme value theory. More precisely, the estimator $\hat m(X_{(1)})$ underestimates $m(0)$ and $\hat m(X_{(n)})$ overestimates $m(1)$ in small samples; see Section~\ref{sec:appendix_inconsistency} in the Appendix for more details.\footnote{We are grateful to Matt Masten for this observation. More generally, in the isotonic regression discontinuity design, the isotonic regression estimator underestimates the causal effect at the cutoff, which can be used as a tuning-free lower bound on the estimated causal effect.}

To estimate $m(0)$ consistently, we the consider boundary-corrected estimators $\hat m(cn^{-a})$ for some $c,a>0$ and study the large sample behavior of this estimator for a range of values $c,a>0$. The following assumption imposes a mild restriction on the distribution of the data.

\begin{assumption}\label{as:dgp}
	$(Y_i,X_i)_{i=1}^n$ is an i.i.d. sample of $(Y,X)$ such that (i) $\E[|\varepsilon|^{2+\delta}|X]\leq C<\infty$ for some $\delta>0$ and $m$ is uniformly bounded; (ii) the distribution of $X$ has a Lebesgue density $f$, uniformly bounded away from zero and infinity, and $f(0)=\lim_{x\downarrow 0}f(x)$ exists; (iii) the conditional variance $\sigma^2(x)=\Var(Y|X=x)$ is uniformly bounded and $\sigma^2(0)=\lim_{x\downarrow 0}\sigma^2(x)$ exists; (iv) $m$ is continuously differentiable in the neighborhood of zero with $m'(0)=\lim_{x\downarrow0}m'(x)>0$; (v) $m\in\mathcal{M}[0,1]$.
\end{assumption}

For a stochastic process $\{Z_t: t\in A\subset\R\}$, let $D^L_{A}(Z_t)(s)$ denote the left derivative of the greatest convex minorant at $s\in A$.\footnote{The greatest convex minorant of a function $g:A\to\R$ is defined as the maximal convex function $h$ such that $h(x)\leq g(x),\forall x\in A$.} We say that $m$ is $\gamma$-H\"{o}lder continuous in the neighborhood of zero if there exists a constant $C<\infty$ such that for all $x>0$ in this neighborhood
\begin{equation*}
|m(x) - m(0)|\leq C|x|^\gamma.
\end{equation*}
Our first result describes the large sample behavior of boundary corrected isotonic regression estimators for a range of boundary corrections.

\begin{theorem}\label{thm:isotonic_clt}
	Suppose that Assumption~\ref{as:dgp} is satisfied and that $c>0$. Then
	\begin{enumerate}
		\item[(i)] for $a\in(0,1/3)$
		\begin{equation*}
		n^{1/3}\left(\hat m\left(cn^{-a}\right) - m(cn^{-a})\right) \cw\left|\frac{4m'(0)\sigma^2(0)}{f(0)}\right|^{1/3}\argmax_{t\in\R}\{W_t-t^2\}.
		\end{equation*}
		\item[(ii)] for $a\in[1/3,1)$
		\begin{equation*}
		n^{(1-a)/2}\left(\hat m\left(cn^{-a}\right) - m(0)\right) \cw D^L_{[0,\infty)}\left(\sqrt{\frac{\sigma^2(0)}{cf(0)}}W_t + \frac{t^2c}{2}m'(0)\one_{a=1/3} \right)(1),
		\end{equation*}
		where $(W_t)_{t\in\R}$ is a two-sided Brownian motion,\footnote{The two-sided Brownian motion is defined as two independent Brownian motions originating from zero and moving in the opposite directions.} and for $a\in(1/3,1)$, we can replace Assumption~\ref{as:dgp} (iv) by the $\gamma$-H\"{o}der continuity with $\gamma>(1-a)/2a$.
	\end{enumerate}
\end{theorem}
The proof of this result can be found in the appendix. The most challenging part of the proof is establishing tightness when $a\in(1/3,1)$. The difficulty comes from the fact that the quadratic term vanishes in this asymptotic regime  and the standard tightness results for isotonic estimators do not apply.\footnote{See \cite{kim1990cube} and \cite{van2000weak}, Theorem 3.2.5.} For the Grenander estimator, \cite{kulikov2006behavior} suggest a solution to this problem based on the Koml\'{o}s-Major-Tusn\'{a}dy strong approximation. In our case, we would need to rely on the strong approximation to the generic empirical process, see \cite{koltchinskii1994komlos}, which leads to suboptimal results due to slower convergence rates and uniform boundedness restrictions. Our proof does not rely on the strong approximation and is based on a more standard partitioning argument.

Theorem~\ref{thm:isotonic_clt} shows that for "slow" boundary corrections with $a\in(0,1/3)$, the asymptotic distribution coincides with the one at the interior point, cf. \cite{wright1981asymptotic}. However, since the estimator is centered at $m(cn^{-a})$, it is inconsistent for $m(0)$ because the "bias" $n^{1/3}(m(cn^{-a}) - m(0))$ does not vanish for all $a\in(0,1/3)$.

The choice $a=1/3$ leads to the cube-root consistent estimator similar to the tuning-free isotonic regression estimator at the interior point
\begin{equation*}
	n^{1/3}(\hat m(cn^{-1/3}) - m(0)) \cw D^L_{[0,\infty)}\left(\sqrt{\frac{\sigma^2(0)}{cf(0)}}W_t + \frac{t^2c}{2}m'(0) \right)(1).
\end{equation*}
For the interior point $x\in(0,1)$, \cite{wright1981asymptotic} shows that
\begin{equation*}
	n^{1/3}(\hat m(x) - m(x)) \cw D^L_{(-\infty,\infty)}\left(\sqrt{\frac{\sigma^2(x)}{f(x)}}W_t + \frac{t^2}{2}m'(x) \right)(1).
\end{equation*}
This suggests that the boundary correction with $a=1/3$ and $c=1$ should behave similarly to the tuning-free isotonic regression estimator at the interior point. Alternatively, one can try to improve the point estimates with a data-driven MSE-optimal choice which we investigate in Section~\ref{sec:mc_experiments}.

For the "fast" boundary corrections with $a\in(1/3,1)$, the convergence rate is slower than $O_P(n^{-1/3})$, which comes with a benefit of relaxing the smoothness requirement for $m$. For instance, for $a=1/2$, we only need the $\gamma$-H\"{o}lder smoothness with $\gamma>1/2$ and do not require that $m'$ exists, in which case, the convergence rate is $O_P(n^{-1/4})$. Due to the slower than the cube-root convergence rate, we do not recommend the "fast" boundary corrections for the point estimation, unless one expects that $m$ is not differentiable in the neighborhood of zero. However, the "fast" corrections are useful for inference, because as we shall show in the following section, the bootstrap works without additional subsampling or nonparametric smoothing.

\begin{remark}\label{remark:right_boundary}
	One can show that for a non-decreasing function $m:[-1,0]\to\R$ and $a\in(0,1/3)$ the asymptotic distribution of $\hat m(-cn^{-a})$ is the same as in Theorem~\ref{thm:irdd_asymptotics}, while for $a\in[1/3,1)$
	\begin{equation*}
	n^{(1-a)/2}\left(\hat m\left(-cn^{-a}\right) - m(0)\right) \cw D^L_{(-\infty,0]}\left(\sqrt{\frac{\sigma^2(0)}{cf(0)}}W_t + \frac{t^2c}{2}m'(0)\one_{a=1/3} \right)(-1),
	\end{equation*}
	where $f(0),m(0)$, and $\sigma^2(0)$ are defined as the left limits.
\end{remark}

\subsection{Trimmed wild bootstrap}
It is well-known that the bootstrap fails for various isotonic estimators at the interior point.\footnote{See \cite{kosorok2008bootstrapping} and \cite{sen2010inconsistency} for a formal statement and \cite{delgado2001subsampling}, \cite{leger2006bootstrap}, an \cite{abrevaya2005bootstrap} for earlier evidences.} Several resampling schemes are available in the literature, including the smoothed nonparametric or $m$-out-of-$n$ bootstrap, see \cite{sen2010inconsistency} and \cite{patra2018consistent}; reshaping the objective function, see \cite{cattaneo2017bootstrap}; and smoothed residual bootstrap, see \cite{guntuboyina2018nonparametric}. Interestingly, as we shall show below, for the boundary point, the "fast" boundary corrections restore the consistency of the bootstrap. Consequently, we focus on more conventional bootstrap inferences.

The wild bootstrap, see \cite{wu1986jackknife} and \cite{liu1988bootstrap}, is arguably the most natural resampling scheme for the nonparametric regression. Unlike the naive nonparametric bootstrap, the wild bootstrap imposes the structure of the nonparametric regression model in the bootstrap world, so we may expect it to work better in finite samples than resampling methods that do not incorporate such information. At the same time, unlike the residual bootstrap, it does not rule out the heteroskedasticity.

The bootstrap procedure is as follows. First, we obtain the isotonic regression estimator $\hat m$, construct the trimmed estimator\footnote{Note that the estimator $\tilde m$ depends on $c$ and $a$ and that for the sake of simplicity of presentation we suppress this dependence.}
\begin{equation*}
\tilde m(x) = \begin{cases}
\hat m(x), & x \in(cn^{-a},1) \\
\hat m(cn^{-a}), & x\in[0,cn^{-a}],
\end{cases}
\end{equation*}
and compute residuals $\tilde\varepsilon_i = Y_i - \tilde m(X_i)$ for $1\leq i\leq n$. Second, we construct the wild bootstrap samples as follows:
\begin{equation*}
Y_i^* = \tilde m(X_i) + \eta_i^*\tilde\varepsilon_i,\qquad 1\leq i\leq n,
\end{equation*}
where $(\eta_i^*)_{i=1}^n$ are i.i.d.\ random variables, independent of $(Y_i,X_i)_{i=1}^n$, and such that $\E\eta_i^* = 0$, $\Var(\eta_i^*) = 1$, and $\E|\eta_i^*|^{2+\delta}<\infty$.

Let $\Pr^*(.) = \Pr(.|(Y_i,X_i)_{i=1}^\infty)$ denote the bootstrap probability conditionally on the data, and let $\hat m^*$ be the isotonic regression estimator computed from the bootstrapped sample $(Y_i^*,X_i)_{i=1}^n$. 

Note that it follows from \cite{groeneboom1983}, Corollary 2.2 that the limiting distribution in Theorem~\ref{thm:isotonic_clt} is absolutely continuous with respect to the Lebesgue measure for every $a\in(1/3,1)$. The following result establishes the consistency of the trimmed wild bootstrap procedure.
\begin{theorem}\label{thm:bootstrap}
	Suppose that Assumption~\ref{as:dgp} is satisfied with (i) strengthened to (i') $\E[\varepsilon^{4}|X]\leq C$ for some $C<\infty$; and (iv) replaced by (iv') $m$ is $\gamma$-H\"{o}lder continuous in the neighborhood of zero with $\gamma>(1-a)/2a$. Then for every $u\in\R$ and $a\in(1/3,1)$
	\begin{equation*}
		\left|\mathrm{Pr}^*\left(n^\frac{1-a}{2}\left(\hat m^*(cn^{-a}) - \hat m(cn^{-a})\right) \leq u\right) - \mathrm{Pr}^*\left(n^\frac{1-a}{2}\left(\hat m(cn^{-a}) - m(0)\right) \leq u\right)\right| \xrightarrow{P} 0.
	\end{equation*}
\end{theorem}
Theorem~\ref{thm:bootstrap} shows that the boundary correction alone restores the consistency of the bootstrap. The reason is that for "fast" corrections the quadratic term $m'(0)$ disappears from the asymptotic distribution and this is precisely the term that is not consistently estimated by the bootstrap. Note that the bootstrap is consistent for every $a\in(1/3,1)$. However, there is a trade-off between the consistency rate of the bootstrapped estimator $O_P(n^{(a-1)/2})$ and the rate $O_P(n^{(1-3a)/2})$ at which the quadratic term disappears; see the proof of Theorem~\ref{thm:isotonic_clt}. This trade-off should translate into the trade-off between the rates at which the confidence intervals shrink and the rate at which coverage errors vanish. Therefore, it seems reasonable to take $a=1/2$ as a rule-of-thumb, which leads to the same $O_P(n^{-1/4})$ rate balancing the two rates.\footnote{The inference-optimal choice of $c$ is more difficult. One can still use the rule-of-thumb $c=1$. We find in Monte Carlo experiments that there is some scope for improvements over this default choice.} Note that the slower $O_P(n^{-1/4})$ convergence rate is the price to pay for making the bootstrap work. However, the slower rate also comes with a benefit since this result does not require the existence of $m'$ and relies only on the $\gamma$-H\"{o}lder continuity with smoothness index $\gamma>1/2$.

\begin{remark}
	It is also possible to show that for the non-decreasing function $m:[-1,0]\to\R$ and $a\in(1/3,1)$, the bootstrap is consistent in probability.
\end{remark}

\section{Isotonic regression discontinuity designs}\label{sec:irdd}
In this section, we apply our results on the isotonic regression at the boundary to monotone regression discontinuity designs. 

\subsection{Setup}
Following \cite{hahn2001identification}, we focus on the potential outcomes framework
\begin{equation*}
Y = Y_1D + Y_0(1-D),
\end{equation*}
where $D\in\{0,1\}$ is a binary treatment indicator (1 if treated and 0 otherwise), $Y_1,Y_0\in\R$ are unobserved potential outcomes of treated and untreated units, and $Y$ is the actual observed outcome.

The causal parameter of interest is the average treatment effect at the cutoff $x_0\in\R$ of some running variable $X\in\R$, denoted
\begin{equation*}
	\theta = \E[Y_1-Y_0|X=x_0].
\end{equation*}
Without further assumptions, $\theta$ is not identified in the sense that it depends on the distribution of unobserved potential outcomes $(Y_0,Y_1)$; see \cite{holland1986statistics}.\footnote{While regression discontinuity designs identify local effects, there exist several approaches that aim to extrapolate away from the cutoff; see \cite{bertanha2020regression}, \cite{angrist2015wanna}, and other references therein. Note that if the causal effect changes monotonically in the running variable, one can obtain bounds on the causal effect between the two points.}

\cite{hahn2001identification} shows formally that the causal effect is identified as
\begin{equation}\label{eq:id}
	\theta = \lim_{x\downarrow x_0}\E[Y|X=x] - \lim_{x\uparrow x_0}\E[Y|X=x],
\end{equation}
while for fuzzy designs, the causal effect is identified as\footnote{Identification in regression discontinuity designs might be testable; see \cite{mccrary2008manipulation}, \cite{cattaneo2019simple}, \cite{bugni2018testing}, \cite{canay2017approximate}.}
\begin{equation}\label{eq:id2}
	\theta = \frac{\lim_{x\downarrow x_0}\E[Y|X=x] - \lim_{x\uparrow x_0}\E[Y|X=x]}{\lim_{x\downarrow x_0}\Pr(D=1|X=x) - \lim_{x\uparrow x_0}\Pr(D=1|X=x)};
\end{equation}
see Proposition~\ref{prop:identification} in the Supplementary Material for a statement of these results under weak conditions imposed on conditional means of potential outcomes.

\subsection{Estimators}
We are interested in estimating the average causal effect $\theta$ of a binary treatment $D$ on some outcome $Y$. For a unit $i$, with $1\leq i\leq n$, we observe $(Y_i,D_i,X_i)$. Assuming that $\theta$ is identified from the distribution of $(Y,D,X)$ according to Eq.~(\ref{eq:id}), the estimator of $\theta$ is obtained by plugging-in the corresponding boundary corrected isotonic regression estimators before and after the cutoff.

Isotonic regression discontinuity designs (iRDD) exploit the monotonicity of the expected outcome $m(x) = \E[Y|X=x]$ and the treatment assignment probability $p(x) = \Pr(D=1|X=x)$. For concreteness, assume that both functions are non-decreasing. We also assume that $X$ has compact support $[-1,1]$ and normalize the cutoff to $x_0=0$. This restriction is without loss of generality up to the monotone transformation. Let $\hat m_-$ and $\hat m_+$ be solutions to
\begin{equation*}
	\min_{\phi\in\mathcal{M}[-1,0)}\sum_{i\in I_-}(Y_{i} - \phi(X_{i}))^2\qquad\text{and}\qquad \min_{\phi\in\mathcal{M}[0,1]}\sum_{i\in I_+}(Y_{i} - \phi(X_{i}))^2
\end{equation*}
respectively, where $I_-$ and $I_+$ are sets of indices corresponding to negative and positive values of observations of the running variable. Therefore, $\hat m_-$ and $\hat m_+$ are the isotonic regression estimators before and after the cutoff. Note that the isotonic regression is a solution to the constrained least-squares problem over the closed convex cone of monotone functions and should have the projection interpretation when the monotonicity is violated. We might expect that the deviations from monotonicity far from the cutoff should be less harmful as opposed to the deviations at the cutoff.

\paragraph{Choice of $c$ and $a$.} Following the discussion in Section~\ref{sec:isotonic_regression}, we recommend setting $c=1$\footnote{We also consider the MSE-optimal choice of $c$ in Monte Carlo experiments and find that it tends to increase the MSE in small samples.} and $a=1/3$ for the point estimation and $a=1/2$ for inference. For the point estimation, this choice delivers the one-sided counterpart to the asymptotic distribution at the interior obtained with a tuning-free isotonic regression, while for inference, this choice leads to the consistent bootstrap estimates of the asymptotic distribution, balancing the accuracy of the asymptotic approximation and the convergence rate of the estimator. The choice $a=1/3$ for point estimation can also be justified under the quadratic loss function as the one that balances the rates of the bias and the variance. On the other hand, we expect that the choice $a=1/2$ for inference should balance the rate of convergence of confidence intervals and the rate of the accuracy of bootstrap approximations. The inference-optimal choice of $c$ is a more difficult problem that we leave for future research. One possible solution is to develop the loss-function based approach that pays attention to the coverage errors and confidence interval length; see \cite{calonico2020optimal} or \cite{lazarus2018har} for this type of analysis.\footnote{Recently, \cite{han2019berry} develop Berry-Esseen inequalities that might be leveraged on to characterize the coverage error.}

\medskip\medskip

Therefore, for point estimates in sharp designs we focus on the following boundary-corrected iRDD estimator
\begin{equation*}
\hat{\theta} = \hat m_+\left(cn^{-1/3}\right) - \hat m_-\left(-cn^{-1/3}\right).
\end{equation*}
For fuzzy designs, we also need to estimate treatment assignment probabilities before and after the cutoff, denoted $\hat p_-$ and $\hat p_+$, solving
\begin{equation*}
	\min_{\pi\in\mathcal{M}[-1,0)}\sum_{i\in I_-}(D_{i} - \pi(X_{i}))^2\qquad \text{and} \qquad \min_{\pi\in\mathcal{M}[0,1]}\sum_{i\in I_+}(D_{i} - \pi(X_{i}))^2.
\end{equation*}
The fuzzy iRDD estimator is computed as
\begin{equation*}
	\hat\theta^{F} = \frac{\hat m_+\left(cn^{-1/3}\right) - \hat m_-\left(-cn^{-1/3}\right)}{\hat{p}_+\left(cn^{-1/3}\right) - \hat p_-\left(-cn^{-1/3}\right)}.
\end{equation*}

\subsection{Large sample distribution}
Put $p(x)=\Pr(D=1|X=x)$ and for a function $g:[-1,1]\to\R$, define $g_+=\lim_{x\downarrow 0}g(x)$ and $g_-=\lim_{x\uparrow0}g(x)$ with some abuse of notation. The following assumption imposes several mild restrictions on the distribution of the data.

\begin{assumption}\label{as:irdd_dgp}
	$(Y_i,D_i,X_i)_{i=1}^n$ is an i.i.d. sample of $(Y,D,X)$ such that (i) $\E[|\varepsilon|^{2+\delta}|X]\leq C<\infty$ for some $\delta>0$ and $m$ is uniformly bounded; (ii) the distribution of $X$ has Lebesgue density $f$, uniformly bounded away from zero and infinity on the support of $X$, and such that $f_-$ and $f_+$ exist; (iii) $\sigma^2$ is uniformly bounded on $[-1,1]$ and $\sigma^2_+$ and $\sigma^2_-$ exist; (iv) $m$ is continuously differentiable in the right and left neighborhoods of zero with $m'_-,m'_+>0$; (v) $m\in\mathcal{M}[-1,1]$.
\end{assumption}
Assumption~\ref{as:irdd_dgp} is comparable to assumptions typically used in the RDD literature, e.g., see \cite{hahn2001identification}, Theorem 4. Note that we are agnostic about the smoothness of the marginal density of $X$, and only assume the existence of one-sided derivatives of conditional means. The following results describes the large sample approximation to the distribution of the sharp iRDD estimator.

\begin{theorem}\label{thm:irdd_asymptotics}
	Suppose that Assumption~\ref{as:irdd_dgp} is satisfied. Then
	\begin{equation*}
	n^{1/3}(\hat\theta - \theta)\cw D_{[0,\infty)}^L\left(\sqrt{\frac{\sigma^2_+}{cf_+}}W_t^+ + \frac{t^2c}{2}m_+'\right)(1) - D_{(-\infty,0]}^L\left(\sqrt{\frac{\sigma_-^2}{cf_-}}W_t^- + \frac{t^2c}{2}m_-'\right)(-1),
	\end{equation*}
	where $W_t^+$ and $W_t^-$ are two independent standard Brownian motions originating from zero and running in opposite directions.
\end{theorem}
The proof of this result can be found in the appendix. Theorem~\ref{thm:irdd_asymptotics} shows that the large sample distribution of the sharp iRDD estimator can be approximated by the difference of slopes of the greatest convex minorants of two scaled independent Brownian motions plus the parabola originating from zero and running in opposite directions. To describe the large sample distribution for the fuzzy iRDD, define additionally the conditional covariance function $\rho(x)=\E[\varepsilon(D-p(X))|X=x]$.
\begin{theorem}\label{thm:irdd_fuzzy}
	Suppose that Assumption~\ref{as:irdd_dgp} is satisfied. Suppose additionally that $p\in\mathcal{M}[-1,1]$ is continuously differentiable in the right and left neighborhoods of zero with $p'_-,p'_+>0$, and $p_+,p_-\in(0,1)$. Then
	\begin{equation*}
	\begin{aligned}
	n^{1/3}(\hat\theta^F - \theta)\cw \frac{1}{p_+ - p_-}\xi_1 - \frac{m_+ - m_-}{(p_+ - p_-)^2}\xi_2
	\end{aligned}
	\end{equation*}
	with
	\begin{equation*}
	\begin{aligned}
	\xi_1 & = D_{[0,\infty)}^L\left(\sqrt{\frac{\sigma^2_+}{cf_+}}W_t^+ + \frac{t^2c}{2}m_+'\right)(1) - D_{(-\infty,0]}^L\left(\sqrt{\frac{\sigma_-^2}{cf_-}}W_t^- + \frac{t^2c}{2}m_-'\right)(-1) \\
	\xi_2 & = D_{[0,\infty)}^L\left(\sqrt{\frac{p_+(1-p_+)}{cf_+}}B_t^+ + \frac{t^2c}{2}p_+'\right)(1) - D_{(-\infty,0]}^L\left(\sqrt{\frac{p_-(1-p_-)}{cf_-}}B_t^- + \frac{t^2c}{2}p_-'\right)(-1),
	\end{aligned}
	\end{equation*}
	where $W_t^+,W_t^-,B^+_t$, and $B^-_t$ are standard Brownian motions such that any two processes with different signs are independent, and
	\begin{equation*}
		\Cov(W_t^+,B_s^+) = \frac{\rho_+(t\wedge s)}{\sqrt{\sigma^2_+p_+(1-p_+)}} ,\qquad	\Cov(W_t^-,B_s^-) = \frac{\rho_-(t\wedge s)}{\sqrt{\sigma^2_-p_-(1-p_-)}}.
	\end{equation*}
\end{theorem}

Both results follow from the CLT for the boundary-corrected isotonic regression estimator obtained in Theorem~\ref{thm:isotonic_clt}; see appendix for the formal proof. A consequence of Theorems~\ref{thm:irdd_asymptotics} and \ref{thm:irdd_fuzzy} is that the boundary-corrected iRDD estimators $\hat \theta$ and $\hat\theta^F$ are consistent in probability for the causal effect parameter $\theta$ and provide valid point estimates.

\subsection{Trimmed wild bootstrap}\label{sec:irdd_wb}
In this section, we study the trimmed wild bootstrap. The resampling procedure is as follows. First, we construct the trimmed estimator
\begin{equation*}
\tilde m(x) = \begin{cases}
\hat m_-(x),  & x\in (-1,-cn^{-1/2}) \\
\hat m_-(-cn^{-1/2}), & x\in [-cn^{-1/2},0] \\
\hat m_+(cn^{-1/2}), & x\in(0,cn^{-1/2}] \\
\hat m_+(x), & x\in (cn^{-1/2},1).
\end{cases}
\end{equation*}
Second, we simulate the wild bootstrap samples as follows
\begin{equation*}
Y_i^* = \tilde m(X_i) + \eta_i^*\tilde\varepsilon_i,\qquad i=1,\dots,n,
\end{equation*}
where $(\eta_i^*)_{i=1}^n$ are i.i.d. multipliers, independent of $(Y_i,D_i,X_i)_{i=1}^n$, and $\tilde\varepsilon_i = Y_i - \tilde m(X_i)$. We call this procedure trimming since it trims the estimator close to boundaries when we generate bootstrap samples. Trimming is needed in addition to the boundary correction of the iRDD estimator
\begin{equation*}
	\check{\theta} = \hat m_+\left(cn^{-1/2}\right) - \hat m_-\left(-cn^{-1/2}\right)
\end{equation*}
and its bootstrapped counterpart
\begin{equation*}
	\check{\theta}^* = \hat m^{*}_+(cn^{-1/2}) - \hat m^{*}_-(-cn^{-1/2}),
\end{equation*}
where $\hat m^{*}_+$ and $\hat m^{*}_-$ are isotonic estimators computed from the bootstrapped sample $(Y_i^*,D_i,X_i)_{i=1}^n$ similarly to $\hat m_-$ and $\hat m_+$. 

The following results establishes the bootstrap consistency for the sharp iRDD.
\begin{theorem}\label{thm:irdd_bootstrap}
	Suppose that Assumptions~\ref{as:irdd_dgp} are satisfied with (i) strengthened to (i') $\E[\varepsilon^{4}|X]\leq C$ for some $C<\infty$; and (iv) replaced by (iv') $m$ is $\gamma$-H\"{o}lder continuous in the left and the right neighborhoods of zero with $\gamma>1/2$ and continuous on $[-1,0)$ and $(0,1]$. If multipliers $(\eta_i^*)_{i=1}^n$ are such that $\E\eta_i^* = 0$, $\Var(\eta_i^*)=1$, and $\E|\eta_i^*|^{2+\delta}<\infty$ for some $\delta>0$, then for every $u\in\R$
	\begin{equation*}
	\left|\mathrm{Pr}^*\left(n^{1/4}(\check \theta^* - \check \theta) \leq u\right) - \Pr\left(n^{1/4}(\check{\theta} - \theta) \leq u\right)\right| \xrightarrow{P} 0,
	\end{equation*}
	where $\Pr^*(.) = \Pr(.|(X_i,Y_i)_{i=1}^\infty)$.
\end{theorem}

\section{Monte Carlo experiments}\label{sec:mc_experiments}
In this section, we study the finite-sample performance of our iRDD estimator. We simulate $5,000$ samples of size $n\in\{200,500,1000\}$ as follows:
\begin{equation*}
Y = m(X) + \theta\one_{[0,1]}(X) + \sigma(X)\varepsilon, 
\end{equation*}
where $\varepsilon\sim N(0,1)$ and $\varepsilon\si X$. In the baseline DGP, we set $m(x) = x^3+0.25x$, $\theta = 1$, $X\sim 2\times \mathrm{Beta}(2,2)-1$, and $\sigma(x)=1$ (homoskedasticity). We compute the boundary-corrected estimator using the pool adjacent violators algorithm.

Figures~\ref{fig:mc_distr} illustrates the finite sample distribution of the boundary-corrected iRDD estimator for samples of different sizes. The exact finite-sample distribution is centered around the population value of the parameter and concentrates around the population parameter as the sample size increases. 

\begin{figure}[ht]
	\begin{subfigure}[b]{0.3\textwidth}
		\includegraphics[width=\textwidth]{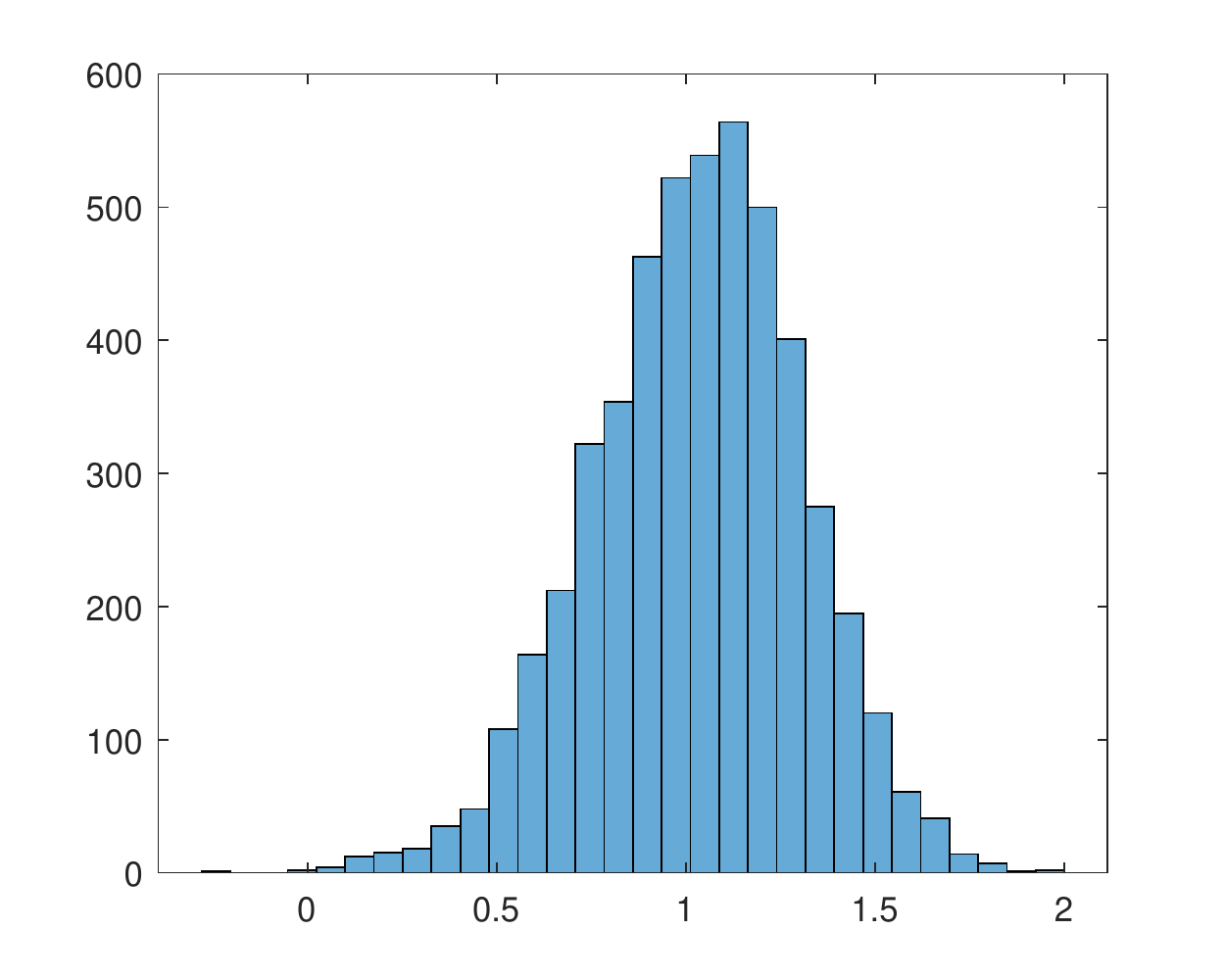}
		\caption{$n = 200$}
	\end{subfigure}
	\begin{subfigure}[b]{0.3\textwidth}
		\includegraphics[width=\textwidth]{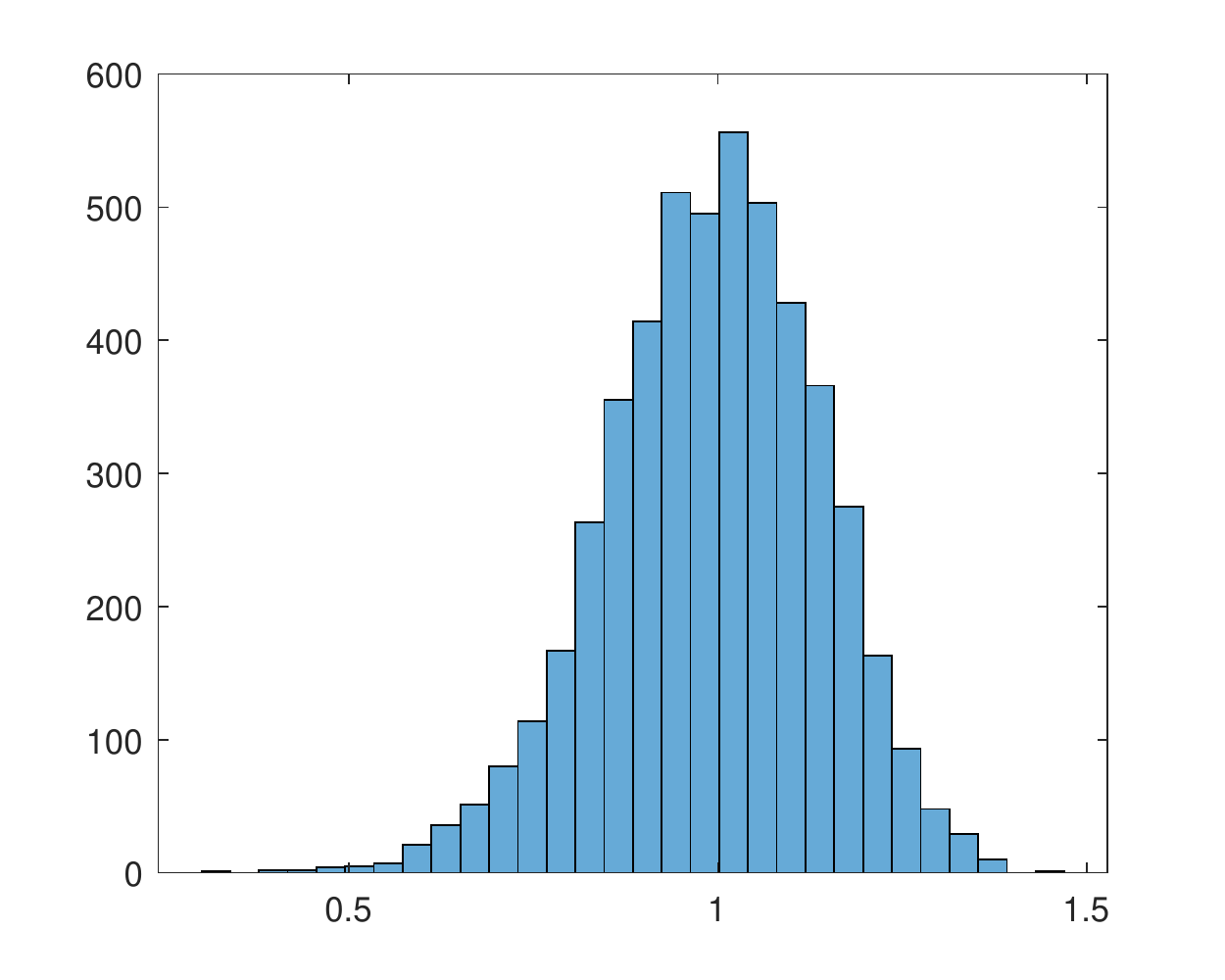}
		\caption{$n = 500$}
	\end{subfigure}
	\begin{subfigure}[b]{0.3\textwidth}
		\includegraphics[width=\textwidth]{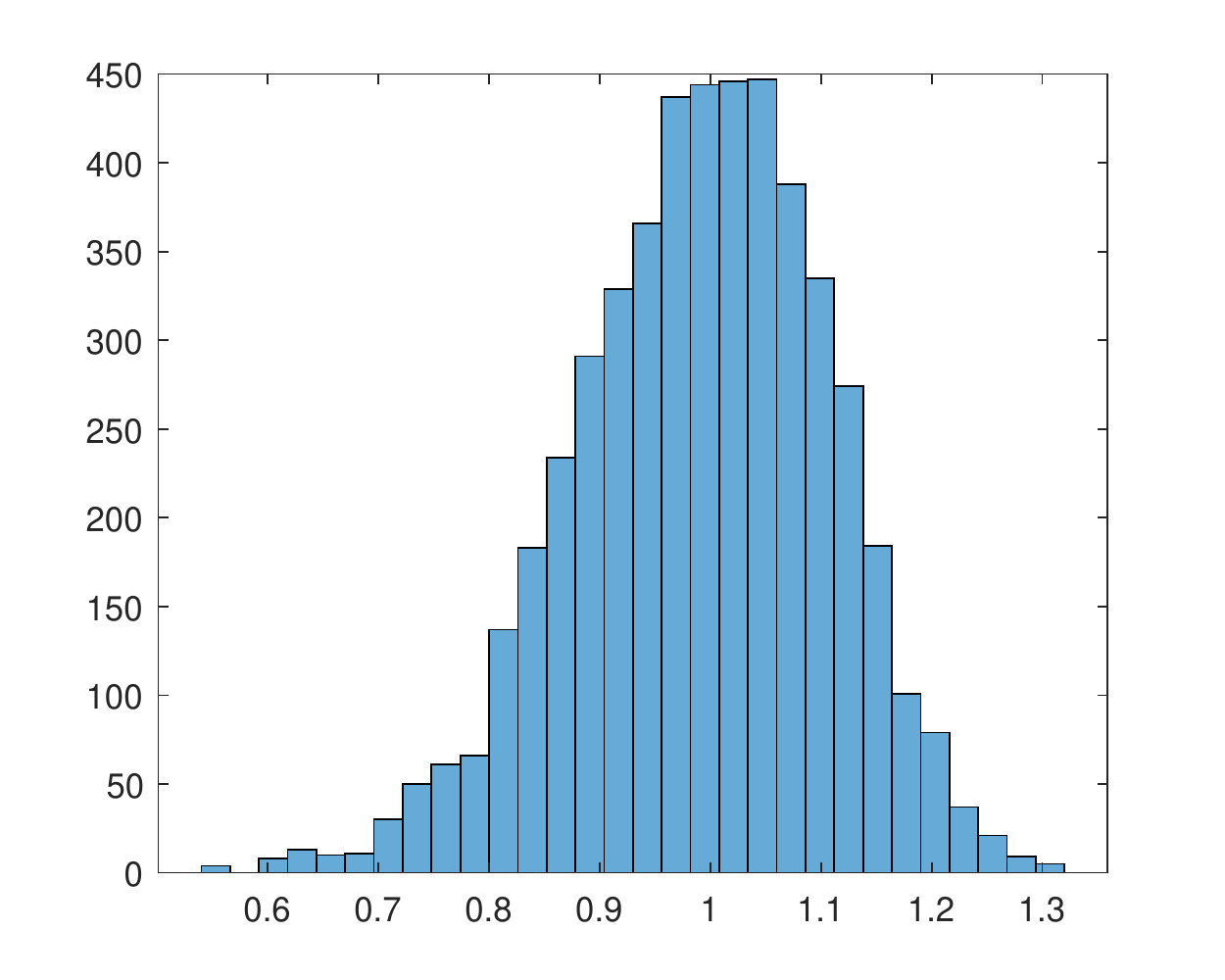}
		\caption{$n = 1000$}
	\end{subfigure}
	\begin{subfigure}[b]{0.3\textwidth}
		\includegraphics[width=\textwidth]{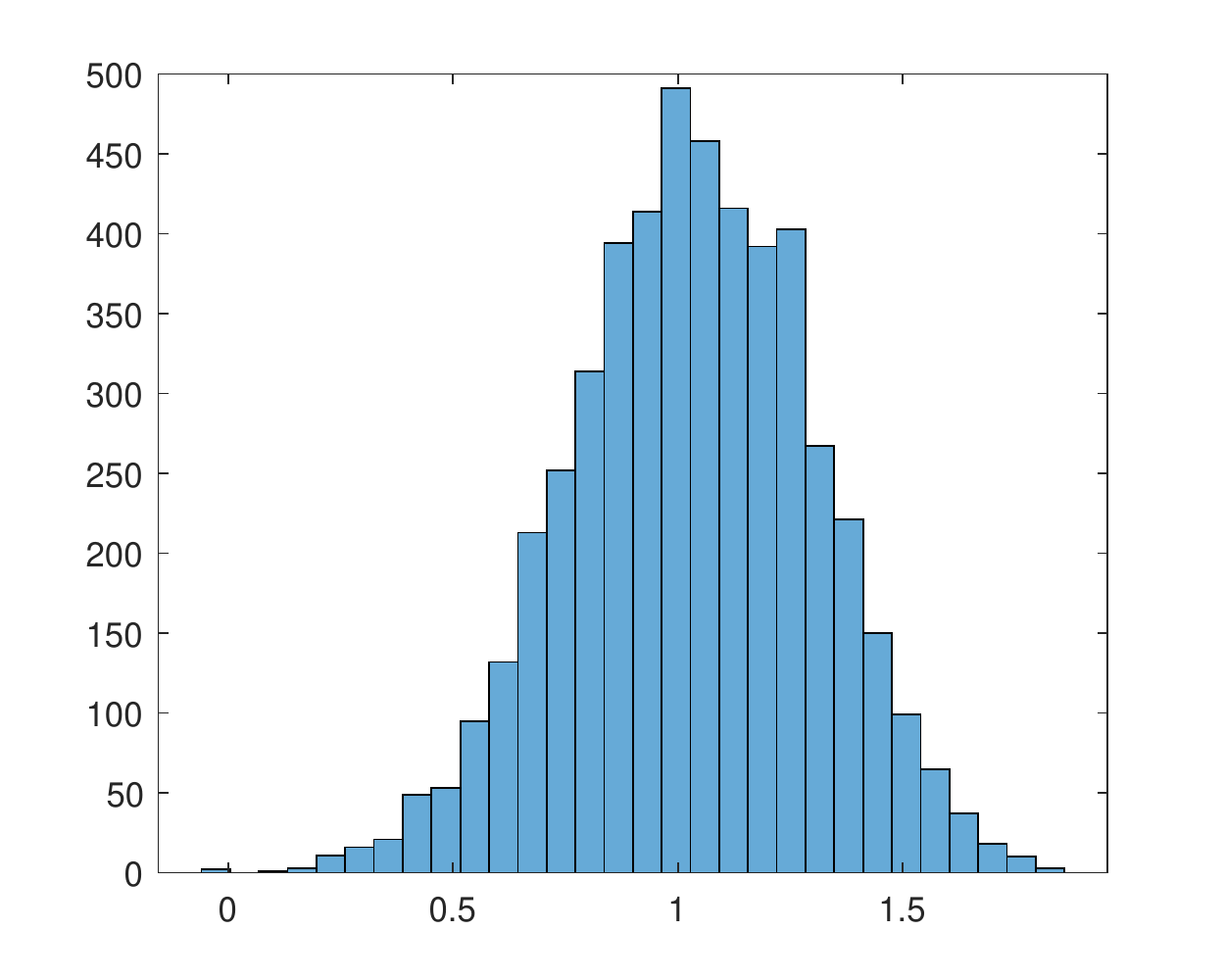}
		\caption{$n = 200$}
	\end{subfigure}\qquad
	\begin{subfigure}[b]{0.3\textwidth}
		\includegraphics[width=\textwidth]{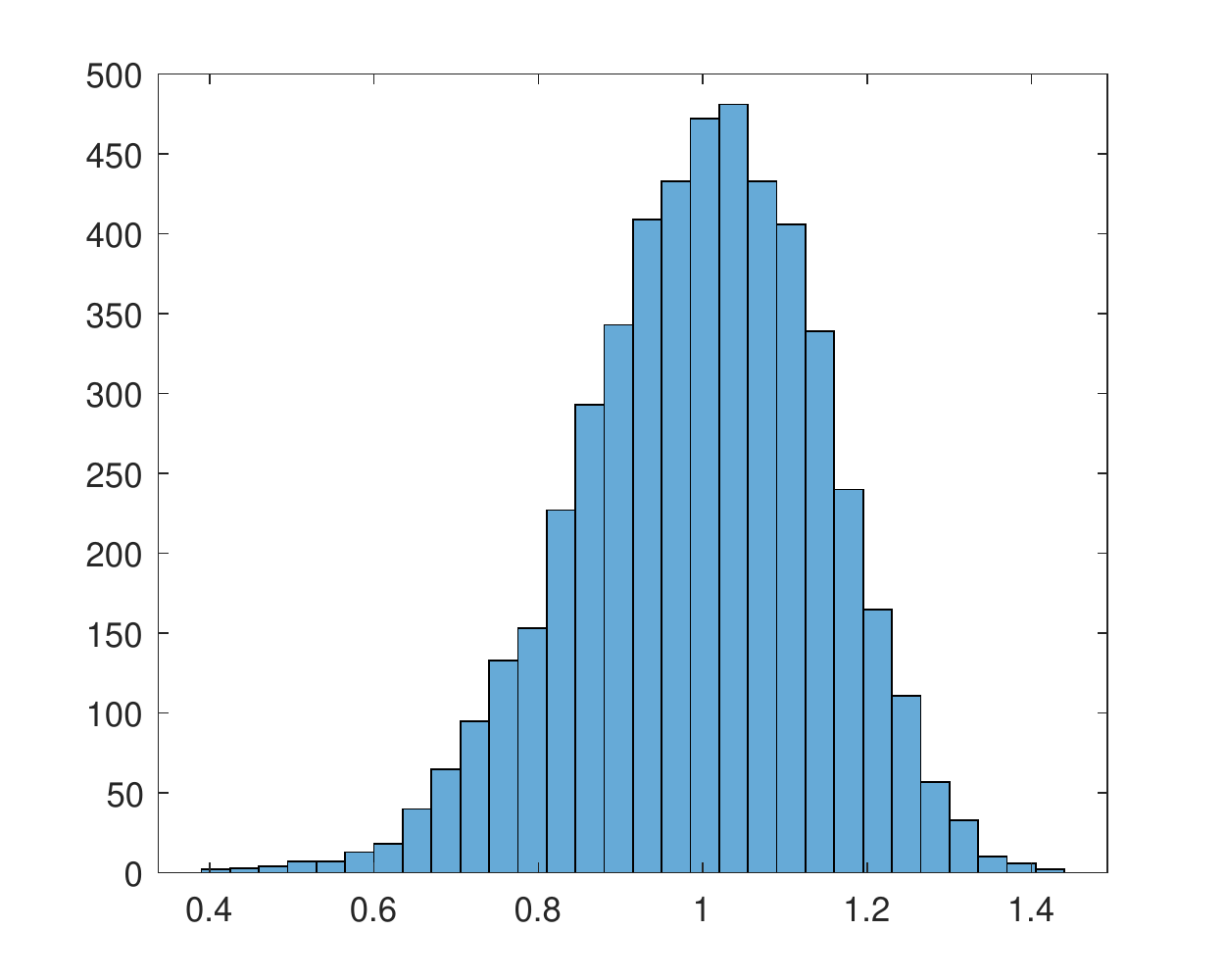}
		\caption{$n = 500$}
	\end{subfigure}\quad
	\begin{subfigure}[b]{0.3\textwidth}
		\includegraphics[width=\textwidth]{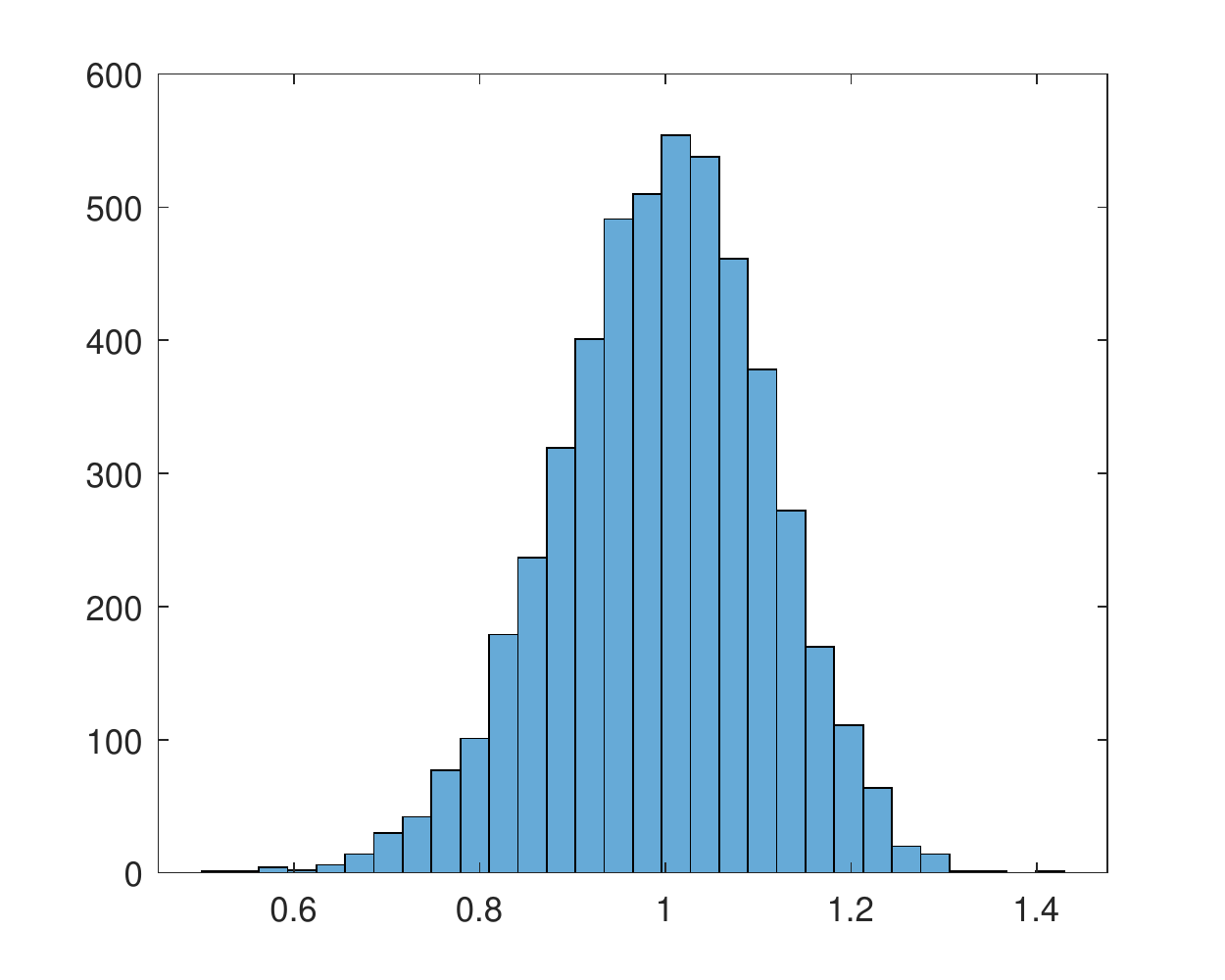}
		\caption{$n = 1000$}
	\end{subfigure}
	\caption{Finite sample distribution of the iRDD estimator: homoskedastic design in (a)-(c) and heteroskedastic in (d)-(f). Results based on 5,000 MC experiments.}
	\label{fig:mc_distr}
\end{figure}

Table~\ref{tab:mse} report results of more comprehensive Monte Carlo experiments for several data-generating processes and shows the exact finite-sample bias, variance, and MSE of our iRDD estimator. We consider the following variations of the baseline DGP with two different functional forms and different amount of density near the cutoff:
\begin{enumerate}
	\item DGP 1 sets $X\sim 2\times\mathrm{Beta}(2,2)-1$ and $m(x)=\exp(0.25x)$;
	\item DGP 2 sets $X\sim 2\times\mathrm{Beta}(0.5,0.5)-1$ (low density) and $m(x)=\exp(0.25x)$;
	\item DGP 3 sets $X\sim 2\times\mathrm{Beta}(2,2)-1$ and $m(x)=x^3+0.25x$;
	\item DGP 4 sets $X\sim 2\times\mathrm{Beta}(0.5,0.5)-1$ (low density) and $m(x)=x^3+0.25x$;
\end{enumerate}
To see the benefits of shape restrictions in small samples, we benchmark our iRDD estimators with the unrestricted local polynomial\footnote{We rely on the implementation available in the {\vtext} R package; see \cite{calonico2014robust}, \cite{calonico2019regression}, and \cite{calonico2020optimal} for details.} and the k-nearest neighbors estimators.\footnote{Note that the iRDD estimator with the boundary correction $\lfloor n^{1/3}\rfloor$ might resemble the k-nearest neighbors estimator with $k_n = \lfloor n^{1/3}\rfloor$ neighbors. However, such comparison is deceptive because the k-nearest neighbors estimator with $\lfloor n^{1/3}\rfloor$ neighbors is restricted to take a constant value in the \textit{entire} neighborhood of the boundary, while the isotonic regression estimator is \textit{unrestricted}.} Results of experiments are consistent with the asymptotic theory. The bias, the variance, and the MSE decrease with the sample size. As expected, the MSE is higher when the density near the cutoff is lower. The shape-constrained iRDD estimator outperforms the unrestricted estimators. In many cases, we achieve more than 5-10 fold reduction of the MSE. In the Supplementary Material, Section~\ref{app:mc}, we also consider DGPs, where the regression function violates the monotonicity constraint and regressions that are steep at the cutoff. The results are still favorable towards the iRDD estimator. 

We shall acknowledge that the local linear and the k-NN estimators are applicable in more general settings as they do not rely on the monotonicity assumption and we do not recommend using the iRDD estimator for non-monotone designs.\footnote{The monotonicity of regression is a testable; see \cite{chetverikov2019testing} and references therein.} However, under the monotonicity, our estimator adapts well to the boundary and seems to outperforms the existing methods. We shall also note that the case of the steep regression function is difficult for all existing approaches and that it becomes difficult to distinguish between the discontinuous and the steep continuous regression functions; see \cite{bertanha2020regression} for related impossibility results.

\begin{table}
	\centering                     	
	\begin{threeparttable}
		\caption{Results of Monte Carlo experiments}
		\begin{tabular}{ccccccccccc}                                                                                                                                                     
			\hline\hline && \multicolumn{3}{c}{iRDD} &  \multicolumn{3}{c}{LP} &  \multicolumn{3}{c}{k-NN} \\ 		\cline{3-5} \cline{6-11}			 &$n$ & Bias & Var & MSE & Bias & Var & MSE & Bias & Var & MSE \\
			\hline\textbf{DGP 1} & 200 & 0.001 & 0.043 & 0.043&0.008& 0.319& 0.319 & 0.010 & 0.500 & 0.500 \\                                                                                                                                     
			& 500 & -0.009 & 0.022 & 0.022& -0.007 & 0.110 & 0.110& 0.010 & 0.280 & 0.280 \\                                                                                                                                    
			&1000  & -0.008 & 0.013 & 0.013& -0.002 & 0.056 & 0.056& 0.000 & 0.220 & 0.220 \\                                                                                                                                   
			\hline\textbf{DGP 2} & 200 & -0.111 & 0.080 & 0.092&0.003& 0.747& 0.747 &  0.040 & 0.500 & 0.500 \\                                                                                                                                    
			& 500 & -0.079 & 0.042 & 0.049& -0.002 & 0.225 & 0.225& 0.010 & 0.290 & 0.290 \\                                                                                                                                    
			& 1000 & -0.063 & 0.027 & 0.031& 0.001 & 0.102 & 0.102& 0.000 & 0.220 & 0.220 \\                                                                                                                                   
			\hline\textbf{DGP 3} & 200 & -0.131 & 0.041 & 0.057&-0.002& 0.322& 0.322 & 0.030 & 0.510 & 0.510 \\                                                                                                                                    
			& 500 & -0.114 & 0.019 & 0.031& -0.005 & 0.109 & 0.109& 0.010 & 0.290 & 0.290 \\                                                                                                                                    
			& 1000 & -0.098 & 0.011 & 0.020& 0.001 & 0.056 & 0.056&  0.010 & 0.220 & 0.220  \\                                                                                                                                   
			\hline\textbf{DGP 4} & 200 & -0.337 & 0.075 & 0.158&-0.019& 0.773& 0.774 & 0.050 & 0.490 & 0.490 \\                                                                                                                                    
			& 500 & -0.261 & 0.037 & 0.087& -0.021 & 0.213 & 0.213&  0.010 & 0.290 & 0.290 \\                                                                                                                                    
			& 1000 & -0.213 & 0.021 & 0.055& -0.007 & 0.104 & 0.104& 0.010 & 0.220 & 0.220 \\                                                                                                          
			\hline\hline                                                                      
		\end{tabular}  
		\begin{tablenotes}
			\small
			\item Note: the figure shows exact finite-sample bias, variance, and MSE of iRDD, local polynomial (LP), and k-NN estimators. Results based on 5000 experiments.  Local linear estimator with kernel=`triangular' and bwselect=`mserd'.
		\end{tablenotes}
		\label{tab:mse}
	\end{threeparttable}
\end{table} 	

We asses the validity of the trimmed wild bootstrap approximations, In Figure~\ref{fig:mc_bootstrap}, we plot the exact distribution  $n^{1/4}(\hat\theta - \theta)$ and the bootstrap distribution $n^{1/4}(\hat\theta^* - \hat\theta)$ for samples of size $n\in\{200,1000\}$. Both estimators are computed using $a=1/2$ and $c=1$. As we can see from panels (b) and (e), the naive wild bootstrap without trimming and boundary correction does not work. On the other hand, the trimmed wild bootstrap mimics the finite-sample distribution. In our simulations, we use Rademacher multipliers for the bootstrap in all our experiments, i.e., $\eta_i\in\{-1,1\}$ with equal probabilities.

Lastly, we provide results of more detailed Monte Carlo experiments in the Supplementary Material. These experiments feature DGPs with heteroskedasticity, steep regression function near the cutoff, and regressions that are non-monotone away from the cutoff. The results show 2-10 fold reduction in the MSE across specifications.

\begin{figure}[ht]
	\begin{subfigure}[b]{0.3\textwidth}
		\includegraphics[width=\textwidth]{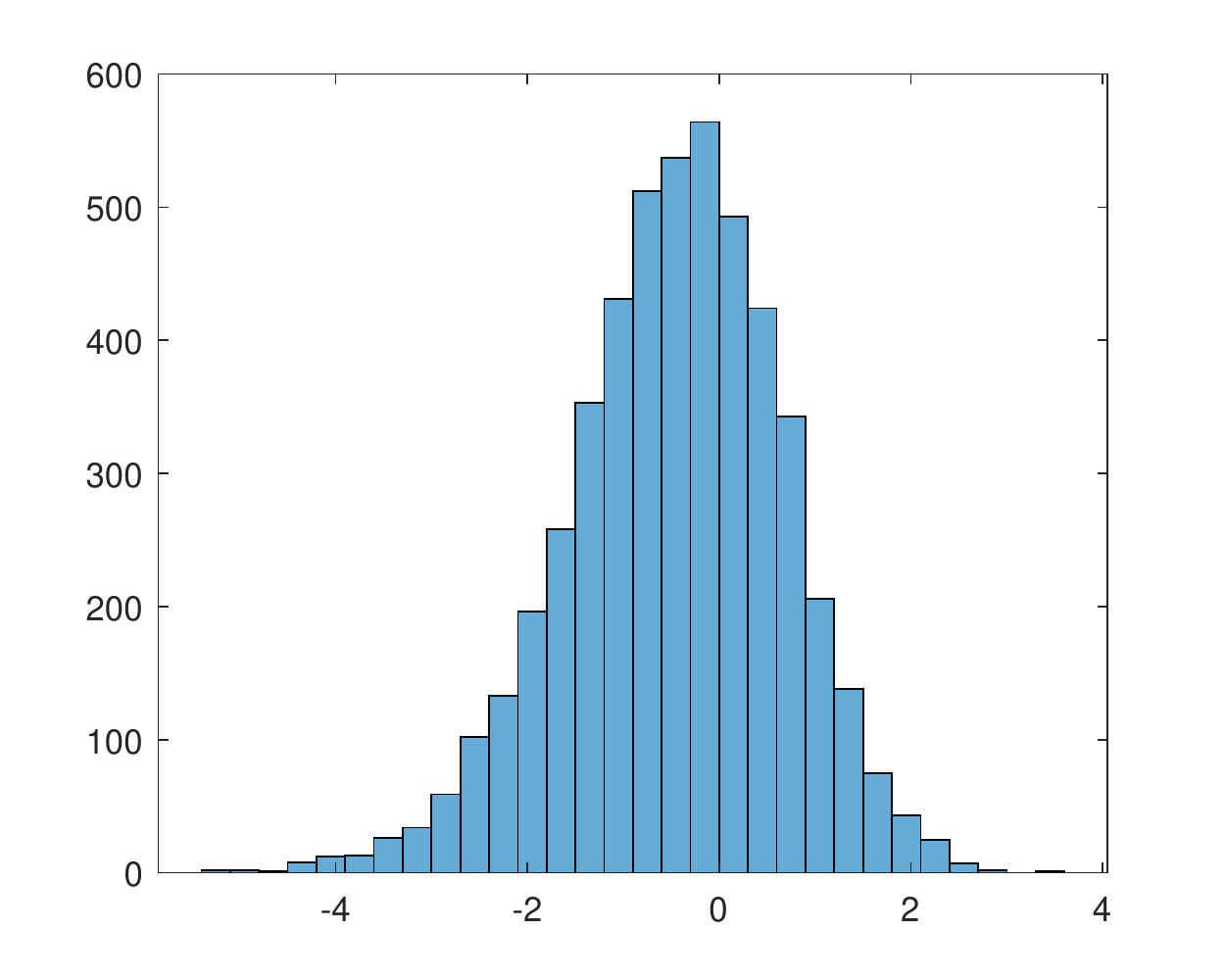}
		\caption{Exact distribution}
	\end{subfigure}
	\begin{subfigure}[b]{0.3\textwidth}
		\includegraphics[width=\textwidth]{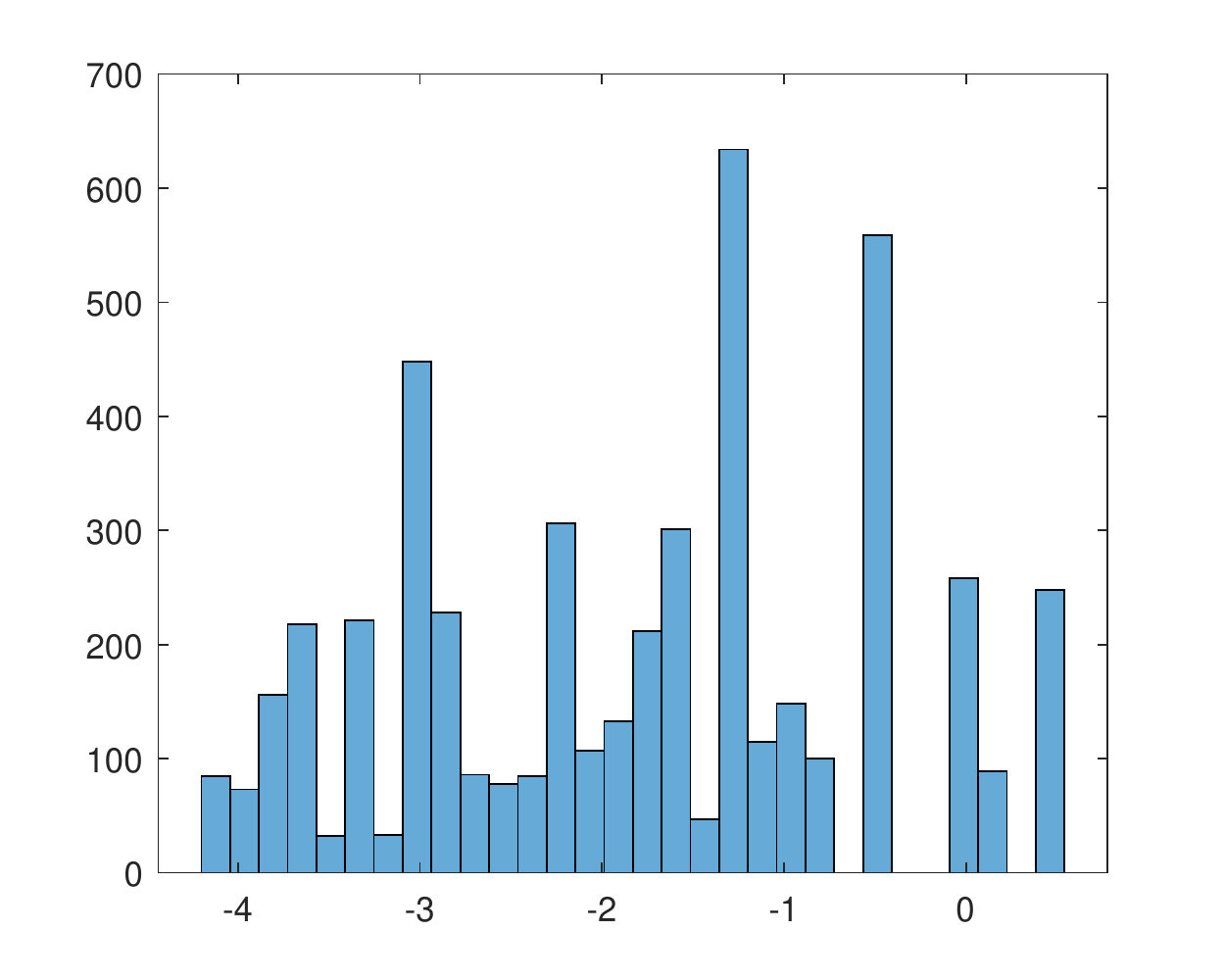}
		\caption{Naive bootstrap}
	\end{subfigure}
	\begin{subfigure}[b]{0.3\textwidth}
		\includegraphics[width=\textwidth]{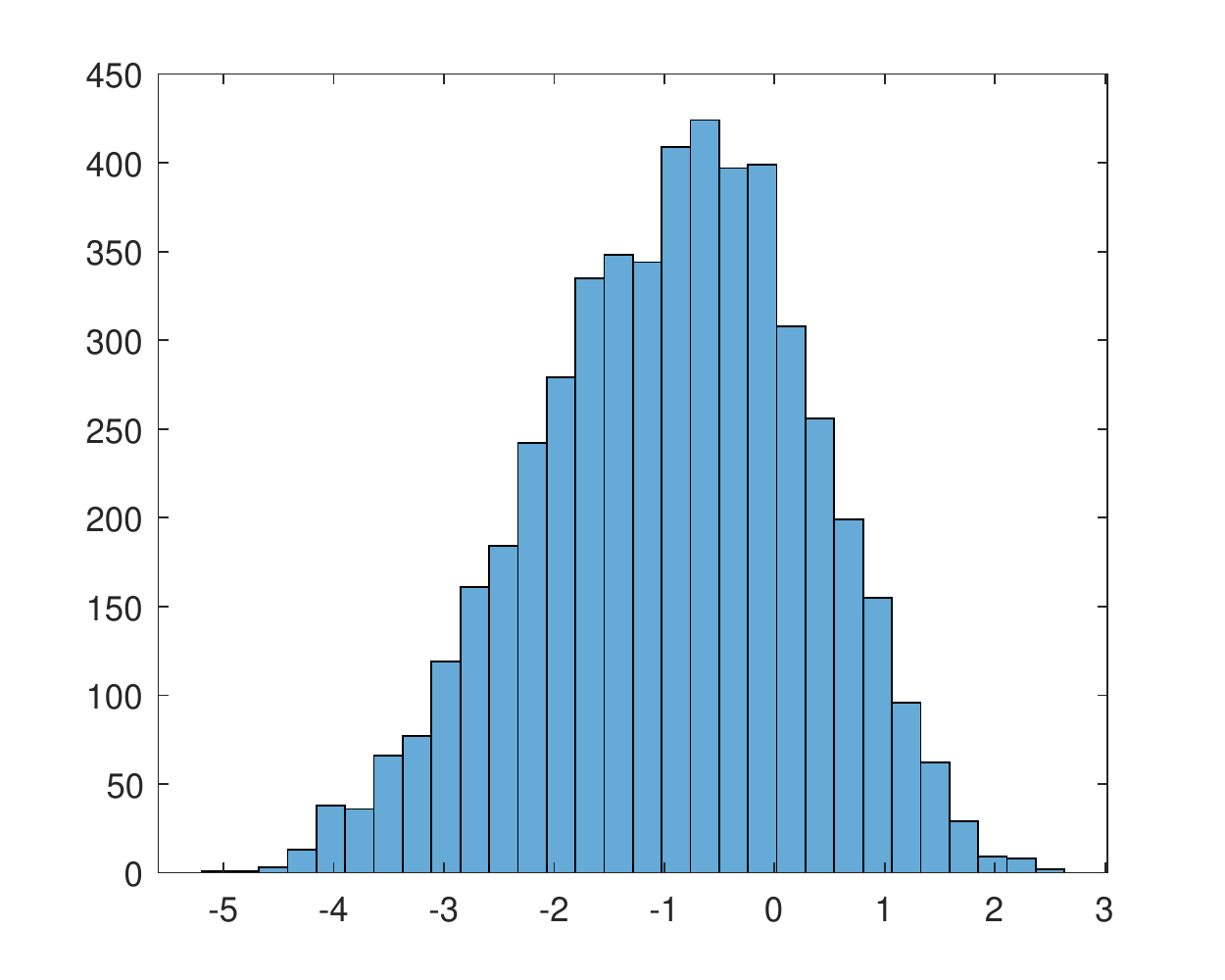}
		\caption{Trimmed bootstrap}
	\end{subfigure}
	\begin{subfigure}[b]{0.3\textwidth}
		\includegraphics[width=\textwidth]{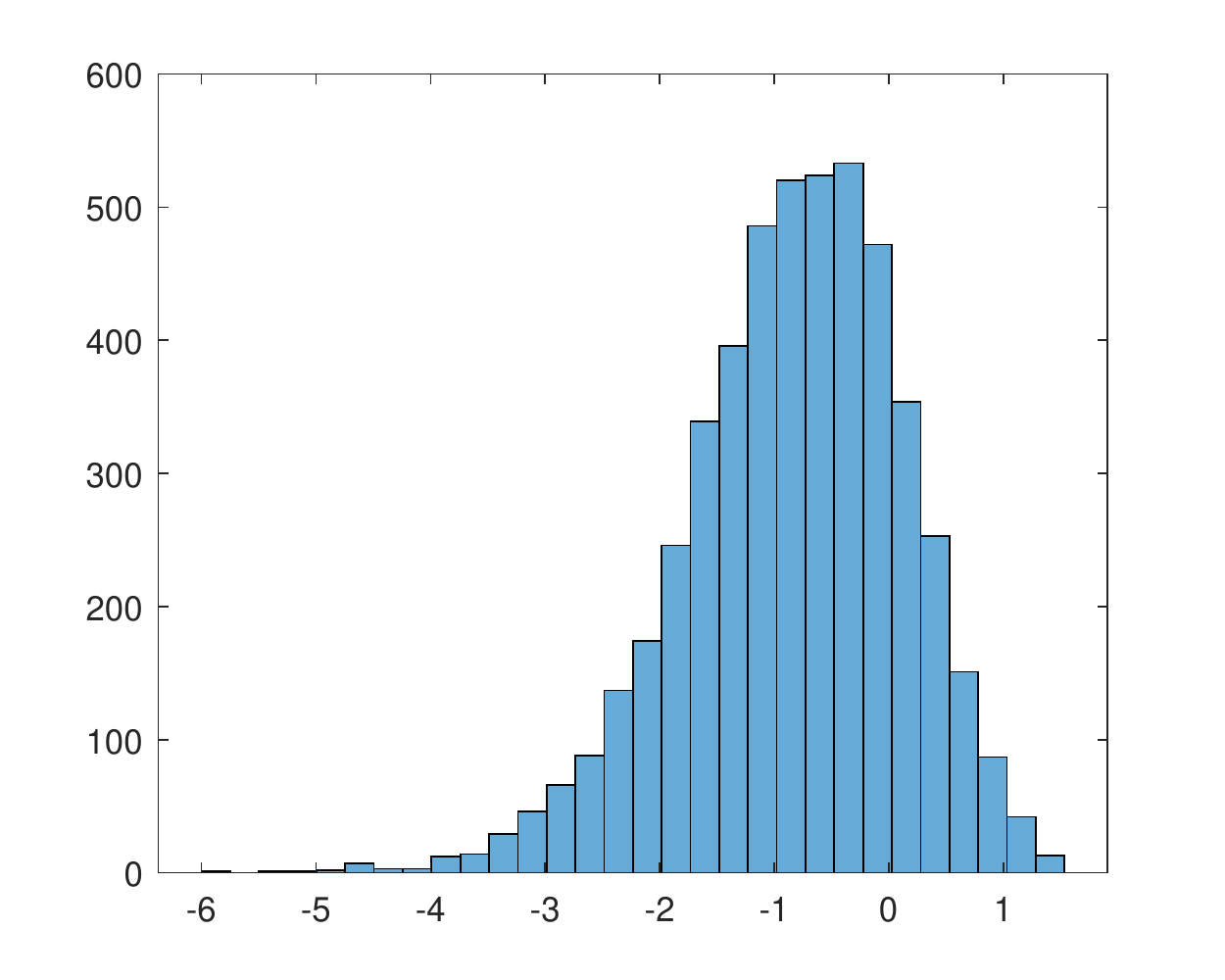}
		\caption{Exact distribution}
	\end{subfigure}\qquad
	\begin{subfigure}[b]{0.3\textwidth}
		\includegraphics[width=\textwidth]{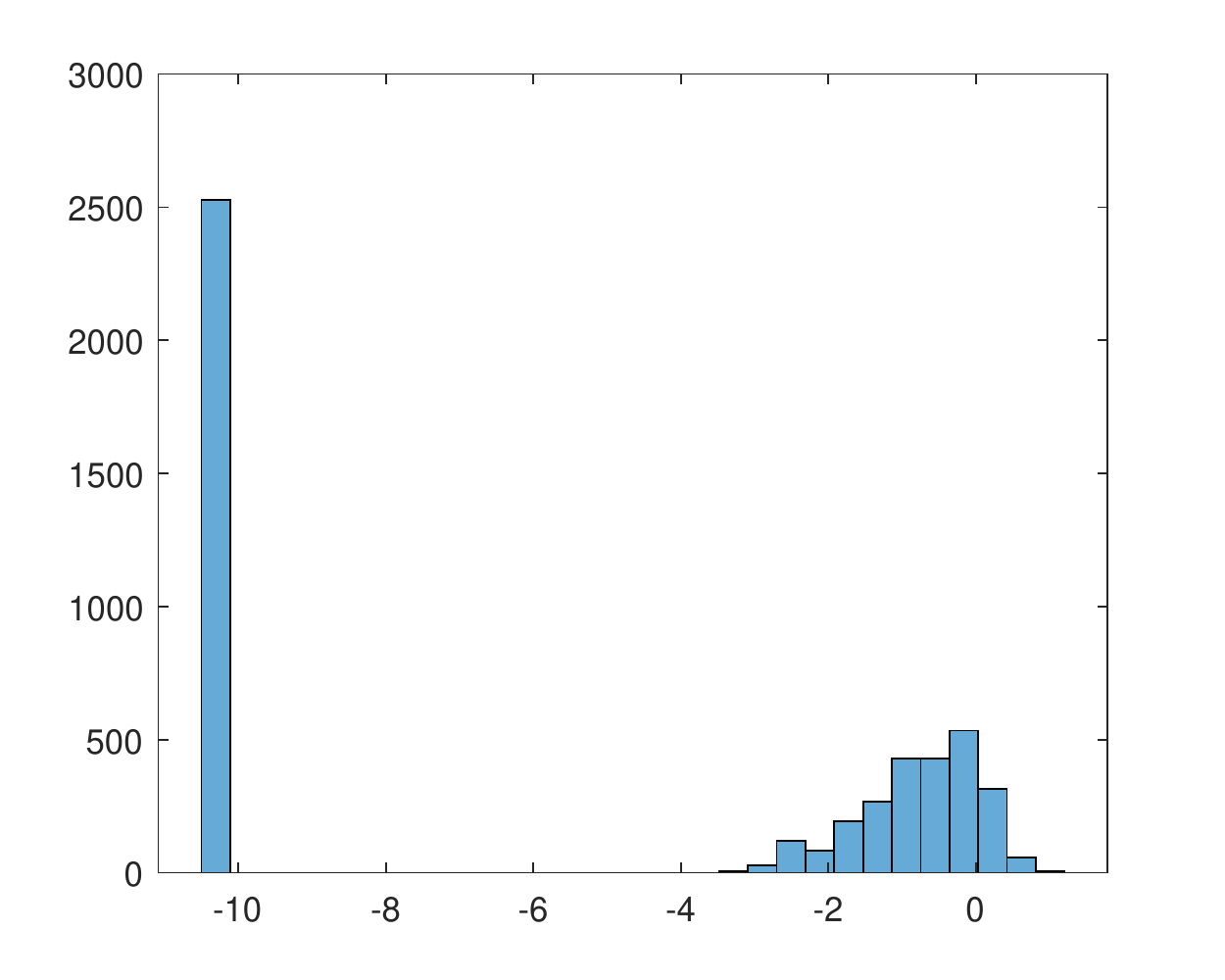}
		\caption{Naive bootstrap}
	\end{subfigure}\quad
	\begin{subfigure}[b]{0.3\textwidth}
		\includegraphics[width=\textwidth]{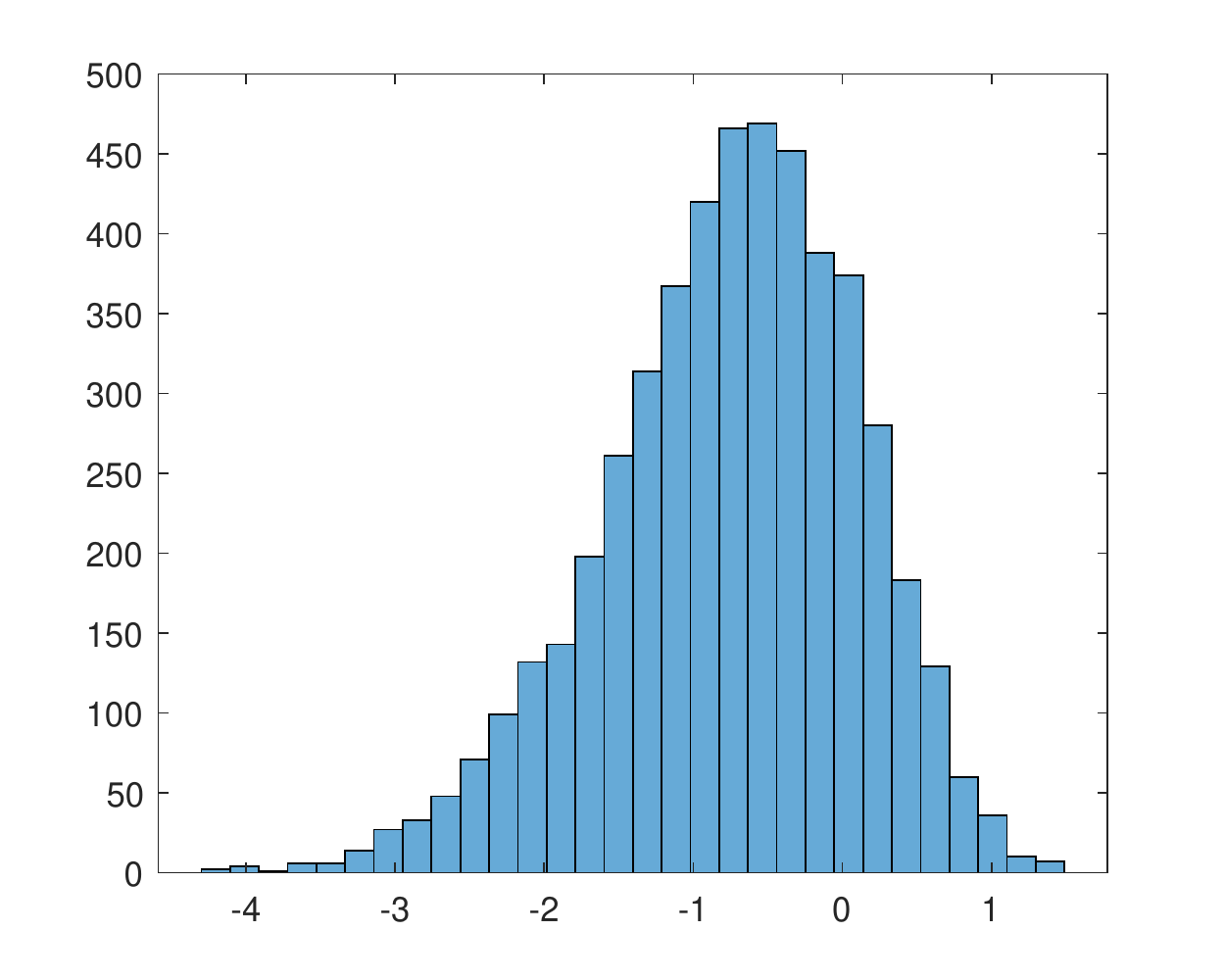}
		\caption{Trimmed bootstrap}
	\end{subfigure}
	\caption{Finite sample distribution and the bootstrap distribution. Sample size: $n=200$ in panels (a)-(c) and $n=1,000$ in panels (d)-(f). 5,000 MC experiments.}
	\label{fig:mc_bootstrap}
\end{figure}

\section{Empirical illustration}\label{sec:empirical_application}
Do incumbents have any electoral advantage? An extensive literature, going back at least to \cite{cummings1966}, aims to answer this question. Estimating the causal effect is elusive because incumbents, by definition, are more successful politicians. Using the regression discontinuity design, \cite{lee2008randomized}, documents that for the U.S. Congressional elections during 1946-1998, the incumbency advantage is 7.7\% of the votes share on the next election. The design is plausibly monotone, since we do not expect that candidates with a larger margin should have a smaller vote share on the next election, at least on average. The unconstrained regression estimates presented by \cite{lee2008randomized} also support the monotonicity empirically.

We revise the main finding of \cite{lee2008randomized} with our sharp iRDD estimator. The dataset is publicly available as a companion for the book by \cite{angrist2008mostly}. We use the rule-of-thumb choices of tuning parameters, $c=1$, $a=1/3$ (point estimation), and $a=1/2$ (inference). Figure~\ref{fig:incumbency} presents the isotonic regression estimates\footnote{We use the piecewise-constant interpolation, but a higher-order polynomial interpolation is another alternative that would produce visually more appealing estimates.} of the average vote share for the Democratic party at next elections as a function of the vote share margin at the previous election (left panel). There is a pronounced jump in average outcomes for Democrats who barely win the election, compared to the results for the penultimate election (right panel). We find the point estimate of 13.8\% with the 95\% confidence interval $[6.6\%,26.5\%]$. While we reject the hypothesis that the incumbency advantage did not exist, our confidence intervals give a wider range of estimates. Our confidence interval may be conservative if the underlying regression is two times differentiable, however, it is robust to the failure of this assumption as well as to the inference after the model selection issues.

Of course, different approaches work differently in finite samples and it is hard to have a definite comparison. \cite{lee2008randomized} estimates the causal effect by fitting parametric regressions with the global fourth-degree polynomial, which might be unstable at the boundary; see \cite{gelman2019high}. We, on the other hand, rely on the nonparametric boundary-corrected isotonic regression. Interestingly, the local linear estimator with the data-driven bandwidth parameter as implemented in the STATA package {\vtext} estimates the causal effect of 6.6\%. 

We also compute iRDD estimates using isotonic regressions without the boundary correction, and evaluating the regression function at the most extreme to the boundary observations. With this approach, we obtain point estimates of 6.6\%. As a consequence of the monotonicity, the isotonic regression estimator is biased upwards to the left of the cutoff and downwards to the right of the cutoff; see the Appendix, Section~\ref{sec:appendix_inconsistency}. Therefore, our iRDD estimators without the boundary correction provide a lower bound on the estimated causal effect, which might be of significant interest given that it is obtained completely tuning-free.

\begin{figure}[ht]
	\begin{subfigure}[b]{0.5\textwidth}
		\includegraphics[width=\textwidth]{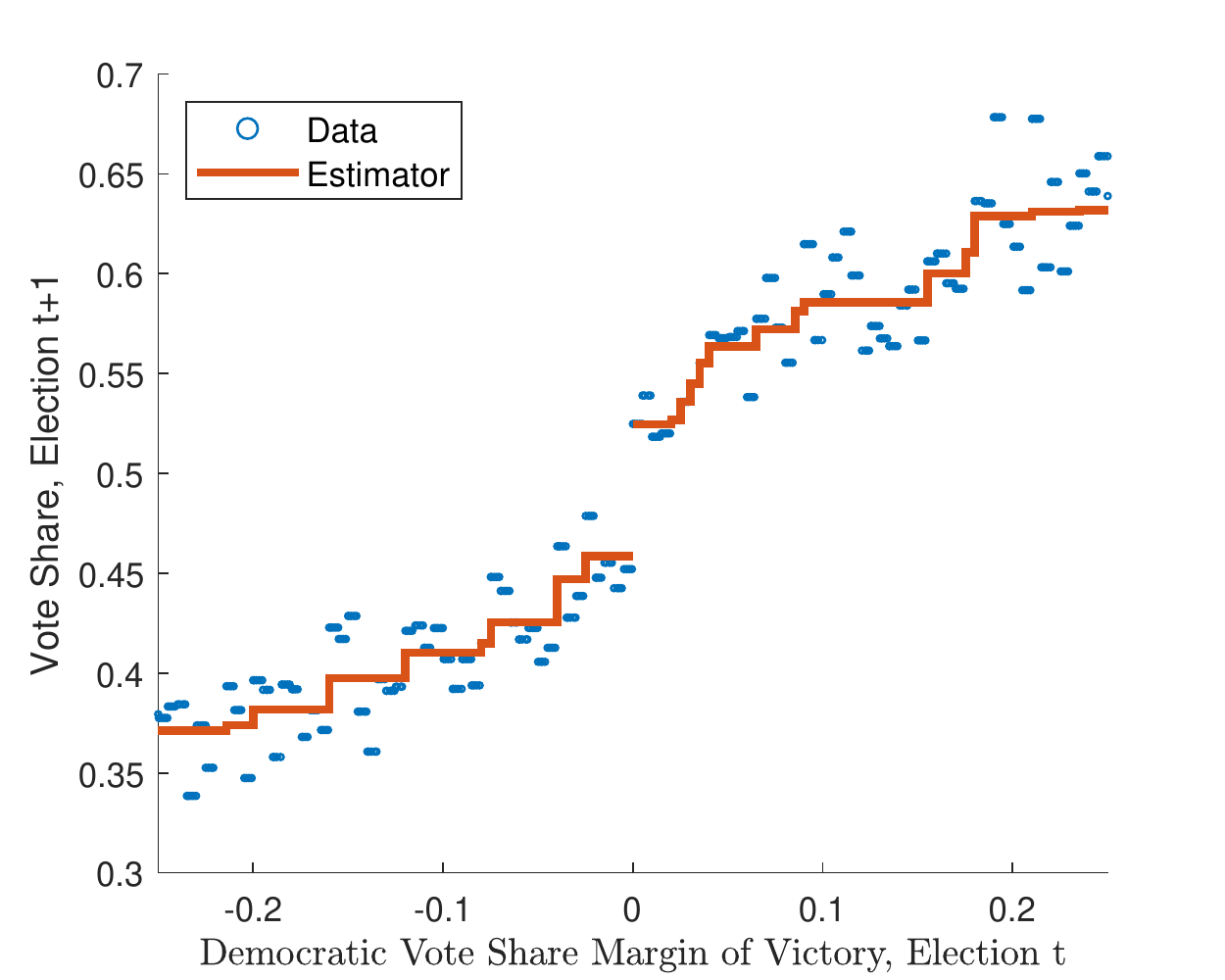}
	\end{subfigure}
	\begin{subfigure}[b]{0.5\textwidth}
		\includegraphics[width=\textwidth]{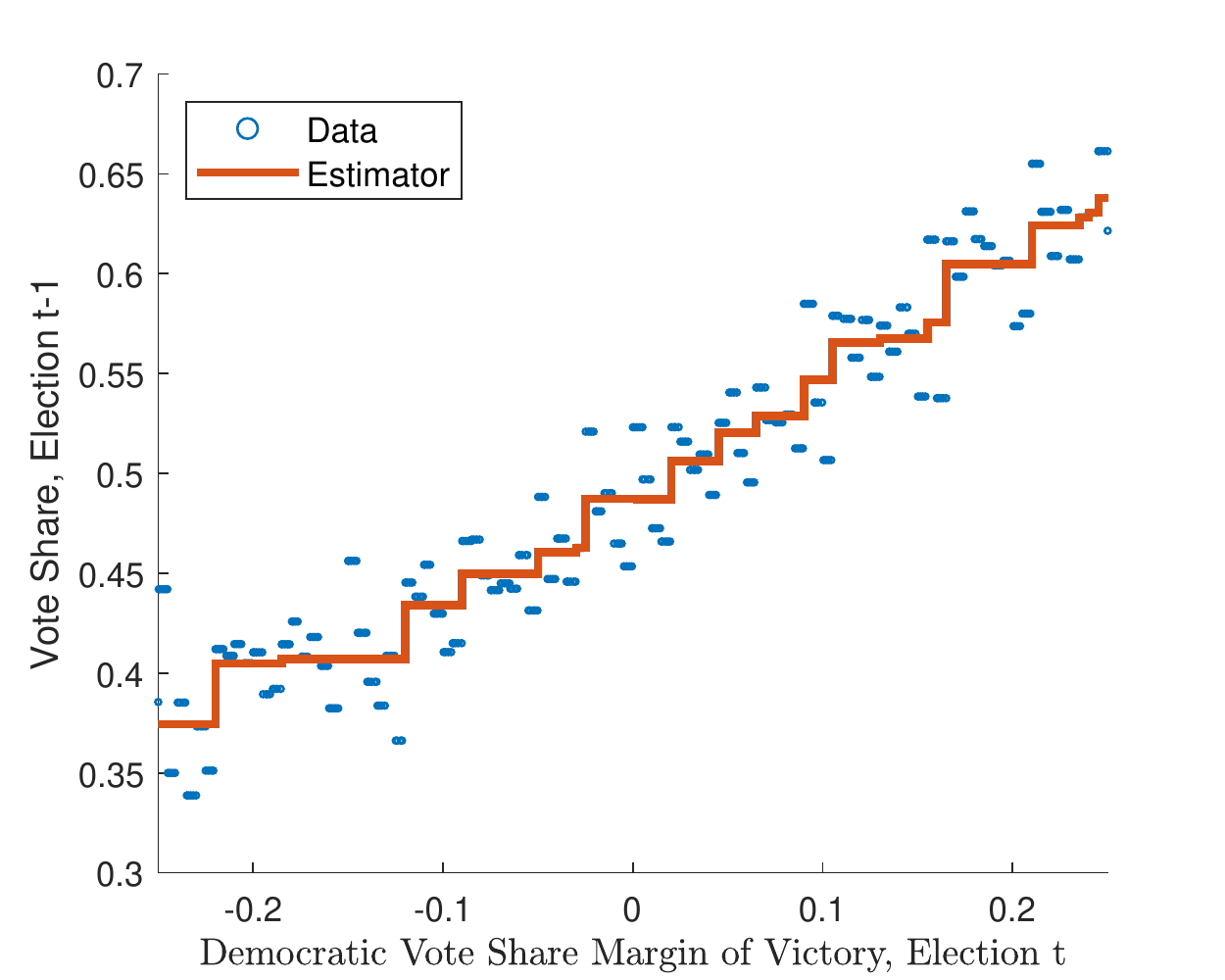}
	\end{subfigure}
	\caption{Incumbency advantage. Sample size: 6,559 observations with 3,819 observations below the cutoff.}
	\label{fig:incumbency}
\end{figure}

\section{Conclusion}\label{sec:conclusion}
This paper studies the isotonic regression estimator at the boundary of its support, an object that is particularly interesting and required in the analysis of monotone regression discontinuity designs. To the best of our knowledge, our paper is the first to develop the asymptotic theory for the isotonic regression at the boundary point. Building on these results, we offer a new perspective on monotone regression discontinuity designs and contribute more broadly to the growing literature on nonparametric identification, estimation, and inference under shape restrictions. 

While the present paper focuses on the isotonic regression discontinuity designs in the classical sharp and fuzzy setting, it also opens several directions for future research. First, since various functional relations in economics are expected to be monotone, the isotonic regression might be of interest in other applications beyond the regression discontinuity designs. Second, it could be interesting to investigate how the monotonicity and other shape restrictions can be incorporated in other regression discontinuity designs; see \cite{cattaneo2017} for a review. Designing the coverage-optimal choices of the boundary correction is another important point that should be addressed in the future research. Finally, in the large-sample approximations that we use for the inference, we do not assume the existence of derivatives and rely instead on the weaker one-sided H\"{o}lder continuity condition. This indicates that our results might be honest to the relevant H\"{o}lder class, but additional study is needed to confirm this conjecture. 

\section*{Acknowledgements}
We thank the Editor, the Associate Editor, and anonymous referees for comments that helped us to improve significantly the paper. We are also grateful to Alex Belloni, Federico Bugni, Matias Cattaneo, Xiaohong Chen, Peter Hansen, Jonathan Hill, Jia Li, Matt Masten, Andrew Patton, Andres Santos, Valentin Verdier, Jon Wellner, and other participants of various seminars and conferences for helpful comments and conversations. A special thanks goes to Kenichi Nagasawa for his insightful comments on the earlier draft. All remaining errors are ours.

\bibliography{references}

\newpage
\setcounter{page}{1}
\setcounter{section}{0}
\setcounter{equation}{0}
\setcounter{table}{0}
\setcounter{figure}{0}
\renewcommand{\theequation}{A.\arabic{equation}}
\renewcommand\thetable{A.\arabic{table}}
\renewcommand\thefigure{A.\arabic{figure}}
\renewcommand\thesection{A.\arabic{section}}
\renewcommand\thepage{Appendix - \arabic{page}}
\renewcommand\thetheorem{A.\arabic{theorem}}

\begin{center}
	{\LARGE\textbf{APPENDIX A}}	
\end{center}
\bigskip

\section{Inconsistency at the boundary}\label{sec:appendix_inconsistency}
Put
\begin{equation*}
F_n(t) = \frac{1}{n}\sum_{i=1}^n\one\{X_i\leq t\}\qquad \text{and}\qquad M_n(t) = \frac{1}{n}\sum_{i=1}^nY_i\one\{X_i\leq t\}.
\end{equation*}

By \cite{barlow1972statistical}, Theorem 1.1, $\hat m(x)$ is the left derivative of the greatest convex minorant of the cumulative sum diagram
\begin{equation*}
t\mapsto (F_n(t),M_n(t)),\qquad t\in[0,1]
\end{equation*}
at $t=x$; see Figure~\ref{fig:switch}. The natural estimator of $m(0)$ is $\hat m(X_{(1)})$, which corresponds to the slope of the first-segment
\begin{equation*}
\hat m(X_{(1)}) = \min_{1\leq i\leq n}\frac{1}{i}\sum_{j=1}^iY_{(j)},
\end{equation*}
where $X_{(j)}$ is the $j^{th}$ order statistics and $Y_{(j)}$ is the corresponding observation of the dependent variable.

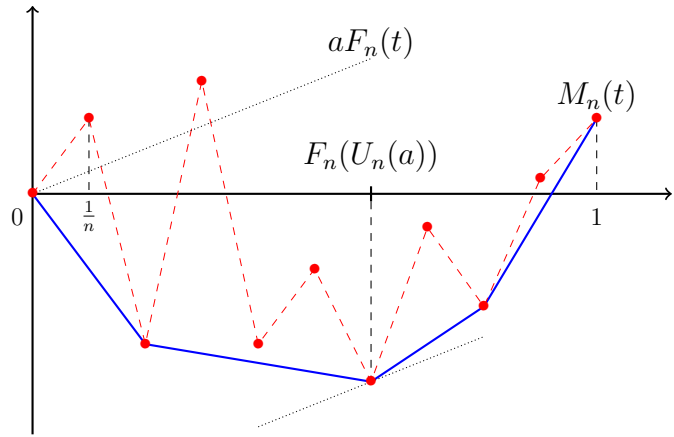
\begin{figure}[hb]
	\centering
	\begin{tikzpicture}
	\draw[thick,->] (0,0.8) -- (0,6.5);
	\draw[thick,->] (0,4) -- (8.5,4);
	\draw[thick, blue] (0,4) -- (1.5,2) -- (3,1.75) -- (4.5,1.5) -- (6,2.5) --  (7.5,5);
	\draw[line width=0.01mm, red,dashed] (0,4) -- (.75,5) -- (1.5,2) -- (2.25,5.5) -- (3,2) -- (3.75,3) -- (4.5,1.5)--(5.25,3.55)--(6,2.5)--(6.75,4.2)--(7.5,5);
	\draw [densely dotted] (3,.9) -- (4.5,1.5)--(6,2.1);
	\draw [thick] (4.5,4.1) -- (4.5,3.9);
	\draw [densely dotted] (0,4) -- (1.5,4.6)--(3,5.2)--(4.5,5.8);
	\draw [dashed] (4.5,3.9) -- (4.5,1.5);
	\draw [dashed] (7.5,5) -- (7.5,4);
	\draw [dashed] (.75,5) -- (0.75,4);
	\foreach \Point in {(0,4), (.75,5), (1.5,2), (2.25,5.5), (3,2), (3.75,3), (4.5,1.5), (5.25,3.55), (6,2.5), (6.75,4.2), (7.5,5)}{
		\node[red, scale=0.8] at \Point {\textbullet};
	}
	\node at (7.5,5.3) {$M_n(t)$};
	\node[scale=0.8] at (7.5,3.7) {$1$};
	\node[scale=0.8] at (0.75,3.7) {$\frac{1}{n}$};
	\node[scale=0.8] at (-0.2,3.7) {$0$};
	\node at (4.5,6) {$aF_n(t)$};
	\node at (4.5,4.5) { $F_n(U_n(a))$};
	\end{tikzpicture}
	\caption{If $\hat m(x)\leq a$, then the largest distance between the line with slope $a$ going through the origin and the greatest convex minorant (broken blue line) of the cumulative sum diagram $t\mapsto (F_n(t),M_n(t))$ (red dots) will be achieved at some point to the right of $F_n(x)$.}
	\label{fig:switch}
\end{figure}

The following result shows the isotonic regression estimator is can be inconsistent at the boundary.\footnote{We are grateful to the anonymous referee for suggesting how to improve the statement of this result.} We believe that the exact asymptotic behavior of $\hat m(X_{(1)})$ is an interesting problem to study on its own; see \cite{balabdaoui2011grenander} for such analysis in the case of the Grenander estimator. 
\begin{theorem}\label{thm:inconsistency}
	Suppose that the conditional distribution $\varepsilon|X=0$ is not degenerate and that $x\mapsto \Pr(Y\leq y|X=x)$ is continuous for every $y$. Then there exists $\epsilon>0$ such that
	\begin{equation*}
	\liminf_{n\to\infty}\Pr(|\hat m(X_{(1)}) - m(0)|>\epsilon) >0.
	\end{equation*}
\end{theorem}
\begin{proof}	
	Let $F_{\varepsilon|X}$ be the conditional CDF of $\varepsilon|X$. Note that $\int_{-\infty}^0 e\dx F_{\varepsilon|X}(e|0)<0$ since $\E[\varepsilon|X]=0$ and the distribution $\varepsilon|X=0$ is not degenerate. Then, by continuity $\exists\epsilon>0$ such that $\int_{-\infty}^{-\epsilon}e\dx F_{\varepsilon|X}(e|0)<0$. This shows that $F_{\varepsilon|X}(-\epsilon|0)>0$. 	Next, for any $\epsilon>0$
	\begin{equation*}
	\begin{aligned}
	\Pr(|\hat m(X_{(1)}) - m(0)|>\epsilon) & \geq \Pr\left(\min_{1\leq i\leq n}\frac{1}{i}\sum_{j=1}^iY_{(j)} < m(0)-\epsilon\right) \\
	& \geq \Pr(Y_{(1)} < m(0) - \epsilon) \\
	& = \int \Pr(Y\leq m(0)-\epsilon |X=x)\dx F_{X_{(1)}}(x) \\
	& \to \Pr(Y\leq m(0) - \epsilon|X=0) \\
	& = F_{\varepsilon|X=0}(-\epsilon) \\
	& > 0
	\end{aligned}
	\end{equation*}
	where we use the fact that $X_{(1)}\xrightarrow{d} 0$. Therefore, there exists $\epsilon>0$ such that
	\begin{equation*}
		\liminf_{n\to\infty}\Pr(|\hat m(X_{(1)}) - m(0)|>\epsilon) >0.
	\end{equation*}	
\end{proof}

The exact asymptotic behavior of $\hat m(X_{(1)})$ is not clear, but we know that it will underestimate $m(0)$ in general. To see this, note that
\begin{equation*}
\begin{aligned}
	\Pr\left(\hat m(X_{(1)} ) \leq m(0) \right) & \geq \Pr(\hat m(cn^{-1/3}) \leq m(0)) \\
	& \to 1,
\end{aligned}
\end{equation*}
where the first inequality follows from $X_{(1)}\leq cn^{-1/3}$ for some $c>0$ and monotonicity and the second by the consistency of $\hat m(cn^{-1/3})$; see Theorem~\ref{thm:isotonic_clt} (ii). Therefore, 
\begin{equation*}
	\lim_{n\to\infty}\Pr\left(\hat m(X_{(1)} ) \leq m(0) \right) = 1
\end{equation*}
and for large $n$, $\hat m(X_{(1)})$ will be smaller than $m(0)$ with high probability. Similarly, one can show that $\hat m(X_{(n)})$ overestimates $1$.

\section{Proofs of main results}\label{sec:proofs}
\begin{proof}[Proof of Theorem~\ref{thm:isotonic_clt}]
	By \cite{barlow1972statistical}, Theorem 1.1, $\hat m(x)$ is the left derivative of the greatest convex minorant of the cumulative sum diagram
	\begin{equation*}
	t\mapsto (F_n(t),M_n(t)),\qquad t\in[0,1]
	\end{equation*}
	at $t=x$. This corresponds to the piecewise-constant left-continuous interpolation. Put
	\begin{equation*}
	U_n(a) = \argmax_{s\in[0,1]}\left\{aF_n(s) - M_n(s)\right\}.
	\end{equation*}
	Then for any\footnote{For a continuous function $\Phi:[0,1]\to\R$, we define $\argmax_{t\in[0,1]}\Phi(t) = \max\{s\in[0,1]:\; \Phi(s) = \max_{t\in[0,1]}\Phi(t)\}$ to accomodate non-unique maximizers. Recall that continuous function attains its maximum on compact intervals and its argmax is a closed set with a well-defined maximal element.} $x\in(0,1)$ and $a\in\R$
	\begin{equation}\label{eq:switching}
	\begin{aligned}
	\hat m(x)\leq a & \iff F_n(U_n(a)) \geq F_n(x) \\
	& \iff U_n(a) \geq x,
	\end{aligned}
	\end{equation}
	as can be seen from Figure~\ref{fig:switch}, see also \cite{groeneboom2014nonparametric}, Lemma 3.2.
	
	\paragraph{Case (i): $a\in(0,1/3)$.}	For every $u\in\R$
	\begin{equation*}\footnotesize
	\begin{aligned}
	& \Pr\left(n^{1/3}\left(\hat m\left(cn^{-a}\right) - m(cn^{-a})\right) \leq u\right) \\
	& = \Pr\left(\hat m\left(cn^{-a}\right)\leq n^{-1/3}u + m(cn^{-a})\right) \\
	& = \Pr\left(\argmax_{s\in[0,1]}\left\{\left(n^{-1/3}u + m(cn^{-a})\right)F_n(s) - M_n(s)\right\} \geq cn^{-a}\right) \\
	& = \Pr\left(\argmax_{t\in[-cn^{1/3-a},(1-cn^{-a})n^{1/3}]}\left\{\left(n^{-1/3}u + m(cn^{-a})\right)F_n(tn^{-1/3}+cn^{-a}) - M_n(tn^{-1/3}+cn^{-a})\right\} \geq 0\right),
	\end{aligned}
	\end{equation*}
	where the second equality follows by the switching relation in Eq.~(\ref{eq:switching}) and the last by the change of variables $s\mapsto tn^{-1/3} + cn^{-a}$.
	
	The location of the argmax is the same as the location of the argmax of the following process
	\begin{equation*}
	Z_{n1}(t) \triangleq I_{n1}(t) + II_{n1}(t) + III_{n1}(t)
	\end{equation*}
	due to scale and shift invariance with
	\begin{equation*}
	\begin{aligned}
	I_{n1}(t) & = \sqrt{n}(P_n - P)g_{n,t},\qquad g_{n,t}\in\mathcal{G}_{n1} \\
	II_{n1}(t) & = n^{2/3}\E\left[(m(cn^{-a}) - Y)\left(\one_{[0,tn^{-1/3}+cn^{-a}]}(X) - \one_{[0,cn^{-a}]}(X)\right)\right] \\
	III_{n1}(t) & = n^{1/3}u[F_n(tn^{-1/3}+cn^{-a})-F_n(cn^{-a}) ], \\
	\end{aligned}
	\end{equation*}
	where
	\begin{equation*}
	\mathcal{G}_{n1} = \left\{g_{n,t}(y,x) = n^{1/6}(m(cn^{-a}) - y)\left(\one_{[0,tn^{-1/3} + cn^{-a}]}(x)-\one_{[0,cn^{-a}]}(x)\right):\; t\in[-K,K] \right\}.
	\end{equation*}
	We will show that the process $Z_{n1}$ converges weakly to a non-degenerate Gaussian process in $\ell^\infty[-K,K]$ for every $K<\infty$. 
	
	Under Assumption~\ref{as:dgp} (ii)-(iii) the covariance structure of the process $I_{n1}$ converges pointwise to the one of the two-sided scaled Brownian motion (two independent Brownian motions starting from zero and running in the opposite directions). Indeed, when $s,t\geq0$
	\begin{equation*}
	\begin{aligned}
	\Cov(g_{n,t},g_{n,s}) & = n^{1/3}\E\left[|Y-m(cn^{-a})|^2\one_{[cn^{-a},cn^{-a}+n^{-1/3}(t\wedge s)]}(X) \right] + O(n^{-1/3}) \\
	& = n^{1/3}\E\left[(\varepsilon^2 + |m(X) - m(cn^{-a})|^2)\one_{[cn^{-a},cn^{-a}+n^{-1/3}(t\wedge s)]}(X)\right] + o(1) \\
	& = n^{1/3}\int_{cn^{-a}}^{cn^{-a}+n^{-1/3}(t\wedge s)}\left(\sigma^2(x) + |m(x) - m(cn^{-a})|^2\right)f(x)\dx x + o(1) \\
	& = \left(\sigma^2(\xi_n) + |m(cn^{-a}) - m(\xi_n)|^2\right)f(\xi_n)(t\wedge s) + o(1) \\
	& = \sigma^2(0)f(0)(s\wedge t) + o(1),
	\end{aligned}
	\end{equation*}
	where we use the mean-value theorem for some $\xi_n$ between $cn^{-a}$ and $cn^{-a}+n^{-1/3}(t\wedge s)$. Similarly, it can be shown that
	\begin{equation*}
	\Cov(g_{n,t},g_{n,s}) = \begin{cases}
	\sigma^2(0)f(0)(|s|\wedge |t|) + o(1) & s,t\leq 0 \\
	o(1) & \mathrm{sign}(s)\ne \mathrm{sign}(t).
	\end{cases}
	\end{equation*}
	
	The class $\mathcal{G}_{n1}$ is VC subgraph with VC index 2 and the envelop
	\begin{equation*}
	G_{n1}(y,x) = n^{1/6}|y-m(cn^{-a})|\one_{[cn^{-a}-n^{-1/3}K,cn^{-a}+n^{-1/3}K]}(x),
	\end{equation*}
	which is square integrable
	\begin{equation*}
	\begin{aligned}
	\E G_{n1}^2(Y,X) & = n^{1/3}\E[|Y-m(cn^{-a})|^2\one_{[cn^{-a}-n^{-1/3}K,cn^{-a}+n^{-1/3}K]}(X)] \\
	& = n^{1/3}\E[\left(\varepsilon^2 + |m(X)-m(cn^{-a})|^2\right)\one_{[cn^{-a}-n^{-1/3}K,cn^{-a}+n^{-1/3}K]}(X)] \\
	& = n^{1/3}\int_{cn^{-a}-n^{-1/3}K}^{cn^{-a}+n^{-1/3}K}\left(\sigma^2(x) + |m(x) - m(cn^{-a})|^2\right)f(x)\dx x \\
	& = O(1).
	\end{aligned}
	\end{equation*}
	Next, we verify the Lindeberg's condition under Assumption~\ref{as:dgp} (i)
	\begin{equation*}
	\begin{aligned}
	\E G_{n1}^2\one\{G_n>\eta\sqrt{n}\} & \leq \frac{\E G_n^{2+\delta}}{\eta^\delta n^{\delta/2}} \\	
	& = \frac{n^{(2+\delta)/6}}{\eta^\delta n^{\delta/2}}\E\left[|Y-m(cn^{-a})|^{2+\delta}\one_{[cn^{-a}-n^{-1/3}K,cn^{-a}+n^{-1/3}K]}(X)\right] \\
	& = \frac{n^{(2+\delta)/6}}{\eta^\delta n^{\delta/2}}O(n^{-1/3}) \\
	& = o(1).
	\end{aligned}
	\end{equation*}
	Lastly, under Assumption~\ref{as:dgp} (iii), for every $\delta_n\to0$
	\begin{equation*}
	\begin{aligned}
	\sup_{|t-s|\leq\delta_n} \E|g_{n,t} - g_{n,s}|^2 & = n^{1/3}\sup_{|t-s|\leq\delta_n}\E\left[\left|Y-m(cn^{-a})\right|^2\one_{[cn^{-a}+n^{-1/3}t,cn^{-a}+n^{-1/3}s]}(X)\right] \\
	& = n^{1/3}\sup_{|t-s|\leq\delta_n}\E\left[\left(\varepsilon^2 + |m(X)-m(cn^{-a})|^2\right)\one_{[cn^{-a}+n^{-1/3}t,cn^{-a}+n^{-1/3}s]}(X)\right] \\
	& = O(\delta_n) \\
	& = o(1).
	\end{aligned}
	\end{equation*}
	Therefore, by \cite{van2000weak}, Theorem 2.11.22
	\begin{equation*}
	I_{n1}(t) \leadsto \sqrt{\sigma^2(0)f(0)}W_t\qquad \text{in}\qquad \ell^\infty[-K,K].
	\end{equation*}
	Under Assumption~\ref{as:dgp} (ii) and (iv) by Taylor's theorem
	\begin{equation*}
	\begin{aligned}
	II_{n1}(t) & = n^{2/3}\E\left[(m(cn^{-a}) - Y)\left(\one_{[0,cn^{-a}+tn^{-1/3}]}(X) - \one_{[0,cn^{-a}]}(X)\right)\right] \\
	& = n^{2/3}\int_{F(cn^{-a})}^{F(cn^{-a} + tn^{-1/3})}\left(m(cn^{-a}) - m(F^{-1}(u))\right)\dx u \\
	& = -n^{2/3}\frac{m'(0)}{2f(0)}(1+o(1))[F(cn^{-a} + tn^{-1/3})-F(cn^{-a})]^2 \\
	& = - \frac{t^2m'(0)}{2f(0)}[f(0)]^2(1+o(1)) \\
	& = -\frac{t^2}{2}m'(0)f(0) + o(1)
	\end{aligned}
	\end{equation*}
	uniformly over $[-K,K]$. Lastly, by the uniform law of the iterated logarithm
	\begin{equation*}
	\begin{aligned}
	III_{n1}(t) & =   n^{1/3}u[F(tn^{-1/3}+cn^{-a})-F(cn^{-a})] + o(1) \\
	& = utf(0) + o(1)
	\end{aligned}
	\end{equation*}
	uniformly over $t\in[-K,K]$. Therefore, for every $K<\infty$
	\begin{equation}\label{eq:a>1/3}
	Z_{n1}(t) \leadsto utf(0) - \sqrt{\sigma^2(0)f(0)}W_t - \frac{t^2}{2}m'(0)f(0) \triangleq Z_1(t),\quad \text{in}\quad \ell^\infty[-K,K].
	\end{equation}
	Next, we verify conditions of the argmax continuous mapping theorem \cite{kim1990cube}, Theorem 2.7. First, note that since
	\begin{equation*}
	\Var(Z_1(s) - Z_1(t))=\sigma^2(0)f(0)|t-s| \ne 0,\qquad \forall t\ne s,
	\end{equation*}
	by \cite{kim1990cube}, Lemma 2.6, the process $t\mapsto Z_1(t)$ achieves its maximum a.s. at a unique point. Second, by law of iterated logarithm for the Brownian motion
	\begin{equation*}
	\limsup_{t\to\infty}\frac{W_t}{\sqrt{2t\log\log t}} = 1,\qquad \mathrm{a.s.}
	\end{equation*}
	which shows that the quadratic term dominates asymptotically, i.e., $Z_1(t) \to -\infty$ as $|t|\to \infty$. It follows that the maximizer of $t\mapsto Z_1(t)$ is tight. Lastly, by Lemma~\ref{lemma:tightness} in the Supplementary Material, the argmax of $t\mapsto Z_{n1}(t)$ is uniformly tight. Therefore, by the argmax continuous mapping theorem, see \cite{kim1990cube}, Theorem 2.7
	\begin{equation*}
	\begin{aligned}
	& \Pr\left(n^{1/3}\left(\hat m\left(cn^{-a}\right) - m(0)\right) \leq u\right) \\
	& \to 
	\Pr\left(\argmax_{t\in\R}Z_{1}(t) \geq 0\right) \\
	& = \Pr\left(\argmax_{t\in\R}\left\{ut - \sqrt{\frac{\sigma^2(0)}{f(0)}}W_t -  \frac{t^2}{2}m'(0) \right\} \geq 0 \right).
	\end{aligned}
	\end{equation*}
	By the change of variables $t\mapsto \left(\frac{a}{b}\right)^{2/3}s + \frac{c}{2b}$, scale invariance of the Brownian motion $W_{\sigma^2t-\mu} = \sigma W_t - W_\mu$, and scale and shift invariance of the argmax
	\begin{equation*}
	\argmax_{t\in\R}\left\{aW_t - bt^2 + ct \right\} = \left(\frac{a}{b}\right)^{2/3}\argmax_{s\in\R}\{W_s - s^2\} + \frac{c}{2b}.
	\end{equation*}
	This allows us to simplify the limiting distribution as
	\begin{equation*}
	\begin{aligned}
	\Pr\left(n^{1/3}\left(\hat m\left(cn^{-a}\right) - m(0)\right) \leq u\right) & \to \Pr\left(\argmax_{t\in\R}\left\{ut - \sqrt{\frac{\sigma^2(0)}{f(0)}}W_t -  \frac{t^2}{2}m'(0) \right\} \geq 0 \right) \\
	& = \Pr\left(\left|\frac{4m'(0)\sigma^2(0)}{f(0)}\right|^{1/3}\argmax_{s\in\R}\{W_s - s^2\} \leq u\right),
	\end{aligned}
	\end{equation*}
	where we the use symmetry of the objective function.
	
	\paragraph{Case (ii): $a\in[1/3,1)$.} For every $u\in\R$
	\begin{equation*}
	\begin{aligned}
	& \Pr\left(n^{(1-a)/2}\left(\hat m\left(cn^{-a}\right) - m(0)\right) \leq u\right) \\
	& = \Pr\left(\hat m\left(cn^{-a}\right)\leq m(0) + n^{(a-1)/2}u\right) \\
	& = \Pr\left(\argmax_{s\in[0,1]}\left\{\left(n^{(a-1)/2}u + m(0)\right)F_n(s) - M_n(s)\right\} \geq cn^{-a}\right) \\
	& = \Pr\left(\argmax_{t\in [0,n^a/c]}\left\{(n^{(a-1)/2}u + m(0))F_n(cn^{-a}t) - M_n(cn^{-a}t)\right\} \geq 1\right), \\
	\end{aligned}
	\end{equation*}
	where the second equality follows by the switching relation in Eq.~(\ref{eq:switching}) and the last by the change of variables $s\mapsto cn^{-a}t$. 
	
	The location of the argmax is the same as the location of the argmax of the following process
	\begin{equation*}
	Z_{n2}(t) \triangleq I_{n2}(t) + II_{n2}(t) + III_{n2}(t) + IV_{n2}(t)
	\end{equation*}
	with
	\begin{equation*}
	\begin{aligned}
	I_{n2}(t) & = \sqrt{n}(P_n - P)g_{n,t},\quad g_{n,t}\in\mathcal{G}_{n2} \\
	II_{n2}(t) & = n^{(a+1)/2}\E[(m(0)-Y)\one_{[0,cn^{-a}t]}(X)] \\
	III_{n2}(t) & = n^{a}u(F_n(cn^{-a}t) - F(cn^{-a}t)) \\
	IV_{n2}(t) & = n^{a}uF(cn^{-a}t),
	\end{aligned}
	\end{equation*}
	where
	\begin{equation*}
	\mathcal{G}_{n2} = \left\{g_{n,t}(y,x) = n^{a/2}(m(0) - y)\one_{[0,cn^{-a}t]}(x):\quad t\in[0,K] \right\}.
	\end{equation*}
	We will show that the process $Z_{n2}$ converges weakly to a non-degenerate Gaussian process in $\ell^\infty[0,K]$ for every $K<\infty$.
	
	Under Assumption~\ref{as:dgp} (ii)-(iii) the covariance structure of the process $I_{n2}$ converges pointwise to the one of the scaled Brownian motion
	\begin{equation*}
	\begin{aligned}
	\Cov(g_{n,t},g_{n,s}) & = n^{a}\E\left[|Y-m(0)|^2\one_{[0,cn^{-a}(t\wedge s)]}(X) \right] + o(1)\\
	& = n^a\E[\varepsilon^2\one_{[0,cn^{-a}(t\wedge s)]}(X)] + n^{a}\E[|m(X)-m(0)|^2\one_{[0,cn^{-a}(t\wedge s)]}(X)]  + o(1) \\
	& = n^{a}\int_0^{cn^{-a}(t\wedge s)}\sigma^2(x)\dx F(x) + n^a\int_0^{cn^{-a}(t\wedge s)}|m(x) - m(0)|^2\dx F(x) +  o(1) \\
	& = \sigma^2(0)f(0)c(s\wedge t) + o(1).
	\end{aligned}
	\end{equation*}
	The class $\mathcal{G}_{n2}$ is VC subgraph with VC index 2 and envelop
	\begin{equation*}
	G_{n2}(y,x) = n^{a/2}|y-m(0)|\one_{[0,cn^{-a}K]}(x),
	\end{equation*}
	which is square integrable
	\begin{equation*}
	\begin{aligned}
	P G_{n2}^2 & = n^{a}\E[|Y-m(0)|^2\one_{[0,cn^{-a}K]}(X)] \\
	& = n^{a}\E[\varepsilon^2\one_{[0,cn^{-a}K]}(X)] + n^{a}\E[|m(X)-m(0)|^2\one_{[0,cn^{-a}K]}(X)] \\
	& = n^{a}\int_0^{n^{-a}K}\sigma^2(x)\dx F(x) + o(1)\\
	& = O(1).
	\end{aligned}
	\end{equation*}
	Next, we verify the Lindeberg's condition under Assumption~\ref{as:dgp} (i)
	\begin{equation*}
	\begin{aligned}
	\E G_{n2}^2\one\{G_n>\eta\sqrt{n}\} & \leq \frac{\E G_{n2}^{2+\delta}}{\eta^\delta n^{\delta/2}} \\
	& = \frac{n^{\frac{(2+\delta)a}{2}}}{\eta^\delta n^{\delta/2}}\E[|Y - m(0)|^{2+\delta}\one_{[0,cn^{-a}K]}(X)] \\
	& = \frac{n^{\frac{(2+\delta)a}{2}}}{\eta^\delta n^{\delta/2}}O(n^{-a}) = O(n^{-\delta(1-a)/2}) \\
	& = o(1).
	\end{aligned}
	\end{equation*}
	Lastly, under Assumption~\ref{as:dgp} (iii), for every $\delta_n\to0$
	\begin{equation*}
	\begin{aligned}
	\sup_{|t-s|\leq\delta_n} \E|g_{n,t} - g_{n,s}|^2 & = n^{a}\sup_{|t-s|\leq\delta_n}\E\left[\left|Y-m(0)\right|^2\one_{[cn^{-a}t,cn^{-a}s]}(X)\right] \\
	& = n^{a}\sup_{|t-s|\leq\delta_n}\E\left[\varepsilon^2\one_{[cn^{-a}t,cn^{-a}s]}(X)\right] + \E\left[|m(X)-m(0)|^2\one_{[cn^{-a}t,cn^{-a}s]}(X)\right] \\
	& = O(\delta_n) \\
	& = o(1).
	\end{aligned}
	\end{equation*}
	Therefore, by \cite{van2000weak}, Theorem 2.11.22
	\begin{equation*}
	I_{n2}(t) \leadsto \sqrt{\sigma^2(0)f(0)c}W_t\qquad \text{in}\qquad \ell^\infty[0,K].
	\end{equation*}
	Next,
	\begin{equation*}
		II_{n2}(t) = n^{(a+1)/2}\int_0^{F(cn^{-a}t)}(m(0) - m(F^{-1}(u)))\dx u.
	\end{equation*}
	For $a=1/3$, under Assumption~\ref{as:dgp} (iv), by Taylor's theorem  uniformly over $t$ on compact sets
	\begin{equation*}
	\begin{aligned}
	II_{n2}(t) & = -n^{(a+1)/2}\frac{1}{2}\frac{m'(F^{-1}(0))}{f(F(0))}[F(cn^{-a}t)]^2(1 + o(1)) \\
	& = -n^{(1-3a)/2}\frac{t^2c^2}{2}m'(0)f(0)(1+o(1)) \\
	& =	-\frac{t^2c^2}{2}m'(0)f(0) + o(1),
	\end{aligned}
	\end{equation*}
	while for $a\in(1/3,1)$ under $\gamma$-H\"{o}lder continuity of $m$ since $\gamma>(1-a)/2a$
	\begin{equation}\label{eq:quadratic_term}
	\begin{aligned}
		II_{n2}(t) & = n^{(a+1)/2}\int_0^{cn^{-a}t}(m(0) - m(x))f(x)\dx x \\
		& \lesssim  n^{(a+1)/2}\int_0^{cn^{-a}t}|x|^\gamma\dx x  \\
		& = O\left(n^{\frac{1-2a\gamma-a}{2}}\right) \\
		& = o(1).
	\end{aligned}
	\end{equation}
	Next, by the maximal inequality \cite{kim1990cube}, p.199,
	\begin{equation}\label{eq:max_ineq_Fn}
	\E\left[\sup_{t\in[0,K]}|F_n(cn^{-a}t) - F(cn^{-a}t)|\right] \leq n^{-1/2}J(1)\sqrt{PG^2_n},
	\end{equation}
	where $J(1)<\infty$ is the uniform entropy integral and
	\begin{equation*}
	\begin{aligned}
	PG^2_n & = F(cn^{-a}t) \\
	& = f(0)(1+o(1))cn^{-a}.
	\end{aligned}
	\end{equation*}
	Since $a<1$
	\begin{equation*}
	III_{n2}(t) = O_P(n^{(a-1)/2}) = o_P(1).
	\end{equation*}
	uniformly over $[0,K]$.
	
	Lastly,
	\begin{equation*}
	\begin{aligned}
	IV_{n2}(t) & = n^auF(cn^{-a}t) \\
	& = utf(0)c + o(1)
	\end{aligned}
	\end{equation*}
	uniformly over $t\in[0,K]$. Therefore, for every $K<\infty$
	\begin{equation}\label{eq:a=1/3}
	Z_{n2}(t) \leadsto utf(0)c + \sqrt{\sigma^2(0)f(0)c}W_t - \frac{t^2c^2}{2}m'(0)f(0)\one_{a=1/3} \triangleq Z_2(t)\quad \text{in}\quad \ell^\infty[0,K].
	\end{equation}
	Next, we extend processes $Z_{n2}$ and $Z_2$ to the entire real line as follows
	\begin{equation*}
	\tilde Z_{n2}(t) = \begin{cases}
	Z_{n2}(t), & t\geq 0 \\
	t, & t < 0
	\end{cases},
	\qquad 
	\tilde Z_2(t) = \begin{cases}
	Z_2(t), & t\geq 0, \\
	t, & t < 0.
	\end{cases}
	\end{equation*}
	We verify conditions of the argmax continuous mapping theorem \cite{kim1990cube}, Theorem 2.7. First, by the similar argument as before, the argmax of $t\mapsto \tilde Z_2(t)$ is unique and tight. Second, by Lemma~\ref{lemma:tightness} in the Supplementary Material, the argmax of $\tilde Z_{n2}$ is uniformly tight for every $u\in\R$ when $a=1/3$ and for every $u<0$ when $a\in(1/3,1)$. Therefore, by the argmax continuous mapping theorem, see \cite{kim1990cube}, Theorem 2.7,
	\begin{equation*}
	\begin{aligned}
	& \Pr\left(n^{(1-a)/2}\left(\hat m\left(cn^{-a}\right) - m(0)\right) \leq u\right) \\
	& \to 
	\Pr\left(\argmax_{t\in[0,\infty)}Z_{2}(t) \geq 1\right) \\
	& = \Pr\left(\argmax_{t\in[0,\infty)}\left\{ut - \sqrt{\frac{\sigma^2(0)}{cf(0)}}W_t -  \frac{t^2c}{2}m'(0)\one_{a=1/3} \right\} \geq 1 \right) \\
	& = \Pr\left(D^L_{[0,\infty)}\left(\sqrt{\frac{\sigma^2(0)}{cf(0)}}W_t + \frac{t^2c}{2}m'(0)\one_{a=1/3} \right)(1) \leq u\right) \\
	\end{aligned}
	\end{equation*}
	where the last line follows by the switching relation similar to the one in Eq.~(\ref{eq:switching}). 
	
	To conclude, it remains to show that when $a\in(1/3,1)$
	\begin{equation*}
	\Pr\left(n^{(1-a)/2}\left(\hat m\left(cn^{-a}\right) - m(0)\right) \geq 0\right) \to 0.
	\end{equation*}
	This follows since for every $\epsilon>0$
	\begin{align*}
	\Pr\left(n^{(1-a)/2}\left(\hat m\left(cn^{-a}\right) - m(0)\right) \geq 0\right)&\leq \Pr\left(n^{(1-a)/2}\left(\hat m\left(cn^{-a}\right) - m(0)\right) \geq -\epsilon\right)\\
	&\to\Pr\left(\argmax_{t\in[0,\infty)}\left\{-\epsilon t - \sqrt{\frac{\sigma^2(0)}{cf(0)}}W_t  \right\} \leq c\right) \\
	&=\Pr\left(\argmax_{t\in[0,\infty)}\left\{W_t-\sqrt{\frac{cf(0)}{\sigma^2(0)}}\epsilon t   \right\} \leq c\right) \\
	&=\Pr\left(\argmax_{t\in[0,\infty)}\left\{W_t- t   \right\} \leq {\frac{c^2f(0)}{\sigma^2(0)}}\epsilon\right), \\
	\end{align*}
	which tends to zero as $\epsilon\downarrow 0$ as can be seen from
	\begin{equation*}
	\limsup_{t\downarrow0}\frac{W_t}{\sqrt{2t\log\log (1/t)}} = 1,\qquad \mathrm{a.s.}
	\end{equation*}
\end{proof}

\begin{proof}[Proof of Theorem~\ref{thm:bootstrap}]
	Put
	\begin{equation*}
	M_n^*(t) = \frac{1}{n}\sum_{i=1}^nY_i^*\one\{X_i\leq t\}
	\end{equation*}
	and note that now $\hat m^*(x)$ is the left derivative of the greatest convex minorant of the cumulative sum diagram
	\begin{equation*}
	t\mapsto (F_n(t),M_n^*(t)),\qquad t\in[0,1]
	\end{equation*}
	at $t=x$. By the argument similar to the one used in the proof of Theorem~\ref{thm:isotonic_clt} for every $u<0$
	\begin{equation*}
	\begin{aligned}
	& \mathrm{Pr}^*\left(n^{(1-a)/2}\left(\hat m^*(cn^{-a}) - \hat m(cn^{-a})\right) \leq u\right) \\
	& = \mathrm{Pr}^*\left(\argmax_{t\in [0,n^a/c]}\left\{(n^{(a-1)/2}u + \hat m(cn^{-a}))F_n(cn^{-a}t) - M_n^*(cn^{-a}t)\right\} \geq 1\right).
	\end{aligned}
	\end{equation*}	
	The location of the argmax is the same as the location of the argmax of the following process
	\begin{equation*}
	Z_n^*(t) \triangleq I_n^*(t) + II_n^*(t) + III_n^*(t) + IV_{n}^*(t)
	\end{equation*}
	with
	\begin{equation*}
	\begin{aligned}
	I_n^*(t) & = -n^{(a-1)/2}\sum_{i=1}^n\eta_i^*\varepsilon_i\one_{[0,cn^{-a}t]}(X_i) \\
	II_n^*(t) & = n^{(a-1)/2}\sum_{i=1}^n\eta_i^*(\tilde m(X_i) - m(X_i))\one_{[0,cn^{-a}t]}(X_i) \\
	III_n^*(t) & = n^{(1+a)/2}\int_0^{cn^{-a}t}(\hat m(cn^{-a}) - \tilde m(x))\dx F_n(x) \\
	IV_n^*(t) & = n^{a}uF_n(cn^{-a}t) \\
	\end{aligned}
	\end{equation*}
	The process $I_n^*$ is the multiplier empirical process indexed by the class of functions
	\begin{equation*}
	\mathcal{G}_n = \left\{(\epsilon,x)\mapsto -n^{a/2}\epsilon\one_{[0,cn^{-a}t]}(x):\quad t\in[0,K] \right\}.
	\end{equation*}
	This class is of VC subgraph type with VC index 2 and envelop $G_n(\epsilon,x) = n^{a/2}|\epsilon|\one_{[0,cn^{-a}K]}(x)$, which is square-integrable
	\begin{equation*}
	PG_n^2 = n^a\int_0^{cn^{-a}K}\sigma^2(x)\dx F(x) = O(1).
	\end{equation*}
	This envelop satisfies Lindeberg's condition for every $\eta>0$
	\begin{equation*}
	\begin{aligned}
	\E G_n^2\{G_n >\eta\sqrt{n} \} & \leq \frac{n^{\frac{a(2+\delta)}{2}}\E\left[|\varepsilon|^{2+\delta}\one_{[0,cn^{-a}K]}(X)\right]}{\eta^\delta n^{\delta/2}} \\
	& = O(n^{\delta(a-1)/2}) \\
	& = o(1)
	\end{aligned}
	\end{equation*}
	and for every $g_{n,t},g_{n,s}\in\mathcal{G}_n$ and $\delta_n\to 0$
	\begin{equation*}
	\begin{aligned}
	\sup_{|t-s|\leq\delta_n} \E|g_{n,t} - g_{n,s}|^2 & = n^{a}\sup_{|t-s|\leq\delta_n}\E\left[\varepsilon^2\one_{[cn^{-a}s,cn^{-a}t]}(X)\right] \\
	& = O(\delta_n) \\
	& = o(1).
	\end{aligned}
	\end{equation*}
	
	Next, we show that the covariance structure is
	\begin{equation*}
	\begin{aligned}
	\E^*[I_n^*(t)I_n^*(s)] & = n^{a-1}\sum_{i=1}^n\varepsilon_i^2\one_{[0, cn^{-a}(t\wedge s)]}(X_i) \\
	& =  n^{a}\E[\varepsilon^2\one_{[0,cn^{-a}(t\wedge s)]}(X_i)] + R_n(t,s) \\
	\end{aligned}
	\end{equation*}
	with
	\begin{equation*}
	R_{n}(t,s) = n^{a-1}\sum_{i=1}^n\varepsilon_i^2\one_{[0, cn^{-a}(t\wedge s)]}(X_i) - n^{a}\E\left[\varepsilon^2\one_{[0,cn^{-a}(t\wedge s)]}(X_i)\right] \\
	\end{equation*}
	Since $\E[\varepsilon^4|X]\leq C$, the variance of $R_n$ tends to zero
	\begin{equation*}
	\begin{aligned}
	\Var(R_n(t,s)) & = n^{2a-1}\Var(\varepsilon^2\one_{[0,cn^{-a}(t\wedge s)]}(X)) \\
	& \leq n^{2a - 1}\E[\varepsilon^4\one_{[0,cn^{-a}(t\wedge s)]}(X))] \\
	& \leq Cn^{2a - 1}F(cn^{-a}(t\wedge s)) \\
	& = O(n^{a-1}) \\
	& = o(1),
	\end{aligned}
	\end{equation*}
	whence by Chebyshev's inequality $R_n(t,s) = o_P(1)$. Therefore, the covariance structure converges pointwise to the one of the scaled Brownian motion
	\begin{equation*}
	\begin{aligned}
	\E^*[I_n^*(t)I_n^*(s)] & = n^a\int_0^{cn^{-a}(t\wedge s)}\sigma^2(x)\dx F(x) + o_P(1) \\
	& = \sigma^2(0)f(0)c(t\wedge s) + o_P(1),
	\end{aligned}
	\end{equation*}
	where $\E^*[.] = \E[.|(Y_i,X_i)_{i=1}^\infty]$.	By \cite{van2000weak}, Theorem 2.11.22, the class $\mathcal{G}_n$ is Donsker, whence by the multiplier central limit theorem, \cite{van2000weak}, Theorem 2.9.6
	\begin{equation*}
	\sup_{h\in BL_1(\ell^\infty[0,K])}\left|\E^*h(I_n^*) - \E h\left(\sqrt{\sigma^2(0)f(0)c}W_t\right)\right| \xrightarrow{P} 0.
	\end{equation*}
	
	Next, $II_n^*$ is a multiplier empirical process indexed by the degenerate class of functions
	\begin{equation*}
	\mathcal{H}_n = \left\{x\mapsto n^{a/2}\left(\tilde m(x) - m(x)\right)\one_{[0,n^{-a}t]}(x):\; t\in [0,K] \right\}.
	\end{equation*}
	Since this class is of VC subgraph type with VC index 2, by the maximal inequality
	\begin{equation*}
	\begin{aligned}
	\E^*\left[\sup_{t\in[0,K]}|II_n^*(t)|\right] & \lesssim  \sqrt{n^a\int_0^{cn^{-a}K}|\tilde m(x) - m(x)|^2\dx F_n(x)} \\
	& = \sqrt{n^a\int_0^{K}|\tilde m(cn^{-a}y) - m(cn^{-a}y)|^2\dx F_n(cn^{-a}y)} \\
	& = \sqrt{o_P(1)n^aF_n(cn^{-a}K)} \\
	& = o_P(1),
	\end{aligned}
	\end{equation*}
	where we apply Proposition~\ref{prop:uniform} in the Supplementary Material.
	
	Next, changing variables $x\mapsto cn^{-a}y$ and using the fact that $\hat m$ is non-decreasing
	\begin{equation*}
	\begin{aligned}
	III_n^*(t) & = n^{(1+a)/2}\int_0^{t}(\hat m(cn^{-a}) - \tilde m(cn^{-a}y))\dx F_n(cn^{-a}y) \\
	& \leq n^{(1+a)/2}\int_0^1(\hat m(cn^{-a}) - \tilde m(cn^{-a}y))\dx F_n(cn^{-a}y) \\
	& \leq n^{(1+a)/2}\sup_{y\in[0,1]}|\hat m(cn^{-a}) - \tilde m(cn^{-a}y)| F_n(cn^{-a})\\
	& = o_P(1)\left(n^{(1+a)/2}(F_n(cn^{-a}) - F(cn^{-a})) + n^{(1+a)/2}F(cn^{-a}) \right) \\
	& = o_P(1)\left(O_P(1) + O(n^{(1-a)/2})\right) \\
	& = o_P(1),
	\end{aligned}
	\end{equation*}
	where the fourth line follows by Proposition~\ref{prop:uniform} in the Supplementary Material and Theorem~\ref{thm:isotonic_clt} (ii), and the fifth since the variance of the term inside is $O(1)$.
	
	Next
	\begin{equation*}
	IV_n^*(t) = utf(0)c + o_P(1)
	\end{equation*}
	in the same way we treat $III_{n2}+IV_{n2}$ in the proof of Theorem~\ref{thm:isotonic_clt}.
	
	Combining all estimates obtained above together
	\begin{align*}
	&\sup_{h\in BL_1(\ell^\infty[0,K])}\left|\E^*h(Z_n^*) - \E h\left(utf(0)c+\sqrt{\sigma^2(0)f(0)c}W_t\right)\right| \\
	&=	\sup_{h\in BL_1(\ell^\infty[0,K])}\left|\E^*h(Z_n^*)-\E^*h(utf(0)c+I_n^*)\right|\\
	&\quad+ \sup_{h\in BL_1(\ell^\infty[0,K])}\left|\E^* h(utf(0)c+I_n^*)- \E h\left(utf(0)c+\sqrt{\sigma^2(0)f(0)c}W_t\right)\right|\\
	&\leq 	\sup_{t\in[0,K]}\left|II_n^*(t) + III_n^*(t) + IV_{n}^*(t)\right|+ o_P(1) \\
	&=o_P(1).
	\end{align*}			
	
	By Lemma~\ref{lemma:tightness_bootstrap} in the Supplementary Material, the argmax of $Z_n^*(t)$ is uniformly tight, so by Lemma~\ref{lemma:argmax_conditional} in the Supplementary Material
	\begin{equation*}
	\mathrm{Pr}^*\left(n^{(1-a)/2}\left(\hat m^*(cn^{-a}) - \hat m(cn^{-a})\right) \leq u\right) \xrightarrow{P} \Pr\left(D^L_{[0,\infty)}\left(\sqrt{\frac{\sigma^2(0)}{cf(0)}}W_t\right)(1)\leq u\right).
	\end{equation*}
	To conclude, it remains to show that when $a\in(1/3,1)$
	\begin{equation*}
	\mathrm{Pr}^*\left(n^{(1-a)/2}\left(\hat m^*(cn^{-a}) - \hat m(cn^{-a})\right)  \geq 0\right) \to 0.
	\end{equation*}
	This follows since for every $\epsilon>0$
	\begin{align*}
	\mathrm{Pr}^*\left(n^{(1-a)/2}\left(\hat m^*(cn^{-a}) - \hat m(cn^{-a})\right)  \geq 0\right)&\leq \mathrm{Pr}^*\left(n^{(1-a)/2}\left(\hat m^*(cn^{-a}) - \hat m(cn^{-a})\right)  \geq -\epsilon\right)\\
	&\xrightarrow{P} \Pr\left(\argmax_{t\in[0,\infty)}\left\{-\epsilon t - \sqrt{\frac{\sigma^2(0)}{cf(0)}}W_t  \right\} \leq c\right) \\
	&=\Pr\left(\argmax_{t\in[0,\infty)}\left\{W_t- t   \right\} \leq c^2{\frac{f(0)}{\sigma^2(0)}}\epsilon\right), \\
	\end{align*}
	which tends to zero as $\epsilon\downarrow 0$; see the proof of Theorem~\ref{thm:isotonic_clt}. The result follows from Theorem~\ref{thm:isotonic_clt} (ii) with $a>1/3$.
\end{proof}

\begin{proof}[Proof of Theorem~\ref{thm:irdd_asymptotics}]
	Since
	\begin{equation*}
	n^{1/3}(\hat\theta - \theta) = n^{1/3}\left(\hat m_+(cn^{-1/3}) - m_+\right) - n^{1/3}\left(\hat m_-(-cn^{-1/3}) - m_-\right),
	\end{equation*}
	the proof is similar to the proof of Theorem~\ref{thm:isotonic_clt} and Remark~\ref{remark:right_boundary} with $a=1/3$. Strictly speaking, the proof of Theorem~\ref{thm:isotonic_clt} and Remark~\ref{remark:right_boundary} change a little. Now $F(0)\ne 0$ and we will have $\tilde F(x) = F(x)-F(0)$ and $\tilde F_n(x) = F_n(x) - F_n(0)$ instead of $F(x)$ and $F_n(x)$ everywhere in the proof of Theorem~\ref{thm:isotonic_clt} and $\tilde F(x) = F(0) - F(x)$ and $\tilde F_n(x) = F_n(0) - F_n(x)$ in the proof of Remark~\ref{remark:right_boundary}. Recall that if $X_n\si Y_n,\forall n\geq 1$ and $X\si Y$ are such that for $X_n\cw X$ and $Y_n\cw Y$, then $(X_n,Y_n)\cw (X,Y)$. In our case, the estimators $\hat m_+$ and $\hat m_-$ are independent by the independence of the two samples. This also implies that the processes $I_{n2}^+$ and $I_{n2}^-$ (i.e. $I_{n2}$ corresponding to $\hat m_+$ and for $\hat m_-$), are asymptotically uncorrelated, which implies that $W^+_t$ and $W^-_t$ are independent due to the fact that zero-mean Gaussian processes are completely characterized by their covariance function.
\end{proof}

\begin{proof}[Proof of Theorem~\ref{thm:irdd_fuzzy}]
	Given the proof of Theorem~\ref{thm:irdd_asymptotics}, it is enough to show the weak convergence of
	\begin{equation*}
		n^{1/3}\begin{pmatrix}
			\hat m_+(cn^{-1/3}) - m_+(0) \\
			\hat p_+(cn^{-1/3}) - p_+(0)
		\end{pmatrix}\qquad \text{and}\qquad n^{1/3}\begin{pmatrix}
		\hat m_-(cn^{-1/3}) - m_-(0) \\
		\hat p_-(cn^{-1/3}) - p_-(0)
		\end{pmatrix}.
	\end{equation*}
	For the sake of brevity, we sketch the proof for the estimators after the cutoff. To that end, by the argument similar to the one in the proof of Theorem~\ref{thm:isotonic_clt} for all $(u_1,u_2)\in\R^2$
	\begin{equation*}
	\begin{aligned}
		& \Pr\left(n^{1/3}(\hat m_+(cn^{-1/3}) - m_+(0)) \leq u_1,\;n^{1/3}(\hat p_+(cn^{-1/3}) - p_+(0)) \leq u_2\right) \\
		& = \Pr\left(\argmax_{t\in[0,n^a/c]}Z_{2n}(t) \geq 1,\; \argmax_{t\in[0,n^a/c]}\check Z_{2n}(t)\geq 1 \right),
	\end{aligned}
	\end{equation*}
	where
	\begin{equation*}
	\begin{aligned}
		Z_{n2}(t) & = \left(n^{a}u_1 + n^{(a+1)/2}m(0)\right)(F_n(cn^{-a}t) - F_n(0)) - n^{(a+1)/2}(M_n(cn^{-a}t) - M_n(0)), \\
		\check Z_{n2}(t) & = \left(n^{a}u_1 + n^{(a+1)/2}m(0)\right)(F_n(cn^{-a}t) - F_n(0)) - n^{(a+1)/2}(\check M_n(cn^{-a}t) - \check M_n(0)), \\
		\check M_n(t) & = \frac{1}{n}\sum_{i=1}^nD_i\one\{X_i \leq t\}.
	\end{aligned}
	\end{equation*}
	Then
	\begin{equation*}
		\begin{pmatrix}
			Z_{n2}(t) \\
			\check Z_{n2}(t)
		\end{pmatrix} \leadsto
		\begin{pmatrix}
		Z_{2}(t) \\
		\check Z_{2}(t)
		\end{pmatrix}\qquad \text{in}\qquad \ell^\infty[0,K]\times \ell^\infty[0,K],
	\end{equation*}
	where 
	\begin{equation*}
	\begin{aligned}
		Z_{2}(t) & = u_1tf_+c + \sqrt{\sigma^2_+f_+c}W_t^+ - \frac{t^2c^2}{2}m'_+f_+\one_{a=1/3}, \\
		\check Z_{2}(t) & = u_2tf_+c + \sqrt{p_+(1-p_+)f_+c}B_t^+ - \frac{t^2c^2}{2}p'_+f_+\one_{a=1/3},
	\end{aligned}
	\end{equation*}
	and the covariance between $\sqrt{\sigma^2_+f_+c}W_t^+$ and $\sqrt{p_+(1-p_+)f_+c}B_s^+$ is
	\begin{equation*}
	\begin{aligned}
		& \Cov\left(n^{a/2}(m_+ - Y_i)\one_{[0,cn^{-a}t]}(X_i),n^{a/2}(p_+ - D_i)\one_{[0,cn^{-a}s]}(X_i) \right) \\
		= & n^a\Cov\left(\varepsilon_i\one_{[0,cn^{-a}t]}(X_i),(D_i-p(X_i))\one_{[0,cn^{-a}s]}(X_i) \right) + o(1) \\
		= & n^a\int_0^{cn^{-a}(t\wedge s)}\E\left[\varepsilon(D-p(X))|X=x\right]f(x)\dx x + o(1) \\
		\to & \rho_+cf_+(t\wedge s).
	\end{aligned}
	\end{equation*}
	Therefore, by the joint argmax continuous mapping theorem, see, e.g., \cite{abrevaya2005bootstrap}, Theorem 3,
	\begin{equation*}\footnotesize
	\begin{aligned}
		& \Pr\left(n^{1/3}(\hat m_+(cn^{-1/3}) - m_+(0)) \leq u_1,\;n^{1/3}(\hat p_+(cn^{-1/3}) - p_+(0)) \leq u_2\right) \\
		& \to \Pr\left(\argmax_{t\in[0,\infty)}Z_{2}(t) \geq 1,\; \argmax_{t\in[0,\infty)}\check Z_{2}(t)\geq 1 \right) \\
		& = \Pr\left(D^L_{[0,\infty)}\left(\sqrt{\frac{\sigma^2_+}{cf_+}}W_t^+ + \frac{t^2c}{2}m_+\one_{a=1/3}\right)(1) \leq u_1,\; D^L_{[0,\infty)}\left(\sqrt{\frac{p_+(1-p_+)}{cf_+}}B_t^+ + \frac{t^2c}{2}p_+\one_{a=1/3}\right)(1)\leq u_2\right).
	\end{aligned}
	\end{equation*}	
	Finally, put $\hat g = \hat m_+\left(cn^{-1/3}\right) - \hat m_-\left(-cn^{-1/3}\right)$, $\hat h = \hat p_+\left(cn^{-1/3}\right) - \hat p_-\left(-cn^{-1/3}\right)$, $g = m_+ - m_-$, and $h = p_+ - p_-$, and note that
	\begin{equation*}
	\begin{aligned}
		n^{1/3}(\hat\theta^F - \theta) & = n^{1/3}\left(\frac{\hat g}{\hat h} - \frac{g}{h}\right) \\
		& = \frac{n^{1/3}(\hat g - g)h - n^{1/3}(\hat h - h)g}{h\hat h} \\
		& \cw \frac{1}{h}\xi_1 - \frac{g}{h^2}\xi_2.
	\end{aligned}
	\end{equation*}
\end{proof}

\begin{proof}[Proof of Remark~\ref{remark:right_boundary}]
	We sketch only the most important differences below:
	\begin{equation*}
		\begin{aligned}
			& \Pr\left(n^{(1-a)/2}\left(\hat m(-cn^{-a}) - m(0)\right) \leq u\right) \\
			& = \Pr\left(\hat m(-cn^{-a}) \leq un^{(a-1)/2} + m(0)\right) \\
			& = \Pr\left(\argmax_{s\in[-1,0]}\left\{\left(n^{(a-1)/2}u + m(0)\right)F_n(s) - M_n(s)\right\} \geq -cn^{-a}\right) \\
			& = \Pr\left(\argmax_{t\in[-n^a/c,0]}\left\{\left(n^{(a-1)/2}u + m(0)\right)F_n(cn^{-a}t) - M_n(cn^{-a}t)\right\} \geq -1 \right) \\
			& \to \Pr\left(\argmax_{t\in(-\infty,0]}\left\{ut - \sqrt{\frac{\sigma^2(0)}{cf(0)}}W_t -  \frac{t^2c}{2}m'(0)\one_{a=1/3} \right\} \geq -1 \right) \\
			& = \Pr\left(D^L_{(-\infty,0]}\left(\sqrt{\frac{\sigma^2(0)}{cf(0)}}W_t + \frac{t^2c}{2}m'(0)\one_{a=1/3} \right)(-1) \leq u\right) \\
		\end{aligned}
	\end{equation*}
\end{proof}

\newpage
\begin{center}
	{\LARGE\textbf{SUPPLEMENTARY MATERIAL}}	
\end{center}
\bigskip

\setcounter{page}{1}
\setcounter{section}{0}
\setcounter{equation}{0}
\setcounter{table}{0}
\setcounter{figure}{0}
\renewcommand{\theequation}{SM.\arabic{equation}}
\renewcommand\thetable{SM.\arabic{table}}
\renewcommand\thefigure{SM.\arabic{figure}}
\renewcommand\thesection{SM.\arabic{section}}
\renewcommand\thepage{Supplementary Material - \arabic{page}}
\renewcommand\thetheorem{SM.\arabic{theorem}}

\section{Identification}\label{sec:appendix_identification}
In this section, we revise the known identification results for regression discontinuity designs and show that in sharp designs, the identification can be achieved under weaker one-sided continuity conditions. We also show that under the additional assumption our identifying conditions are both necessary and sufficient. 

The regression discontinuity design postulates that the probability of receiving the treatment changes discontinuously at the cutoff. In the iRDD, we also assume that the expected outcome and the probability of receiving the treatment are both monotone in the running variable. We introduce several assumptions below.

\begin{assumption*}[M1]\label{as:monotonicity}
	The following functions are monotone in some neighborhood of $x_0$ (i) $x\mapsto \E[Y_1|X=x]$ and $x\mapsto \E[Y_0|X=x]$; (ii) $x\mapsto \Pr(D=1|X=x)$.
\end{assumption*}

\begin{assumption*}[M2]\label{as:local_responsiveness}
	$\E[Y_1|X=x_0]\geq \E[Y_0|X=x_0]$ in the non-decreasing case or $\E[Y_1|X=x_0]\leq \E[Y_0|X=x_0]$ in the non-increasing case.
\end{assumption*}

\begin{assumption*}[RD]\label{as:rdd}
	Suppose that
	\begin{equation*}
		\lim_{x\downarrow x_0}\Pr(D=1|X=x) \ne \lim_{x\uparrow x_0}\Pr(D=1|X=x).
	\end{equation*}
\end{assumption*}

\begin{assumption*}[OC]
	Under Assumption (M1), suppose that $x\mapsto\E[Y_1|X=x]$ is right-continuous and $x\mapsto \E[Y_0|X=x]$ is left-continuous at $x_0$.
\end{assumption*}

Assumption (M1) ensures that all limits exist at the discontinuity point. For sharp discontinuity design, all individuals with values of the running variable exceeding the cutoff $x_0$ receive the treatment, while all individuals below the cutoff do not. In other words, $D=\one\{X\geq x_0\}$, and, whence $x\mapsto \Pr(D=1|X=x)$ trivially satisfies (M1), (ii). Assumption (RD) is also trivially satisfied for the sharp design. (M2) requires the local responsiveness to the treatment at the cutoff. It is not necessary for the identification, but as we shall see below, it allows us to characterize in some sense both necessary and sufficient conditions. (OC) is weaker than the full continuity at the cutoff. For more general fuzzy designs, we need additionally the conditional independence assumption.
\begin{assumption*}[CI]
	Suppose that $D\si (Y_1,Y_0)|X=x$ for all $x$ in some neighborhood of $x_0$.
\end{assumption*}

\begin{proposition}\label{prop:identification}
	Under Assumptions (M1) and (OC), in the sharp design
	\begin{equation*}
		\lim_{x\downarrow x_0}\E[Y|X=x] - \lim_{x\uparrow x_0}\E[Y|X=x]
	\end{equation*}
	exists and equals to $\theta$. Moreover, under (M1) and (M2), if $\theta$ equals to the expression in Eq.~(\ref{eq:id}), then (OC) is satisfied.
	
	If, additionally, Assumptions (RD) and (CI) are satisfied, and instead of (OC), we assume that $x\mapsto \E[Y_0|X=x]$ and $x\mapsto \E[Y_1-Y_0|X=x]$ are continuous at $x_0$, then
	\begin{equation*}
		\frac{\lim_{x\downarrow x_0}\E[Y|X=x] - \lim_{x\uparrow x_0}\E[Y|X=x]}{\lim_{x\downarrow x_0}\Pr(D=1|X=x) - \lim_{x\uparrow x_0}\Pr(D=1|X=x)}
	\end{equation*}
	exists and equals to $\theta$.
\end{proposition}
\begin{proof}
	Since treatment of non-decreasing and non-increasing cases is similar, we focus only on the former. Under (M1), \cite{rudin1964principles}, Theorem 4.29, ensures that all limits in Eq.~(\ref{eq:id}) and Eq.~(\ref{eq:id2}) exist. In the sharp design
	\begin{equation*}
		\begin{aligned}
			\theta & = \E[Y_1 - Y_0|X=x_0] \\
			& = \lim_{x\downarrow x_0}\left(\E[Y_1|X=x] - \E[Y_0|X=-x]\right) \\
			& = \lim_{x\downarrow x_0}\E[Y|X=x] - \lim_{x\uparrow x_0}\E[Y|X=x],
		\end{aligned}
	\end{equation*}
	where the second line follows under Assumption (OC), and the third since for any $x>0$
	\begin{equation*}
		\E[Y|X=x] = \E[Y_1|X=x]\qquad \text{and}\qquad  \E[Y|X=-x] = \E[Y_0|X=-x],
	\end{equation*}
	which itself is a consequence of $Y=DY_1 + (1-D)Y_0$ and $D=\one\{X\geq x_0\}$. Now suppose that (M1) and (M2) are satisfied and that $\theta$ coincides with the expression in Eq.~(\ref{eq:id}). Then under (M1)
	\begin{equation*}
		\E[Y_1|X=c]\leq \lim_{x\downarrow x_0}\E[Y_1|X=x]=\lim_{x\downarrow x_0}\E[Y|X=x]
	\end{equation*}
	and 
	\begin{equation*}
		\lim_{x\uparrow x_0}\E[Y|X=x]=\lim_{x\uparrow x_0}\E[Y_0|X=x]\leq \E[Y_0|X=x_0].
	\end{equation*}
	Combining the two inequalities under (M2)
	\begin{equation*}
		\lim_{x\uparrow x_0}\E[Y|X=x]\leq \E[Y_0|X=x_0]  \leq \E[Y_1|X=x_0]\leq \lim_{x\downarrow x_0}\E[Y|X=x].
	\end{equation*}
	Finally, $\theta$ is defined as the difference of inner quantities and also equals to the difference of outer quantities by assumption, which is possible only if (OC) holds, i.e,
	\begin{equation*}
		\lim_{x\uparrow x_0}\E[Y|X=x]= \E[Y_0|X=x_0]  \leq \E[Y_1|X=x_0]= \lim_{x\downarrow x_0}\E[Y|X=x].
	\end{equation*}	
	The proof for the fuzzy design is similar to the proof of \cite{hahn2001identification}, Theorem 2, with the only difference that monotonicity ensures existence of limits, so their Assumption (RD), (i) can be dropped.
\end{proof}
Proposition~\ref{prop:identification} shows that for sharp designs, $\theta$ is identified for a slightly larger class of distributions than are typically discussed in the literature. It shows that the continuity at the cutoff of both conditional mean functions is not needed and that under monotonicity conditions (M1) and (M2), the one-sided continuity turns out to be both necessary and sufficient for the identification. We illustrate this point in Figure~\ref{fig:id}. Panel (a) shows that the causal effect $\theta$ can be identified without full continuity. Panel (b) shows that it may happen that the expression in Eq.~(\ref{eq:id}) coincides with $\theta$, yet the two conditional mean functions do not satisfy (OC). Such counterexamples are ruled out by (M2). Inspection of the proof of the Proposition~\ref{prop:identification} reveals that monotonicity can be easily relaxed if we assume instead that all limits in Eq.~(\ref{eq:id}) and Eq.~(\ref{eq:id2}) exist, in which case we recover the result of \cite{hahn2001identification} under weaker (OC) condition for the sharp design.

It is also worth mentioning that for the fuzzy design, the local monotonicity of the treatment in the running variable allows to identify the causal effect for local compliers; see \cite{hahn2001identification}.

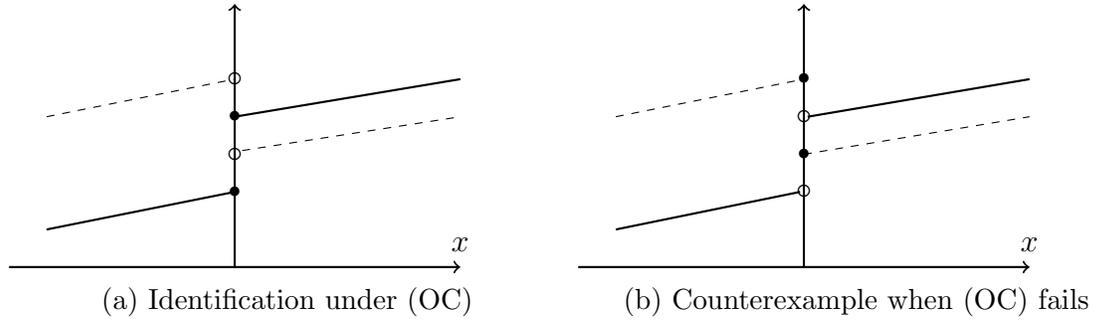
\begin{figure}[ht]
	\centering
	\begin{subfigure}[b]{0.45\textwidth}
		\begin{tikzpicture}
		\draw[thick,->] (-3,0) -- (-3,3.5);
		\draw[thick,->] (-6,0) -- (0,0);
		\draw[thick] (-5.5,0.5) -- (-3,1);
		\draw[dashed] (-2.9,1.55) -- (0,2);	
		\draw[dashed] (-5.5,2) -- (-3,2.5);
		\draw[thick] (-3,2) -- (0,2.5);
		\foreach \Point in {(-3,1), (-3,2)}{
			\node[scale=0.8] at \Point {\textbullet};
		}
		\node at (-3,1.5) {$\circ$};
		\node at (-3,2.5) {$\circ$};
		\node at (0,0.3) {$x$};
		\end{tikzpicture}
		\caption{Identification under (OC)}
	\end{subfigure}	
	\begin{subfigure}[b]{0.45\textwidth}
		\begin{tikzpicture}
		\draw[thick,->] (-3,0) -- (-3,3.5);
		\draw[thick,->] (-6,0) -- (0,0);
		\draw[thick] (-5.5,0.5) -- (-3.05,1);
		\draw[dashed] (-3,1.5) -- (0,2);	
		\draw[dashed] (-5.5,2) -- (-3,2.5);
		\draw[thick] (-2.95,2) -- (0,2.5);
		\foreach \Point in {(-3,1.5), (-3,2.5)}{
			\node[scale=0.8] at \Point {\textbullet};
		}
		\node at (-3,1) {$\circ$};
		\node at (-3,2) {$\circ$};
		\node at (0,0.3) {$x$};
		\end{tikzpicture}
		\caption{Counterexample when (OC) fails}
	\end{subfigure}
	\caption{Identification in the sharp RDD. The thick line represents $\E[Y_0|X=x],x<0$ and $\E[Y_1|X=x],x\geq0$ while the dashed line represents $\E[Y_1|X=x],x<0$ and $\E[Y_0|X=x],x\geq 0$. The thick line coincides with $x\mapsto \E[Y|X=x]$.}
	\label{fig:id}
\end{figure}

\section{Additional technical results}
The proof of Theorem~\ref{thm:isotonic_clt} is based on the argmax continuous mapping theorem, \cite{kim1990cube}, one of the conditions of which is that the argmax is a uniformly tight sequence of random variables. In our setting it is sufficient to show that the argmax of 
\begin{equation*}
	\mathbb{M}_{n1} (s) \triangleq \left(n^{-1/3}u + m(0)\right)[F_n(s+g)-F_n(g)] - [M_n(s+g)-M_n(g)],\qquad s\in[0,1]
\end{equation*}
is $O_P(n^{-1/3})$ for $a\in(0,1/3)$, where $g>0$ is arbitrary small, and that the argmax of 
\begin{equation*}
	\mathbb{M}_{n2}(s) \triangleq (n^{(a-1)/2}u + m(0))F_n(s) - M_n(s),\qquad s\in[0,1]
\end{equation*}
is $O_P(n^{-a})$ for $a\in[1/3,1)$. The following lemma serves this purpose.

\begin{lemma}\label{lemma:tightness}
	Suppose that Assumption~\ref{as:dgp} is satisfied. Then
	\begin{itemize}
		\item[(i)] For $a \in(0,1/3)$ and $u\in \R$ and every $g>0$, $\argmax_{s\in[-g,1-g]}\mathbb{M}_{n1} (s) = O_P(n^{-1/3})$.
		\item[(ii)] For $a=1/3$ and $u\in\R$, $\argmax_{s\in[0,1]}\mathbb{M}_{n2} (s) = O_P(n^{-1/3})$.
		\item[(iii)] For $a\in(1/3,1]$ and $u<0$, $\argmax_{s\in[0,1]}\mathbb{M}_{n2} (s) = O_P(n^{-a})$.
	\end{itemize}
\end{lemma}
\begin{proof}
	\textbf{Case (i): $a\in(0,1/3)$.}
	Put $\mathbb{M}_{1}(s) \triangleq m(g)[F(s+g)-F(g)] - [M(s+g)-M(g)]$ with $M(s)=\int_0^{F(s)}m(F^{-1}(u))\dx u$. For $s_0 = 0$, $\mathbb{M}_1(s_0) = \mathbb{M}_{n1} (s_0) = 0$. By Taylor's theorem, there exists $\xi_{1s}\in(F(s),F(g+s))$ such that
	\begin{align*}
		M(g + s) - M(s) &=\int_{F(g)}^{F(g+s)}m(F^{-1}(z))\dx z\\& = m(g)[F(g+s)-F(s)]+\frac{1}{2}\frac{m'(F^{-1}(\xi_{1s}))}{f(F^{-1}(\xi_{1s}))}(F(g+s)-F(s))^2\\
		&=m(g)[F(g+s)-F(s)]+\frac{1}{2}\frac{m'(F^{-1}(\xi_{1s}))}{f(F^{-1}(\xi_{1s}))}f^2(\xi_{2s})s^2,
	\end{align*}
	where the second line follows by the mean-value theorem for some $\xi_{2s}\in(0,s)$. Then for every $s$ in the neighborhood of $s_0$
	\begin{equation*}
		\begin{aligned}
			\mathbb{M}_1(s) - \mathbb{M}_1(s_0) & = m(g)[F(g+s)-F(s)] - [M(g+s)-M(s)]\\
			&=-\frac{1}{2}\frac{m'(F^{-1}(\xi_{1s}))}{f(F^{-1}(\xi_{1s}))}f^2(\xi_{2s})s^2 \lesssim-s^2
		\end{aligned}
	\end{equation*}
	since under Assumption~\ref{as:dgp} $f$ is bounded away from zero and infinity and $m'$ is finite in  the neighborhood of zero. Next we will bound the modulus of continuity for some $\delta>0$
	\begin{equation}\label{eq:tightness1}
		\begin{aligned}
			\E\sup_{|s|\leq\delta}\left|\mathbb{M}_{n1} (s) - \mathbb{M}_1(s)\right| & \leq \E\Bigg[\sup_{|s|\leq\delta}\left|M_n(s+g) - M_n(g)-M(s+g)+M(g)\right. \\
			&\qquad\quad - \left.[m(0)+n^{-1/3}u](F_n(s+g)-F_n(g)-F(s+g)+F(g))|\right.\Bigg] \\
			&\quad + (n^{-1/3}|u| + |m(0) - m(g)|)|F(g+\delta)-F(g)| \\
			& \leq \E\sup_{|s|\leq\delta}\left|(P_n - P)h_s\right| + \E\sup_{|s|\leq\delta}|R_n(s)| + O((n^{-1/3} + g)\delta),
		\end{aligned}
	\end{equation}
	where $h_s\in\mathcal{H}_\delta = \{h_s(y,x) = (y-m(0))[\one_{[0,s+g]}(x)-\one_{[0,g]}(x)]:\; s\in[0,\delta] \}$ and $R_n(s) = n^{-1/3}u(F_n(s+g)-F_n(g)-F(s+g)+F(g))$. By the maximal inequality, \cite{kim1990cube}, p.199, the first term in the upper bound in Eq.~(\ref{eq:tightness1}) is $\E\sup_{|s|\leq\delta}\left|(P_n - P)h_s\right| \leq n^{-1/2}J(1)\sqrt{PH_\delta^2}$, where $J(1)$ is the uniform entropy integral, which is finite since $\mathcal{H}_\delta$ is a VC-subgraph class of functions with VC index 2, $H_\delta(y,x) = 	|y-m(0)|\one_{[g,g+\delta]}(x)$ is the envelop of $\mathcal{H}_\delta$, and
	\begin{equation*}
		\begin{aligned}
			PH_\delta^2 & = \E[(\sigma^2(X) + |m(X) - m(0)|^2)\one_{[g,g+\delta]}(X)] \\
			& = \int_{g}^{g+\delta}(\sigma^2(x) + |m(x) - m(0)|^2)\dx F(x) = O(\delta).
		\end{aligned}
	\end{equation*}
	Next, by the maximal inequality
	\begin{equation*}
		\begin{aligned}
			\E\sup_{|s|\leq\delta}|R_n(s)| & \leq n^{-1/3}|u|\E\sup_{|s|\leq\delta}|F_n(s+g)-F_n(g)-F(s+g)+F(g)| \\
			& \leq n^{-1/3}n^{-1/2}J(1)\sqrt{PH_\delta^2}|u| = O(n^{-5/6}\delta^{1/2}),
		\end{aligned}
	\end{equation*}
	where $J(1)<\infty$ is the uniform entropy integral and $H_\delta(x) = \one_{[g,g+\delta]}(x)$ is the envelop. Next, setting $\delta=1$, we get $\sup_{s\in[-g,1-g]}\left|\mathbb{M}_{n1} (s) - \mathbb{M}_1(s)\right| = o_P(1).$ Since $m(0)<m(x)$ and $f(x)>0$ for all $x\in(0,1]$, the function $s\mapsto \mathbb{M}_1(s)$ is strictly decreasing with a maximum achieved at $-g$, whence by \cite{van2000weak}, Corollary 3.2.3 (i), $\argmax_{t\in[-g,1-g]}\mathbb{M}_{n1} (t) = o_P(1).$ Then $\phi_n(\delta) = \delta^{1/2} + n^{1/6}\delta$ is a good modulus of continuity function for $a =3/2$ and $r_n = n^{1/3}$. Indeed, for this choice $\delta\mapsto\phi_n(\delta)/\delta^a $ is decreasing and $r_n^2\phi_n(r_n^{-1}) = O(n^{1/2})$. Therefore, the result follows by \cite{van2000weak}, Theorem 3.2.5.
	\paragraph{Case (ii): $a=1/3$.} Put $\mathbb{M}_2(s) \triangleq m(0)F(s) - M(s)$ with $M(s)=\int_0^{F(s)}m(F^{-1}(u))\dx u$. For $s_0 = 0$, $\mathbb{M}_2(s_0) = \mathbb{M}_{n2}(s_0) = 0$. By Taylor's theorem, there exists $\xi_{1s}\in(0,F(s))$ such that
	\begin{align*}
		M(s) &=\int_0^{F(s)}m(F^{-1}(u))\dx u\\& = m(0)F(s)+\frac{1}{2}\frac{m'(F^{-1}(\xi_{1s}))}{f(F^{-1}(\xi_{1s}))}(F(s))^2\\
		&=m(0)F(s)+\frac{1}{2}\frac{m'(F^{-1}(\xi_{1s}))}{f(F^{-1}(\xi_{1s}))}f^2(\xi_{2s})s^2,
	\end{align*}
	where the second line follows by the mean-value theorem for some $\xi_{2s}\in(0,s)$. Then
	\begin{equation*}
		\begin{aligned}
			\mathbb{M}_2(s) - \mathbb{M}_2(s_0) &=m(0)F(s) - M(s)\\
			&=-\frac{1}{2}\frac{m'(F^{-1}(\xi_{1s}))}{f(F^{-1}(\xi_{1s}))}f^2(\xi_{2s})s^2 \lesssim-s^2.
		\end{aligned}
	\end{equation*}
	
	Next we will bound the modulus of continuity for some $\delta>0$
	\begin{equation}\label{eq:tightness}
		\begin{aligned}
			\E\sup_{|s|\leq\delta}\left|\mathbb{M}_{n2}(s) - \mathbb{M}_2(s)\right| & \leq \E\sup_{|s|\leq\delta}\left|M_n(s) - M(s) - (m(0) + n^{(a-1)/2}u)(F_n(s)-F(s))\right| \\
			&\qquad\qquad + n^{(a-1)/2}|u|F(\delta) \\
			& \leq \E\sup_{|s|\leq\delta}\left|(P_n - P)g_s\right| + \E\sup_{|s|\leq\delta}|R_n(s)| + O(n^{(a-1)/2}\delta),
		\end{aligned}
	\end{equation}
	where $g_s\in\mathcal{G}_\delta = \{g_s(y,x) = (y-m(0))\one_{[0,s]}(x):\; s\in[0,\delta] \}$ and $R_n(s) = n^{(a-1)/2}u(F_n(s)-F(s).$ By the maximal inequality \cite{kim1990cube}, p.199, the first term in the upper bound in Eq.~(\ref{eq:tightness}) is $\E\sup_{|s|\leq\delta}\left|(P_n - P)g_s\right| \leq n^{-1/2}J(1)\sqrt{PG_\delta^2},$ where $J(1)$ is the uniform entropy integral, which is finite since $\mathcal{G}_\delta$ is a VC-subgraph class of functions with VC-index 2, $G_\delta(y,x) = 	|y-m(0)|\one_{[0,\delta]}(x)$ is the envelop of $\mathcal{G}_\delta$, and
	\begin{equation*}
		\begin{aligned}
			PG_\delta^2 & = \E[\sigma^2(X)\one_{[0,\delta]}(X)] + \E[|m(X) - m(0)|^2\one_{[0,\delta]}(X)]\\
			& = \int_0^\delta(\sigma^2(x) + |m(X) - m(0)|^2)\dx F(x)  = O(\delta).
		\end{aligned}
	\end{equation*}
	Next, by the maximal inequality
	\begin{equation*}
		\begin{aligned}
			\E\sup_{|s|\leq\delta}|R_n(s)| & \leq n^{(a-1)/2}|u|\E\sup_{|s|\leq\delta}|F_n(s) - F(s)| \leq n^{(a-1)/2}n^{-1/2}J(1)\sqrt{PH_\delta^2}|u| = O(n^{(a-2)/2}\delta^{1/2}),
		\end{aligned}
	\end{equation*}
	where $J(1)<\infty$ is the uniform entropy integral and $H_\delta(x) = \one_{[0,\delta]}(x)$ is the envelop. Next, setting $\delta=1$, we get $\sup_{s\in[0,1]}\left|\mathbb{M}_{n2}(s) - \mathbb{M}_2(s)\right| = o_P(1).$ Since $m(0)<m(x)$ and $f(x)>0$ for all $x\in(0,1]$, the function $s\mapsto \mathbb{M}_2(s)$ is strictly decreasing with maximum achieved at $0$, whence by \cite{van2000weak}, Corollary 3.2.3 (i), $\argmax_{t\in[0,1]}\mathbb{M}_{n2}(t) = o_P(1).$ Then the modulus of continuity is $\phi_n(\delta) = \delta^{1/2} + n^{a/2}\delta$. This is a good modulus of the continuity function for $\alpha=3/2$ and $r_n = n^{1/3}$. For this choice $\delta\mapsto\phi_n(\delta)/\delta^\alpha$ is decreasing and $r_n^2\phi_n(r_n^{-1}) = O(n^{1/2})$. Therefore, the result follows by \cite{van2000weak}, Theorem 3.2.5.
	
	\paragraph{Case (iii): $a\in(1/3,1]$} Here \cite{van2000weak}, Theorem 3.2.5 gives the order $O_P(n^{-1/3})$ only, so we will show directly using the "peeling device" that after the change of variables
	\begin{equation*}
		\argmax_{s\in[0,n^a]}\left\{n^{(1-a)/2}\left(m(0)F_n(sn^{-a}) - M_n(sn^{-a})\right) + uF_n(sn^{-a})\right\} = O_P(1).
	\end{equation*}
	Denote the process inside of the argmax as
	\begin{equation*}
		\begin{aligned}
			Z_{n2}(s) & \triangleq n^{(1-a)/2}\left(m(0)F_n(sn^{-a}) - M_n(sn^{-a})\right) + n^auF_n(sn^{-a}).
		\end{aligned}
	\end{equation*}
	
	Decompose $Z_{n2} = I_{n2} + II_{n2} + III_{n2} + IV_{n2}$ similarly as in the proof of Theorem~\ref{thm:isotonic_clt} (with $c=1$). Next, partition the set $[0,\infty)$ into intervals $S_{j} = \left\{s:\; 2^{j-1}<s\leq 2^j \right\}$ with $j$ ranging over integers. Then if the argmax exceeds $2^K$, it will be located in one of the intervals $S_j$ with $j\geq K$ and $2^{j-1}\leq n^a$. Therefore, using the fact that $u<0,II_{n2}\leq 0$, and $Z_{n2}(0)=0$
	\begin{equation*}
		\begin{aligned}
			\Pr\left(\argmax_{s\in[0,n^a]}Z_{n2}(s) > 2^K \right) & \leq \sum_{\substack{j\geq K \\ 2^{j-1}\leq n^a}} \Pr\left(\sup_{s\in S_j}Z_{n2}(s) \geq 0\right) \\
			& \leq \sum_{\substack{j\geq K \\ 2^{j-1}\leq n^a}} \Pr\left(\sup_{s\in S_j}\left|I_{n2}(s) + III_{n2}(s)\right| \geq -n^auF(2^{j} n^{-a})\right) \\
			& \leq \sum_{\substack{j\geq K \\ 2^{j-1}\leq n^a}}\frac{1}{-un^aF(2^{j} n^{-a})}\E\left[\sup_{s\in S_j}|I_{n2}(s) + III_{n2}(s)|\right] \\
			& \lesssim \sum_{\substack{j\geq K \\ 2^{j-1}\leq n^a}}\frac{1}{-un^aF(2^{j} n^{-a})}\left\{2^{j/2} + n^{(a-1)/2}2^{j/2}  \right\} \lesssim \sum_{\substack{j\geq K}}2^{-j/2}, \\
		\end{aligned}
	\end{equation*}
	where the third line follows by Markov's inequality and the fourth by the maximal inequality; cf. Theorem~\ref{thm:isotonic_clt}. The last expression can be made arbitrarily small by the choice of $K$.
\end{proof}

The following lemma shows tightness of the argmax of the bootstrapped process:
\begin{equation*}
	\mathbb{M}_{n}^*(s)\triangleq n^{(1-a)/2}\left(\hat m(n^{-a})F_n(sn^{-a}) - M_n^*(sn^{-a})\right) + n^auF_n(sn^{-a}),\quad s\in[0,n^a].
\end{equation*}

\begin{lemma}\label{lemma:tightness_bootstrap}
	Suppose that assumptions of Theorem~\ref{thm:bootstrap} are satisfied. Then for every $a\in(1/3,1]$ and $u<0$
	\begin{equation*}
		\argmax_{s\in[0,n^a]}\mathbb{M}_{n}^*(s) = O_P(1).
	\end{equation*}
\end{lemma}
\begin{proof}
	Decompose $\mathbb{M}_{n}^* = I_{n}^* + II_{n}^* + III_{n}^* + IV_{n}^*$ similarly to the proof of Theorem~\ref{thm:bootstrap} (with $c=1$). Next, partition the set $[0,\infty)$ into intervals $S_{j} = \left\{s:\; 2^{j-1}<s\leq 2^j \right\}$ with $j$ ranging over integers. Let $\|.\|_{S_j}$ be the supremum norm over $S_j$. Then if the argmax exceeds $2^K$, it will be located in one of the intervals $S_j$ with $j\geq K$ and $2^{j-1}\leq n^a$. Therefore, using the fact that $u<0,II_{n2}\leq 0$, and $\mathbb{M}_{n}^*(0)=0$
	\begin{equation*}
		\begin{aligned}
			& \mathrm{Pr}^*\left(\argmax_{s\in[0,n^a]}\mathbb{M}_{n}^*(s) > 2^K \right) \\
			& \leq \sum_{\substack{j\geq K \\ 2^{j-1}\leq n^a}} \mathrm{Pr}^*\left(\sup_{s\in S_j}\mathbb{M}_{n}^*(s) \geq 0\right) \\
			& \leq \sum_{\substack{j\geq K \\ 2^{j-1}\leq n^a}} \mathrm{Pr}^*\left(\left\|I_{n}^* + II_n^* + III_n^* + IV_n^* - n^auF(.n^{-a})\right\|_{S_j} \geq -n^auF(2^{j} n^{-a})\right) \\
			& \leq \sum_{\substack{j\geq K \\ 2^{j-1}\leq n^a}}\frac{1}{-un^aF(2^{j} n^{-a})}\E^*\left\|I_{n}^* + II_n^* + III_n^* + IV_n^* - n^auF(.n^{-a})\right\|_{S_j} \\
			& \lesssim \sum_{\substack{j\geq K \\ 2^{j-1}\leq n^a}}\frac{1}{-un^aF(2^{j} n^{-a})}2^{j/2}O_P(1)  \lesssim \sum_{\substack{j\geq K}}2^{-j/2}O_P(1), \\
		\end{aligned}
	\end{equation*}
	where the third line follows by Markov's inequality and the fourth by computations below. The last expression is $o_P(1)$ for every $K=K_n\to\infty$. To see that all terms above are controlled as was claimed, first note that the process $I_n^*$ is a multiplier empirical process indexed by the class of functions $\mathcal{G}_n = \left\{(\epsilon,x)\mapsto -n^{a/2}\epsilon\one_{[0,n^{-a}t]}(x):\quad t\in S_j \right\}.$ This class is of VC subgraph type with VC index 2, whence by the maximal inequality
	\begin{equation*}
		\begin{aligned}
			\E^*\left[\sup_{s\in S_j}|I_n^*(t)|\right] & \lesssim  n^{a/2}\sqrt{\frac{1}{n}\sum_{i=1}^n\varepsilon_i^2\one_{[0, n^{-a}2^j]}(X_i)} =  \sqrt{n^a\E[\varepsilon^2\one_{[0,n^{-a}2^j]}(X)] + o_P(1)} =O_P(2^{j/2}),
		\end{aligned}
	\end{equation*}
	where the second line follows since $\E[\varepsilon^4|X]\leq C$. Next, it follows from the proof of Theorem~\ref{thm:bootstrap} (replacing $K$ by $2^j$) that $\E^*\left[\sup_{s\in S_j}|II_n^*(t)|\right] = o_P(2^{j/2})$
	and that $\|III_n^*\|_{S_j} = o_P(1)$. Lastly, by the maximal inequality $\sup_{s\in S_j}|IV_n^*(s) - n^auF(n^{-a}s)| = O_P(n^{(a-1)/2}2^{j/2}).$
\end{proof}

The following lemma is a conditional argmax continuous mapping theorem for bootstrapped processes.
\begin{lemma}\label{lemma:argmax_conditional}
	Suppose  that for every $K<\infty$
	\begin{enumerate}
		\item[(i)] $\sup_{h\in BL_1(\ell^\infty[0,K])}|\E^*h(Z_n^*) - \E h(Z)|\xrightarrow{P}0$;
		\item[(ii)]	$\limsup_{n\to\infty}\mathrm{Pr}^*\left(\argmax_{t\in[0,n^a]}Z_n^*(t) > K\right) = o_P(1),\qquad K\to\infty$;
		\item[(iii)] $t\mapsto Z(t)$ has unique maximizer on $[0,\infty)$, which is a tight random variable. Then
		\begin{equation*}
			\mathrm{Pr}^*\left(\argmax_{t\in[0,n^a]}Z_n^*(t)\geq z \right) \xrightarrow{P} \mathrm{Pr}\left(\argmax_{t\in[0,\infty)}Z(t)\geq z \right),\qquad \forall z>0.
		\end{equation*}
	\end{enumerate}
\end{lemma}
\begin{proof}
	For every $K$
	\begin{align*}
		\mathrm{Pr}^*\left(\argmax_{t\in[0,n^a]}Z_n^*(t)\geq z \right) & = \mathrm{Pr}^*\left(\argmax_{t\in[0,K]}Z_n^*(t)\geq z \right) + R_{n,K}, 
	\end{align*}
	where by (ii)
	\begin{equation*}
		\begin{aligned}
			\limsup_{n\to\infty}R_{n,K} & = \limsup_{n\to\infty}\mathrm{Pr}^*\left(\argmax_{t\in[0,K]}Z_n^*(t) < z,\argmax_{t\in[K,n^a]}Z_n^*(t)\geq z\right) \\
			& \leq \limsup_{n\to\infty}\mathrm{Pr}^*\left(\argmax_{t\in[0,n^a]}Z_n^*(t) > K\right) = o_P(1),\qquad K\to\infty,
		\end{aligned}
	\end{equation*}
	by (i) and (iii)
	\begin{equation*}
		\begin{aligned}
			\mathrm{Pr}^*\left(\argmax_{t\in[0,K]}Z_n^*(t)\geq z \right) & = \mathrm{Pr}\left(\argmax_{t\in[0,K]}Z(t)\geq z \right) + o_P(1) \\
			& = \mathrm{Pr}\left(\argmax_{t\in[0,\infty)}Z(t)\geq z \right) + o_P(1),\qquad K\to\infty.
		\end{aligned}
	\end{equation*}
	
	More precisely, we used the continuous mapping theorem for the bootstrapped process \cite{kosorok2008introduction}, Proposition 10.7:
	\begin{equation*}
		\begin{aligned}
			\mathrm{Pr}^*\left(\argmax_{t\in[0,K]}Z_n^*(t)\geq z \right) & = \mathrm{Pr}^*\left(\sup_{t\in[z,K]}Z_n^*(t) \geq \sup_{t\in[0,K]}Z_n^*(t) \right) \\
			& \xrightarrow{P} \mathrm{Pr}\left(\sup_{t\in[z,K]}Z(t) \geq \sup_{t\in[0,K]}Z(t) \right) \\
			& = \mathrm{Pr}\left(\argmax_{t\in[0,K]}Z(t)\geq z \right),
		\end{aligned}
	\end{equation*}
	where the convergence is actually uniform over $z$ in arbitrary closed subset of the set of continuity points of $z\mapsto \mathrm{Pr}\left(\argmax_{t\in[0,K]}Z(t)\geq z \right)$; see~\cite{kosorok2008introduction}, Lemma 10.11.
\end{proof}

The following result is is probabilistic statement of the fact that for monotone functions converging pointwise to a continuous limit we also have the uniform convergence.
\begin{proposition}\label{prop:uniform}
	Suppose that assumptions of Theorem~\ref{thm:isotonic_clt} are satisfied. If $m$ is continuous on $[0,1]$, then
	\begin{equation*}
		\sup_{y\in[0,1]}|\tilde m(cn^{-a}y) - m(0)|\xrightarrow{P}0.
	\end{equation*}
\end{proposition}
\begin{proof}
	For every $y\in[0,1]$, by Theorem~\ref{thm:isotonic_clt}, $|\tilde m(cn^{-a}y) - m(0)| \xrightarrow{P} 0.$. Since $m$ is uniformly continuous, one can find $0\leq y_1\leq\dots\leq y_p\leq 1$ such that $|m(cn^{-a}y_{j}) - m(cn^{-a}y_{j-1})|<\epsilon/2$ for all $j=2,\dots,p$. Then on the event $\{|\tilde m(cn^{-a}y_j) - m(cn^{-a}y_j)|\leq \epsilon/2, \;\forall j=1,\dots,p \}$ by the monotonicity of $\tilde m$, for every $x$, there exists $j=2,\dots,p$ such that
	\begin{equation*}
		m(cn^{-a}x) - \epsilon \leq \tilde m(cn^{-a}y_{j-1})\leq \tilde m(cn^{-a}x)\leq \tilde m(cn^{-a}y_j) \leq m(cn^{-a}x) + \epsilon,
	\end{equation*}
	whence
	\begin{equation*}
		\begin{aligned}
			\Pr\left(|\tilde m(cn^{-a}y) - m(cn^{-a}y)| \leq \epsilon,\forall y\in[0,1]\right) \geq 1 - \sum_{j=1}^p\Pr\left(|\tilde m(cn^{-a}y_j) - m(cn^{-a}y_j)| > \epsilon/2\right).
		\end{aligned}
	\end{equation*}
	Since $p$ is fixed, the sum of probabilities tends to zero by the pointwise consistency of $\tilde m$, which gives the result as $\epsilon>0$ is arbitrary.
\end{proof}

\begin{proposition}\label{prop:optimal_constant}
	Suppose that Assumptions of Theorem~\ref{thm:isotonic_clt} are satisfied. Then for arbitrary $c>0$
	\begin{equation*}
		n^{1/3}(\hat m (cAn^{-1/3}) - m(0)) \xrightarrow{d} \left|\frac{m'(0)\sigma^2(0)}{2f(0)}\right|^{1/3}D^L_{[0,\infty)}(W_t - t^2)(c),
	\end{equation*}
	where $A=\left(\frac{2}{m'(0)}\sqrt{\frac{\sigma^2(0)}{f(0)}}\right)^{2/3}$.
\end{proposition}
\begin{proof}
	By Theorem~\ref{thm:isotonic_clt} (ii), for $B>0$
	\begin{equation*}
		\begin{aligned}
			\Pr\left(n^{1/3}(\hat m(cAn^{-1/3}) - m(0)\leq uB\right) \to & \to \Pr\left(D^L_{[0,\infty)}\left(\sqrt{\frac{\sigma^2(0)}{cAf(0)}}W_t + \frac{t^2cA}{2}m'(0)\right)(1) \leq uB\right) \\
			& = \Pr\left(\argmax_{t\in[0,\infty)}\left\{uBt - \sqrt{\frac{\sigma^2(0)}{cAf(0)}}W_t - \frac{t^2cA}{2}m'(0)  \right\} \geq 1\right) \\
			& = \Pr\left(\argmax_{t\in[0,\infty)}\left\{ut - \frac{1}{B}\sqrt{\frac{\sigma^2(0)}{Af(0)}}W_t - \frac{t^2A}{2B}m'(0)  \right\} \geq c\right) \\
			& = \Pr\left(\argmax_{t\in[0,\infty)}\left\{ut - W_t - t^2 \right\} \geq c\right) \\
			& = \Pr\left(D^L_{[0,\infty)}(W_t - t^2)(c) \leq u\right) \\
		\end{aligned}
	\end{equation*}
	where the first equality follows by the switching relation, the second by the change of variables Brownian scaling and invariance of argmax to the scaling, the third by plugging-in corresponding value of $A$ and $B = \left|\frac{m'(0)\sigma^2(0)}{2f(0)}\right|^{1/3}$, and the last by another application of the switching relation.
\end{proof}

\section{Additional Monte Carlo experiments}\label{app:mc}
In this section we report additional results of Monte Carlo experiments for a larger class of data generating processes. First, in Table~\ref{tab:mse2}, we present additional results for the DGPs 1-4 under heteroskedasticity, setting $\sigma(x) = \sqrt{x+1}$.
\begin{table}
	\centering                     	
	\begin{threeparttable}
		\caption{MC results: heteroskedasticity}
		\begin{tabular}{ccccccccccc}                                                  			\hline\hline && \multicolumn{3}{c}{iRDD} &  \multicolumn{3}{c}{LP} &  \multicolumn{3}{c}{k-NN} \\ 		\cline{3-5} \cline{6-11}			 &$n$ & Bias & Var & MSE & Bias & Var & MSE & Bias & Var & MSE \\
			\hline\textbf{DGP 1} &200 & 0.005 & 0.043 & 0.043& 0.009& 0.319& 0.319 &  0.010 & 0.490 & 0.490 \\                                                                                                                                     
			& 500 & -0.007 & 0.022 & 0.022& -0.008 & 0.111 & 0.111& 0.010 & 0.280 & 0.280  \\                                                                                                                                   
			&1000  & -0.007 & 0.013 & 0.013 & -0.001 & 0.056 & 0.056 & 0.000 & 0.220 & 0.220\\ 
			\hline\textbf{DGP 2} & 200 & -0.101 & 0.083 & 0.093& 0.006& 0.791& 0.791 & 0.090 & 0.200 & 0.210 \\                                                                                                                                    
			& 500 & -0.075 & 0.043 & 0.049& 0.000 & 0.224 & 0.224 &  0.040 & 0.090 & 0.090 \\                                                                                                                                    
			&1000  & -0.061 & 0.027 & 0.031& -0.000 & 0.103 & 0.103 & 0.030 & 0.070 & 0.070 		 \\ 
			\hline\textbf{DGP 3} & 200 & -0.127 & 0.041 & 0.057&-0.004& 0.323& 0.323 & 0.020 & 0.510 & 0.510 \\                                                                                                                                    
			& 500 & -0.111 & 0.019 & 0.031& -0.005 & 0.109 & 0.109& 0.020 & 0.280 & 0.280 \\                                                                                                                                    
			&1000  & -0.096 & 0.011 & 0.020& 0.000 & 0.057 & 0.057& 0.000 & 0.230 & 0.230 \\   
			\hline\textbf{DGP 4} & 200 & -0.286 & 0.073 & 0.155&-0.019& 0.753& 0.754 & 0.020 & 0.480 & 0.490 \\                                                                                                                                    
			& 500 & -0.223 & 0.035 & 0.084& -0.019 & 0.213 & 0.213	& 0.000 & 0.280 & 0.280 \\                                                                                                                                    
			& 1000 & -0.181 & 0.019 & 0.052& -0.009 & 0.104 & 0.104& 0.010 & 0.210 & 0.210  \\
			
			\hline\hline 
		\end{tabular}  
		\begin{tablenotes}
			\small
			\item Exact finite-sample bias, variance, and MSE of iRDD, local polynomial (LP), and k-NN estimators. 5000 experiments.  Local linear estimator with   kernel=`triangular' and bwselect=`mserd'.
		\end{tablenotes}
		\label{tab:mse2}
	\end{threeparttable}
\end{table}
The results are largely similar to the homoskedastic designs in Table~\ref{tab:mse}. Next, we consider whether the data-driven choice of $c$ can improve the rule-of-thumb $c=1$. 
\begin{table}[h]
	\centering                     	
	\begin{threeparttable}
		\caption{MC experiments for the iRDD estimator with the MSE-optimal boundary correction.}
		\begin{tabular}{cccccccc}                                                                                                                                                     
			\hline\hline && \multicolumn{3}{c}{Homoskedasticity} &  \multicolumn{3}{c}{Heteroskedasticity} \\ 		\cline{3-5} \cline{6-8}			 &$n$ & Bias & Var & MSE & Bias & Var & MSE  \\ 
			\hline\textbf{DGP 1} & 200 & 0.163 & 0.070 & 0.096 & 0.172 & 0.070 & 0.100 \\                                                                                                          
			& 500 & 0.060 & 0.018 & 0.021 & 0.063 & 0.018 & 0.022 \\                                                                                                          
			& 1000 & 0.043 & 0.011 & 0.013 & 0.045 & 0.011 & 0.013 \\                                                                                                         
			\hline\textbf{DGP 2} & 200 & 0.024 & 0.097 & 0.097 & 0.045 & 0.100 & 0.102 \\                                                                                                          
			& 500 & 0.016 & 0.031 & 0.031 & 0.022 & 0.032 & 0.032 \\                                                                                                          
			& 1000 & 0.013 & 0.020 & 0.020 & 0.016 & 0.020 & 0.020 \\                                                                                                         
			\hline\textbf{DGP 3} & 200 & 0.280 & 0.085 & 0.163 & 0.286 & 0.085 & 0.167 \\                                                                                                          
			& 500 & 0.104 & 0.021 & 0.032 & 0.106 & 0.021 & 0.032 \\                                                                                                          
			& 1000 & 0.070 & 0.012 & 0.017 & 0.071 & 0.012 & 0.017 \\                                                                                                         
			\hline\textbf{DGP 4} & 200 & 0.241 & 0.149 & 0.207 & 0.249 & 0.149 & 0.211 \\                                                                                                          
			& 500 & 0.082 & 0.041 & 0.048 & 0.085 & 0.042 & 0.049 \\                                                                                                          
			& 1000 & 0.050 & 0.025 & 0.027 & 0.052 & 0.025 & 0.028 \\                                                                                                         
			\hline                                                                                                                                                                        
		\end{tabular}  
		\begin{tablenotes}
			\small
			\item Finite-sample bias, variance, and MSE of $\hat\theta$. 5,000 experiments.
		\end{tablenotes}
		\label{tab:mse_optimal}
	\end{threeparttable}
\end{table}
According to Proposition~\ref{prop:optimal_constant}, the MSE-optimal estimator is $\hat m^*(0) = \hat m (c^*\hat An^{-1/3}),$ where $\hat A = \left(\frac{2}{\hat m'(0)}\sqrt{\frac{\hat \sigma^2(0)}{\hat f(0)}}\right)^{2/3}$, $\hat m',\hat \sigma^2,\hat f$ are consistent estimators of $m',\sigma^2,f$, and $c^*$ minimizes $\E|D^L_{[0,\infty)}(W_t - t^2)(c)|^2$. \cite{kulikov2006behavior} find in simulations that $c^*\approx 0.345$. To have an idea of how much the MSE-optimal estimator performs in small samples, we consider the oracle choice of $m',\sigma^2$, and $f$. Results of these Monte Carlo experiments are presented in Table~\ref{tab:mse_optimal}. We find that the MSE-optimal estimator is inferior in small samples compared to the rule-of-thumb choice $c=1$, cf. Tables~\ref{tab:mse} and \ref{tab:mse2}. The increase in the MSE might come from the fact that asymptotic MSE might be a poor approximation to the finite-sample MSE; cf. \cite{kulikov2006behavior}.

Lastly, we augment the baseline DGPs 1-4 with four additional DGPs featuring 1) non-monotone functions; 2) functions that are steep at the cutoff; 3) functions that change the slope at the cutoff. The simulations designs are as follows:
\begin{itemize}
	\item DGP 5 sets $X\sim 2\times\mathrm{Beta}(2,2)-1$ and $m(x)= (1-e^{4x})\one_{x\geq 0} + 0.2x\one_{x<0}$;
	\item DGP 6 sets $X\sim 2\times\mathrm{Beta}(0.5,0.5)-1$  and $m(x)$ is the same as in the DGP 5;
	\item DGP 7 sets $X\sim 2\times\mathrm{Beta}(2,2)-1$ and $m(x)=(-e^{-1.25}+e^{-(x-0.5)^2})\one_{x\geq 0.5} - (e^{-1.25}-e^{-5(x-.5)^2})\one_{0\leq x<0.5} + 0.2x\one_{x<0}$;
	\item DGP 8 sets $X\sim 2\times\mathrm{Beta}(0.5,0.5)-1$  and $m(x)$ is the same as in DGP7.
\end{itemize}
Note that in DGPs 5-6, the regression function is increasing before the cutoff and steeply decreasing after the cutoff. The DGPs 7-8 feature the regression function that is steeply increasing after the cutoff for all $x\in[0,0.5)$ and decreasing subsequently on $[0.5,1]$. These DGPs are the most difficult because they violate the monotonicity constraint and at the same time feature steep regression functions.\footnote{As explained in Section~\ref{sec:irdd}, since the isotonic regression estimator is the constrained least-squares estimator, we may expect the projection interpretation under the misspecification.} We compare the performance of our iRDD estimator to the local polynomial estimator in Table~\ref{tab:mse_extended_homo}. The results show 2-10 fold reduction in the MSE across specifications with the most improvement achieved when the sample size is small.

\begin{table}
	\centering{\footnotesize                     	
		\begin{threeparttable}
			\caption{MC results: steep and non-monotone regressions}
			\begin{tabular}{cccccccc}                                                                                                                                                     
				\hline\hline && \multicolumn{3}{c}{iRDD} &  \multicolumn{3}{c}{LP} \\ \hline
				&$n$ & Bias & Var & MSE & Bias & Var & MSE  \\\hline
				\multicolumn{8}{c}{Homoskedasticity}     \\     \hline                                
				\textbf{DGP 5} & 200 & -0.1340 & 0.0440 & 0.0620&0.0193& 0.3366& 0.3369
				\\                                            & 500 & -0.1190 & 0.0200 & 0.0340& 0.0446 & 0.1097 & 0.1117 \\                                                                                                                                    
				&1000  & -0.1020 & 0.0110 & 0.0220 & 0.0427 & 0.0559 & 0.0577 \\ \hline                                                                                                                                  
				\textbf{DGP 6} & 200 & -0.3080 & 0.0790 & 0.1740&0.0272& 0.8498& 0.8503
				\\                                           & 500 & -0.2380 & 0.0380 & 0.0950& 0.0559 & 0.2156 & 0.2186  \\                                                                                                                                    
				& 1000 & -0.1910 & 0.0220 & 0.0590&  0.0678 & 0.1041 & 0.1086 \\ \hline                                                                                                                         
				\textbf{DGP 7} & 200 & -0.1319 & 0.0409 & 0.0583&-0.0170& 0.3194& 0.3197
				\\                                            & 500 & -0.1131 & 0.0185 & 0.0313& -0.0114 & 0.1116 & 0.1117 \\                                                                                                                                    
				&1000  & -0.0986 & 0.0104 & 0.0202& -0.0026 & 0.0535 & 0.0535 \\ \hline                                                                                                                                  
				\textbf{DGP 8} & 200 & -0.2891 & 0.0701 & 0.1537&-0.0061& 0.7805& 0.7804 \\                                                                                                                                    
				& 500 & -0.2243 & 0.0344 & 0.0847 & -0.0247 & 0.2284 & 0.2290 \\                                                                                                                                    
				& 1000 & -0.1816 & 0.0197 & 0.0527& -0.0170 & 0.1051 & 0.1054 \\ \hline
				\multicolumn{8}{c}{Heteroskedasticity}     \\  \hline
				\textbf{DGP 5} &200 & -0.1300 & 0.0450 & 0.0620& 0.0202& 0.3369& 0.3373\\  
				& 500 & -0.1170 & 0.0200 & 0.0330& 0.0431 & 0.1104 & 0.1123  
				\\                                          & 1000 & -0.1000 & 0.0110 & 0.0210	& 0.0422 & 0.0567 & 0.0585 \\ \hline
				\textbf{DGP 6} &200 & -0.3020 & 0.0800 & 0.1710& 0.0273& 0.8629& 0.8635
				\\                                    & 500 & -0.2340 & 0.0390 & 0.0930& 0.0624 & 0.2149 & 0.2187  
				\\                                  &1000  & -0.1890 & 0.0220 & 0.0580& 0.0682 & 0.1025 & 0.1072 \\ \hline
				\textbf{DGP 7} &200 & -0.1280 & 0.0410 & 0.0580& -0.0176& 0.3194& 0.3196 \\                                 & 500 & -0.1110 & 0.0190 & 0.0310& -0.0111 & 0.1092 & 0.1093  
				\\                                & 1000 & -0.0970 & 0.0110 & 0.0200& -0.0021 & 0.0534 & 0.0534 \\ \hline
				\textbf{DGP 8} & 200 & -0.2810 & 0.0720 & 0.1510& -0.0038& 0.7862& 0.7861
				\\                                 & 500 & -0.2190 & 0.0350 & 0.0830& -0.0233 & 0.2269 & 0.2274  
				\\                               &1000  & -0.1790 & 0.0200 & 0.0520& -0.0168 & 0.1065 & 0.1068 \\                                                                      
			\end{tabular}  
			\begin{tablenotes}
				\small
				\item Note: exact finite-sample bias, variance, and MSE of $\hat\theta$. 5000 experiments. Local linear estimator with  kernel=`triangular' and bwselect=`mserd'.
			\end{tablenotes}
			\label{tab:mse_extended_homo}
	\end{threeparttable}}
\end{table} 

Finally, we report the empirical coverage and the length of wild bootstrap confidence intervals in Table~\ref{tab:cov_len}. Since the inference-optimal choice of $c$ is less obvious, we report results for a range of value of $c$. We find that for $c=1$, the coverage is approximately $90\%$ for a nominal coverage $95\%$ and that the coverage gets closer to the nominal $95\%$ coverage for larger values of $c$ as a sample size increases. Designing inference-optimal data-driven methods to select $c$ is an important point left for the future research; see e.g., \cite{calonico2020optimal} and \cite{lazarus2018har}.

\begin{table}[http]              
	\centering
	\begin{tabular}{cccc|cc|cc}                                                             \hline\hline	
		&&  \multicolumn{2}{c}{$n=200$}&  \multicolumn{2}{c}{$n=500$}&  \multicolumn{2}{c}{$n=1000$} \\	\cline{3-5} \cline{6-8}			 &$c$ & Coverage & Length & Coverage & Length & Coverage & Length  \\ 
		\hline           
		\multirow{5}*{\textbf{DGP 1}}		&1&0.892&    1.108&0.910&    0.897	&0.909 &   0.763 \\
		&2&0.920&    0.915&0.925&    0.722&0.926 &   0.603\\ 
		&3&0.928&    0.851& 0.938&    0.647&0.929 &   0.536 \\  
		&4&0.937&   0.822&0.946 &   0.608&0.946 &   0.497\\ 
		&5&0.938&    0.814& 0.930&    0.585&0.933 &   0.472 \\                                                                                                          
		\hline
		\multirow{5}*{\textbf{DGP 2}}		&1&0.852&    1.460&0.889&    1.242&0.882&    1.078\\  
		&2&0.888&    1.186&0.904 &   0.980 & 0.907&    0.843\\  
		&3&0.897&    1.049&0.903&    0.854&0.912&    0.736 \\  
		&4&0.891&    0.963&0.903&    0.782& 0.921&    0.669\\ 
		&5&0.886&    0.914&0.899&    0.731& 0.912&    0.625 \\                                                                                                         
		\hline                        
		\multirow{5}*{\textbf{DGP 3}}			&1&0.892&    1.139&0.903&    0.929	&0.913&    0.786\\  
		&2&0.911&    0.965&0.917 &   0.760&0.924 &   0.633\\ 
		&3&0.915&    0.897&0.925&    0.693&0.932&    0.570  \\  
		&4&0.907&    0.875&0.937 &   0.657& 0.940&    0.536\\  
		&5&0.869 &   0.872&0.922 &   0.642& 0.937&    0.515 \\                                                                                                          
		\hline
		\multirow{5}*{\textbf{DGP 4}}		&1&0.853&    1.550&0.886  &  1.310&0.895&    1.134\\  
		&2&0.885 &   1.303&0.915 &   1.071 &0.913&    0.906\\  
		&3&0.899&    1.186&  0.913 &   0.962 &0.926&    0.808  \\  
		&4&0.878&    1.121&0.920 &   0.898 &0.939&    0.754\\  
		&5&0.846&    1.080&0.917 &   0.863&0.935 &   0.718 \\                                                                                                         
		\hline
		\hline                                                            
	\end{tabular}
	\caption{Coverage and length. 5000 experiments.}     
	\label{tab:cov_len}
\end{table} 

\section{Examples of monotone discontinuity designs}\label{sec:appendix_literature}
In Table~\ref{tab:iRDD_examples}, we collect a list of empirical papers with monotone regression discontinuity designs. We focus only on papers where the global monotonicity is economically plausible and is empirically supported. It is worth stressing that monotonicity restricts only how the average outcome changes with the running variable and that in some references monotonicity appears due to the restricted set of values of the running variable, e.g., elderly people. However, we do not include papers where we might have global piecewise monotonicity with known change points, so the scope of the empirical applicability is probably larger.

\begin{sidewaystable}[h]
	\centering{\footnotesize
		\caption{Monotone regression discontinuity designs}
		\begin{tabular}{cccccccc}
			\hline\hline
			Study & Outcome(s) & Treatment(s) & Running variable \\    
			\hline  
			\cite{lee2008randomized} & Votes share in next election & Incumbency & Initial votes share \\
			\cite{Duflo20111739} & Endline scores & Higher-achieving peers  & Intitial attainment \\
			\cite{Abdulkadiroglu2014137} & Standardized test scores & Attending elite school & Admission scores	\\
			\cite{Lucas2014234} & Probability of graduation & Attending elite secondary school & Admission scores \\
			\cite{Hoekstra2009717} & Earnings & Attending flagship state university & SAT score \\
			\cite{clark2010selective} & Test scores, university enrollment & Attending selective high school & Assignment test\\
			\cite{Kaniel2017337} & Net capital flow & Appearance in the WSJ ranking & Returns \\
			\cite{Schmieder2012701} & Unemployment duration & Unemployment benefits & Age \\
			\cite{Card20082242} & Health care utilization & Coverage under Medicare & Age \\
			\cite{Shigeoka20142152} & Outpatient visits & Cost-sharing policy & Age \\
			\cite{Carpenter2009164} & Alcohol-related mortality & Ability to drink legally & Age \\
			\cite{jacob2004remedial} & Academic achievements & Summer school, grade retention & Test scores \\
			\cite{BaumSnow2009654} & Income, property value	& Tax credit program & Fraction of eligible \\
			\cite{Buettner2006477} & Business tax rate & Fiscal equalization transfers & Tax base \\
			\cite{Card20071511} & Job finding hazard & Severance pay & Job tenure\\
			\cite{Chiang20091045} & Medium run test scores & Sanctions threat & School performance \\
			\cite{Ferreira2010661} & Probability to move to a new house & Ability to transfer tax benefits & Age \\
			\cite{Lalive2007108} & Unemployment duration & Unemployment benefits & Age\\
			\cite{Litschig2013206} & Education, literacy, poverty & Government transfers & Size of municipality \\	
			\cite{Ludwig2007159} & Mortality, educational attainment & Head Start funding & County poverty rate \\
			\cite{Matsudaira2008829} & Test scores  & Summer school & Test scores \\
			\cite{chay2005does} & Housing prices & Regulatory status & Pollution levels \\
			\cite{greenstone2008does} & Housing prices & Superfund clean-up status & Ranking of hazard \\
			\hline
		\end{tabular}
		\label{tab:iRDD_examples}}
\end{sidewaystable}

\end{document}